\documentclass[11pt]{article}

\usepackage{amsmath,amssymb,amsthm,mathtools}
\usepackage{accents}
\usepackage{tikz}
\usepackage{tikz-cd}
\usetikzlibrary{matrix,shapes,calc,fit,backgrounds}
\usepackage{fullpage}
\usepackage{booktabs}
\usepackage{array}
\usepackage{enumitem}
\usepackage{placeins}
\usepackage{xcolor}
\usepackage[colorlinks=true,linkcolor=blue,citecolor=teal]{hyperref}
\newtheorem{theorem}{Theorem}[section]
\newtheorem{lemma}[theorem]{Lemma}
\newtheorem{proposition}[theorem]{Proposition}
\newtheorem{corollary}[theorem]{Corollary}
\theoremstyle{definition}
\newtheorem{definition}[theorem]{Definition}
\newtheorem{example}[theorem]{Example}
\newtheorem{remark}[theorem]{Remark}
\newtheorem*{namedtheorem}{Theorem}
\newtheorem*{namedproposition}{Proposition}

\newcommand{\C}{\mathbb{C}}
\newcommand{\R}{\mathbb{R}}
\newcommand{\Z}{\mathbb{Z}}
\DeclareMathOperator{\Hess}{Hess}
\DeclareMathOperator{\Lag}{Lag}
\DeclareMathOperator{\Or}{O}
\DeclareMathOperator{\Sym}{Sym}
\newcommand{\Mas}{\mu}         % Maslov index
\newcommand{\bV}{V}            % reference symplectic vector space
\newcommand{\bF}{F}            % reference Lagrangian subspace of V (Hess's polarization)

\newcommand{\HH}{\mathbb{H}}
\newcommand{\g}{\mathfrak{g}}
\newcommand{\bxi}{\boldsymbol{\xi}}   % Lie algebra element of g^c
\newcommand{\beps}{\boldsymbol{\varepsilon}}  % nontrivial element of ker sigma
\newcommand{\bu}{u}                   % lifting morphism P^c -> P
\newcommand{\dR}{\mathrm{dR}}
\renewcommand{\u}{\mathfrak{u}}      % redefine breve accent; used only as \u(1)
\newcommand{\ssigma}{\sigma}         % symplectic-specific cover, typeset as sigma
\newcommand{\dd}{\mathrm{d}}
\DeclareMathOperator{\Ad}{Ad}
\DeclareMathOperator{\curv}{curv}
\DeclareMathOperator{\Hom}{Hom}
\DeclareMathOperator{\Hol}{Hol}
\DeclareMathOperator{\Mp}{Mp}
\DeclareMathOperator{\Sp}{Sp}
\DeclareMathOperator{\SO}{SO}
\DeclareMathOperator{\SU}{SU}
\newcommand{\hofib}[2]{\mathrm{ho}\,{#1}^{-1}(#2)}
\newcommand{\hv}[1]{\accentset{\rightharpoonup}{#1}}  % morphism of a groupoid; \hv u is the descended identification
\DeclareMathOperator{\GL}{GL}
\DeclareMathOperator{\U}{U}
\DeclareMathOperator{\PU}{PU}
\DeclareMathOperator{\PO}{PO}
\DeclareMathOperator{\Spin}{Spin}
\DeclareMathOperator{\Aut}{Aut}
\DeclareMathOperator{\Heis}{Heis}
\DeclareMathOperator{\ori}{or}
\DeclareMathOperator{\inc}{inc}
\DeclareMathOperator{\pr}{pr}
\DeclareMathOperator{\ind}{ind}
\DeclareMathOperator{\id}{id}
\DeclareMathOperator{\sq}{sq}        % Squaring map
\DeclareMathOperator{\ML}{ML}        % Metalinear group

\newcommand{\Uloc}{\underline{\U(1)}}          % sheaf of LOCALLY CONSTANT U(1)-functions
\newcommand{\Usm}{\mathcal{U}(1)}              % sheaf of SMOOTH U(1)-functions

\newcommand*\circled[1]{\tikz[baseline=(char.base)]{%
  \node[shape=circle,draw,inner sep=1pt] (char) {#1};}}

\makeatletter
\renewcommand\paragraph{\@startsection{paragraph}{4}{\z@}%
  {2.6ex \@plus1ex \@minus.2ex}{-0.6em}%
  {\normalfont\normalsize\itshape}}
\makeatother

\newcommand{\Edge}{\operatorname{Edge}}
\newcommand{\Face}{\operatorname{Face}}
\newcommand{\PsCocone}{\mathbf{PsCocone}}

\begin{document}
%========================= FRONT MATTER =========================
\title{Metaplectic-c prequantizations and the Einstein--Brillouin--Keller quantizability condition}

\author{Hyunmoon Kim\thanks{Ewha Womans University, Seoul, Republic of Korea}}

\date{\today}

\maketitle

\begin{abstract}
We show that every symplectic manifold admits a symplectic form whose metaplectic-c prequantizations give quantizability conditions on all its Lagrangian submanifolds. Metaplectic-c prequantizations descend to prequantizations over the Lagrangian Grassmannian bundle that are fibrewise flat with holonomies primitive fourth roots of unity on the generators of the fundamental groups. This fibrewise holonomy is consistently interpreted as that of both Leray $2$-orientations and normalized Blattner--Kostant--Sternberg half-forms. The prequantizations pull back to flat circle bundles on each Lagrangian submanifold whose holonomy character defines a quantizability condition. The first and second Stiefel--Whitney classes are identified as splitting obstructions to the Lagrangianwise holonomies. These holonomies split into a product of holonomies on any orientable Lagrangian submanifold, and split consistently for all Lagrangian submanifolds of a manifold with metaplectic structure. When the first Chern class vanishes, the holonomy splits into the Bohr--Sommerfeld factor and a quarter phase Maslov correction, recovering the classical Einstein--Brillouin--Keller condition. We develop the underlying obstruction theory for circle extensions associated to general Lie group double coverings with first hypercohomology groupoids.
\end{abstract}

\medskip
\noindent\textbf{Keywords.} Geometric quantization; metaplectic-c structures; Lagrangian submanifolds; Maslov index; Einstein--Brillouin--Keller condition; hypercohomology groupoids.

\smallskip
\noindent\textbf{Mathematics Subject Classification (2020).} Primary 53D50; Secondary 53D12, 81S10, 81Q20, 55N30, 53C27.

\tableofcontents

\section{Introduction}
\label{sec:intro}

\subsection{Prequantizations restrict to quantizability}
\label{subsec:restrict}

The prequantizability problem for symplectic manifolds and the quantizability problem for their
Lagrangian submanifolds are distinct but related problems \cite{KarasevMaslov}, one originating
with the Dirac axioms and the other developed from the Schr\"odinger equation
(cf.\ \cite{BatesWeinstein}).

When viewed as topological problems, the relevant obstructions lie in different cohomological degrees.
On a symplectic manifold $M$, prequantizability asks for the existence of a bundle, or the
triviality of a gerbe. For liftings by circle bundles with connection the obstruction lies in
$H^2(M;\Uloc)$, and with no condition on connections it lies
in $H^3(M;\Z)$. On a Lagrangian submanifold $\iota_L:L\hookrightarrow M$, quantizability asks for
the existence of a section, or the triviality of a bundle. For parallel sections of a circle
bundle with flat connection the obstruction lies in $H^1(L;\Uloc)$, and for smooth sections it lies
in $H^2(L;\Z)$. We summarize these obstruction theories in Table~\ref{tab:obstructions}.

For the Kostant--Souriau and the Bohr--Sommerfeld theories, the obstruction problems are related by honest restriction: the \emph{prequantizations} of a prequantizable symplectic
manifold restrict to a \emph{quantizability condition} on its Lagrangian submanifolds.

The obstruction to Kostant--Souriau prequantizability of a symplectic manifold
$(M,\omega)$ is the class
\begin{equation}\label{eq:kwintro}
  S\;\longmapsto\;\exp\Big(2\pi i\textstyle\int_S\tfrac{\omega}{h}\Big)
  \;\in\;\Hom_\Z\big(H_2(M;\Z),\U(1)\big)\;\cong\;H^2(M;\Uloc),
\end{equation}
whose triviality is the \textbf{Kostant--Weil integrality condition} (cf.\ Dirac charge
quantization~\mbox{\cite{Dirac,Freed}}).
When this class is
trivial, the symplectic manifold has a Kostant--Souriau prequantization: a circle bundle with
connection $(Y_{\mathrm{KS}},\gamma_{\mathrm{KS}})$ whose curvature form is the given
symplectic form $\omega$.

On a Lagrangian submanifold $\iota_L:L\hookrightarrow M$, the Kostant--Souriau prequantization
restricts to a circle bundle with flat connection on $L$, and a trivializable circle bundle with flat connection on each loop in $L$. The holonomy map
\begin{equation}\label{eq:bsintro}
  \lambda\;\longmapsto\;
  \exp\Big(2\pi i\oint_\lambda\tfrac{i\,s_\lambda^*\gamma_{\mathrm{KS}}}{h}\Big)
  \;\in\;\Hom_\Z\big(H_1(L;\Z),\U(1)\big)\;\cong\;H^1(L;\Uloc),
\end{equation}
is independent of the choice of trivializing section $s_\lambda$. The triviality of this holonomy map is the \textbf{Bohr--Sommerfeld integrality condition}  and gives the obstruction to the quantizability of the Lagrangian submanifold. When this obstruction is trivial, the Lagrangian submanifold supports a parallel section of
$(Y_{\mathrm{KS}}|_L,\gamma_{\mathrm{KS}}|_L)$. On a cotangent bundle $M=T^\ast X$ with
exact symplectic form $\omega=\dd(p\,\dd q)$, and Kostant--Souriau prequantization
$Y_{\mathrm{KS}}=T^\ast X\times\U(1)$ with $\gamma_{\mathrm{KS}}=i\,p\,\dd q+\hbar\,z^{-1}\dd z$, the
classical Bohr--Sommerfeld condition $\tfrac1h\oint_\lambda p\,\dd q\in\Z$ is the triviality of
\eqref{eq:bsintro}, because
$\tfrac{1}{i}\oint_\lambda s_\lambda^* \gamma_{\mathrm{KS}}=\oint_\lambda p\,\dd q$ for the
constant section $s_\lambda=(\lambda,1)$.

In this paper we show that the \emph{metaplectic-c prequantizations} of a symplectic manifold
$(M,\omega)$ give the \emph{Einstein--Brillouin--Keller (EBK) integrality condition} on
its Lagrangian submanifolds (Theorem~\ref{thm:ebk}).

A \textbf{metaplectic-c prequantization} \cite[Def.~1.4, Def.~1.12]{Hess},
\cite[\S6]{RobinsonRawnsley} is a tower of principal bundles
\begin{equation}\label{eq:mpctowerintro}
\begin{tikzcd}[row sep=2.1em, column sep=small]
(P^c,\gamma) \ar[d, "\bu"'] & \text{principal } \Mp^c(\bV)\text{-bundle over } M \\
P \ar[d, "\rho"'] & \text{symplectic frame bundle, principal } \Sp(\bV)\text{-bundle over } M \\
M & (M,\omega) \text{ symplectic manifold}
\end{tikzcd}
\end{equation}
where $\bV$ is a fixed model symplectic vector space, and $\gamma$ is a $\U(1)$-connection
on the circle bundle $\bu:P^c\to P$ which is $\Mp^c(\bV)$-invariant and satisfies
$\dd\gamma=i\,\rho^{c\,\ast}\omega$ for the projection $\rho^c=\rho\circ\bu$ to $M$ (Definition \ref{def:centralconnection}).

On a symplectic manifold, the obstruction to metaplectic-c prequantizability is the
Kostant--Weil class \eqref{eq:kwintro} translated by the second Stiefel--Whitney class of $TM$
\cite[Thm.~1.14]{Hess}, \cite[(7.3)]{RobinsonRawnsley} (Corollary~\ref{cor:Pc}),
\begin{equation}\label{eq:mpcintro}
  S\;\longmapsto\;
  \exp\Big(2\pi i\textstyle\int_S\tfrac{\omega}{h}\Big)\cdot(-1)^{\langle w_2(TM),S\rangle}
  \;\in\;\Hom_\Z\big(H_2(M;\Z),\U(1)\big)\;\cong\;H^2(M;\Uloc).
\end{equation}
In real coefficients, the triviality of this class is equivalent to the integrality of
\begin{equation}\label{eq:hessformTM}
  \Big[\tfrac{\omega}{h}\Big]+\tfrac{1}{2}c_1(TM)_{\R}\ \in\ H^2(M;\Z)_\R ,
\end{equation}
as stated in \cite[$q^3$]{Czyz}, \cite[Prop.~1.15]{Hess},
\cite[(7.3)]{RobinsonRawnsley}. The same condition was obtained in
\cite[Thm.~1, (32), Rem.~1]{KarasevMaslov} and \cite[pp.~75--76]{DazordPatissier} as an
asymptotic quantizability condition.
Metaplectic-c prequantizability is strictly weaker than the joint condition of
Kostant--Souriau prequantizability and the existence of a metaplectic structure. The latter is obstructed by
$w_2(TM)$ \cite[Prop.~3.3.2]{KostantSpinors}, \cite[Rem.~2.1]{Czyz}, so cotangent bundles of oriented manifolds are
metaplectic, since $w_2\big(T(T^\ast X)\big)=w_1(TX)^2$ for any manifold $X$, as observed by
D.~Burns \cite[Prop.~3.5]{KostantSpinors}. On the other hand, K\"ahler
manifolds with odd first Chern class --- such as the complex Grassmannians
$\mathrm{Gr}(k,2n+1)$ with $0<k<2n+1$, or Enriques surfaces (Example~\ref{ex:enriques})
--- admit K\"ahler forms for which they are metaplectic-c prequantizable, but admit no
metaplectic structure.

Moreover, while it is well known that every symplectic manifold admits a metaplectic-c structure \cite{ForgerHess,RobinsonRawnsley}, metaplectic-c prequantizability requires an additional curvature condition. We provide a modest argument that if a manifold
admits a symplectic form, it admits a symplectic form for which it is metaplectic-c
prequantizable
(Proposition~\ref{prop:universal}).

The \textbf{Einstein--Brillouin--Keller integrality condition} on Lagrangian submanifolds
(Bohr--Sommerfeld condition in \cite[II\S7, (7.6)]{GuilleminSternberg}, generalized
Bohr--Sommerfeld condition in \cite[Def.~3.4]{Gosson94}, Bohr--Sommerfeld condition with
constant shift in \cite{Sniatycki75}) is the triviality of the class
\begin{equation}\label{eq:ebkclassintro}
  \lambda\;\longmapsto\;
  \exp\Big(2\pi i\Big[\oint_\lambda\tfrac{i\,s^*_\lambda \gamma_{\mathrm{KS}}}{h}
    +\tfrac{\Mas(\lambda)}{4}\Big]\Big)
  \;\in\;\Hom_\Z\big(H_1(L;\Z),\U(1)\big)\;\cong\;H^1(L;\Uloc).
\end{equation}
The condition is well-defined when $\mu(\lambda)$ is well-defined modulo $4$. In this paper, the well-definedness of the classical EBK condition on a symplectic manifold will mean we assume $\mu(\lambda)$ is a well-defined integer for all loops in all Lagrangian submanifolds, and it makes sense to ask whether the class \eqref{eq:ebkclassintro} is trivial for any given Lagrangian submanifold.

The \textbf{Maslov correction} $\Mas(\lambda)/4 \pmod\Z$ accounts for the phase a wave function
acquires at caustics and at turning points (cf.\ \cite{BrackBhaduri,Stone}). It has a history of
contributions, from the invariant formulation of the Bohr--Sommerfeld conditions
\cite{Einstein}, its relevance to the Schr\"odinger equation \cite{Brillouin}, extending the
invariant derivation to involve quarter-integer phase shifts \cite{Keller}, to interpreting the
numerator $\Mas(\lambda)$ as a precise geometric count --- a Morse index
\cite[Lem.~8.7--8.8]{Maslov} --- and exhibiting it as a topological invariant
\cite[\S1.4]{Arnold}, and from general Fourier integral operators in
\cite[\S3.3]{Hormander}. The Maslov correction with quarter factor $\Mas(\lambda)/4$ was interpreted as the holonomy \cite{Sniatycki75} of half-forms introduced in \cite{BKS}, as the holonomy of tautological half-forms on the
Lagrangian Grassmannian \cite[\S4]{Weinstein}, and as the holonomy of the $2$-oriented Lagrangian Grassmannian over the Lagrangian Grassmannian in \cite[Def.~3.4]{Gosson94}. We refer to \cite[IV\S2, (2.26)]{GuilleminSternberg} for the discussion on its relationship with the invariant function on Lagrangian pairs of \cite[I\S2]{Leray}, the Wall--Kashiwara triple
index \cite{Wall}, \cite[\S1.5]{LionVergne}, and the cross index of four Lagrangian subspaces of \cite[\S3.3]{Hormander}. We interpret the factor $\mu(\lambda)/4$ from the geometry of the Lie groups (\S\ref{subsec:lagquant}) directly. The unobstructedness of metaplectic-c prequantizability (Proposition \ref{prop:universal}) suggests that \emph{one may view the term $\mu(\lambda)/4$ as a decoupled factor of an intrinsically global object under global-to-local restriction rather than as a ``correction'' term}.

In our derivation, a descent-and-pullback operation replaces the restriction mechanism. Metaplectic-c prequantizations descend
(Proposition~\ref{prop:gammadescends}) to prequantizations over the Lagrangian Grassmannian
bundle
\begin{equation}\label{eq:curvedmaslovintro}
\begin{tikzcd}[row sep=2.1em, column sep=small]
(Y,\tilde\gamma) \ar[d] & \text{principal } \U(1)\text{-bundle over }\Lag(TM),\
  \dd\tilde\gamma=i\,\pi^\ast\omega \\
\Lag(TM) \ar[d, "\pi"'] & \text{Lagrangian Grassmannian bundle} \\
M &
\end{tikzcd}
\end{equation}
and the holonomy of the pullback along a Gauss section $s_L:L\to\Lag(TM)|_L$ gives the EBK
condition. This generalizes a construction of \cite[IV\S2, p.~130]{GuilleminSternberg} on cotangent bundles that measures the integer Maslov index.

A sketch of the separate holonomy factors is as follows. When $M$ admits a metaplectic structure, it admits a fourfold Maslov covering compatible with
the oriented Lagrangian Grassmannian bundle \cite[\S2b]{Seidel}
\begin{equation}\label{eq:mu4coveringintro}
\begin{tikzcd}[row sep=2.1em, column sep=small]
\Lag(TM;\mu_4) \ar[d] & \text{principal } \mu_4\text{-bundle over }\Lag(TM) \\
\Lag(TM;\mu_2) \ar[d] & \text{oriented Lagrangian Grassmannian bundle} \\
\Lag(TM) \ar[d, "\pi"'] & \\
M. &
\end{tikzcd}
\end{equation} If $(M, \omega)$ additionally admits a Kostant--Souriau prequantization $Y_{\mathrm{KS}}$, the
amalgamated product $\Lag(TM;\mu_4)\times_{\mu_4}\pi^\ast Y_{\mathrm{KS}} \to \Lag(TM)$ is a
prequantization over the Lagrangian Grassmannian bundle, and the holonomy of the pullback along a Gauss section separates into a product of holonomies for each factor.

\subsection{The metaplectic-c prequantizability problem on symplectic manifolds}
\label{subsec:prequant}

In Section~\ref{sec:complexes}, we develop an obstruction theory generalizing Hess's solution to the metaplectic-c prequantizability problem \cite[Prop.~1.15]{Hess} to bring Hess's \v{C}ech cochain computations up to date with differential cohomology literature \cite{Brylinski,Gomi,Freed,FNSW,Waldorf,BunkSurvey}. We use groupoid-valued hypercohomology, because the groupoid formulation clarifies the distinction between the action on the homotopy fibre passed to isomorphism classes and the action on the fibre of the map on isomorphism classes (\S \ref{subsec:liftings}) as noted by \cite{HessKrausser}, and hypercohomology allows us to track the connection descent data (\S \ref{subsec:centralconn}) while retaining the homological algebra structure of the classical lifting problem. It applies to Lie group double coverings and associated circle extensions in general, and in particular, presents a unified treatment of $\Spin^c$ and $\Mp^c$ circle extensions.

Let $\sigma:\tilde G\to G$ be a Lie group double covering. Then take the circle central
extension of $G$ associated to $\sigma$
\begin{equation}\label{eq:mpc}
  1 \to \U(1) \to G^c \xrightarrow{\ \sigma^c\ } G \to 1 ,
\end{equation}
where
\[ G^c \;=\; (\tilde G\times\U(1))/\Delta, \]
with $\Delta\cong\Z_2$ the antidiagonal.

Given a principal $G$-bundle $P\to M$, a \emph{$\sigma^c$-lifting} is a principal $G^c$-bundle
$P^c$ together with an equivariant map to $P$ identifying $P^c/\U(1)$ with $P$
(Definition~\ref{def:lifting}). The existence problem for a $\sigma^c$-lifting is solved
\cite{Frenkel} (cf.\ \cite{HessKrausser,Brylinski}) by taking
the long exact sequence of pointed sets associated to the short exact sequence of sheaves of
group-valued functions
\[ 1 \to \Usm \to \mathcal G^c \to \mathcal G \to 1. \]
Since $\U(1)$ is central in $G^c$, the sequence of pointed sets exists in low degrees,
\[ H^1(M; \Usm)\;\to\;H^1(M; \mathcal{G}^c)
   \;\xrightarrow{\ [\cdot/\U(1)]\ }\;H^1(M; \mathcal{G})
   \;\xrightarrow{\ w(\cdot, \sigma^c)\ }\;H^2(M; \Usm). \]
The connecting map returns the Dixmier--Douady class
$w(P,\sigma^c)\in H^2(M;\Usm)\cong H^3(M;\Z)$, and defines the obstruction to the existence of a
$\sigma^c$-lifting of $P$.

In \S\ref{subsec:prescribedcurv} we construct a short exact sequence of complexes of sheaves
\[ 1 \to \Delta^\bullet \to (\tilde{\mathcal G}\times\Usm)^\bullet
   \to (\mathcal G^c)^\bullet \to 1 \]
and take the relevant exact sequence in hypercohomology groupoids. The complexes may be nonabelian in degree zero, so the notion of hypercohomology and exactness is spelled out by hand. On isomorphism classes this gives an exact sequence of pointed sets
\[ \pi_0\mathbf{H}^1(\Delta^\bullet)\;\to\;\pi_0\mathbf{H}^1\big((\tilde{\mathcal G}\times\Usm)^\bullet\big)\;\to\;
   \pi_0\mathbf{H}^1\big((\mathcal G^c)^\bullet\big)\;\xrightarrow{\ \partial\ }\;\HH^2(\Delta^\bullet), \]
where the small groupoid $\mathbf{H}^1((\mathcal G^c)^\bullet)$ is equivalent to the groupoid of principal $G^c$-bundles with central
connection (\S\ref{subsec:centralconn}), and the genuine hypercohomology group
$\HH^2(\Delta^\bullet)\cong H^2(M;\Z_2)\oplus Z^2_{\dR}(M)$
 (Lemma~\ref{lem:gpdLES}). The connecting map $\partial$ returns both the
Stiefel--Whitney obstruction class $w(P,\sigma)\in H^2(M;\Z_2)$ for the $\sigma$-lifting of
$P$, and the curvature $2$-form of the central connection $\gamma$.

This structure is then massaged into the obstruction theories (Theorem~\ref{thm:Pc}, Proposition~\ref{prop:splittings}). Combining the components of $\partial$
(Lemma~\ref{lem:u1agreement}) gives a class
\begin{equation}\label{eq:mainclass}
  w(P,\sigma^c,\omega)\in H^2(M;\Uloc),
  \qquad
  S\longmapsto e^{\frac{i}{\hbar}\int_S\omega}\cdot(-1)^{\langle w(P,\sigma),S\rangle}
  \quad (S\in H_2(M;\Z)),
\end{equation}
where we identify $H^2(M;\Uloc)\cong\Hom(H_2(M;\Z),\U(1))$ by the universal coefficient theorem.

The class $w(P,\sigma^c,\omega)$ is the obstruction to the existence of $\sigma^c$-liftings
\emph{with prescribed central curvature} (Definition~\ref{def:centralconnection}).

\begin{namedtheorem}[\ref{cor:Pc}]
Let $M$ be a paracompact smooth manifold, $\omega$ a closed $2$-form on $M$, and
$\sigma^c:G^c\to G$ a circle extension associated to a Lie group double covering
$\sigma:\tilde G\to G$. Then a principal $G$-bundle $P\to M$ admits a $\sigma^c$-lifting with a
central connection of curvature $\omega$ if and only if $w(P,\sigma^c,\omega)$ is trivial. The
nonempty isomorphism classes of such liftings form a torsor under $H^1(M;\U(1))$.
\end{namedtheorem}
For the metaplectic-c circle extension, these results were obtained in \cite[Thm.~1.14, Thm.~1.17]{Hess}.

The obstructions given by the two connecting maps are related by the Bockstein homomorphism.

\begin{namedtheorem}[\ref{thm:bockstein}]
Let $G$ be a Lie group, and $\sigma^c:G^c\to G$ the circle extension \eqref{eq:mpc} associated
to a Lie group double covering $\sigma:\tilde{G}\to G$. Let $P\to M$ be a principal $G$-bundle
and let $\omega$ be a closed $2$-form on $M$. Then
\[ \beta\big(w(P,\sigma^c,\omega)\big)=w(P,\sigma^c)\;\in\;H^3(M;\Z), \]
where $\beta$ is the Bockstein homomorphism for the exact sequence
\[ 0 \to \Z \to \R \to \U(1) \to 1 . \]
\end{namedtheorem}

\subsection{The Maslov correction on the Lagrangian Grassmannian}
\label{subsec:lagquant}

In Section~\ref{sec:ebk}, we use character identities to consolidate different accounts of half-forms in the literature.
Recall that the Maslov correction originates in semiclassical analysis \cite[Lem.~6.11]{Maslov}, \cite[Prop.~1.2.3]{Duistermaat},
\cite[Ch.~I]{GuilleminSternberg} from the principle of stationary phase in the limit $\hbar\to0$. If
\[ f(x,y)=x_1^2+\dots+x_k^2-y_1^2-\dots-y_{n-k}^2 \]
is the Morse normal form of a function at a nondegenerate critical point, then
\begin{equation}\label{eq:stationaryphase}
  \int_{\R^n} e^{if/\hbar}\,\dd x\,\dd y
  \;=\;(\pi\hbar)^{n/2}\,
  e^{\frac{i\pi}{4}\operatorname{sig}\Hess f}
\end{equation}
as an oscillatory integral. For a general amplitude this factor appears as the leading term of
the stationary phase expansion, and each critical point contributes an eighth root of unity
determined by the signature of the Hessian $\operatorname{sig}\Hess f$. The choices of generating functions presenting the
same Lagrangian submanifold introduce differences of Hessian
signatures \cite[Prop.~1.2.3]{Duistermaat}, and their total contribution gives the Maslov correction
$\Mas(\lambda)/4$.

The oriented Lagrangian
Grassmannian $\Lag(\bV;\mu_2)$ has the particular property that it is a connected, non-simply-connected double covering of
$\Lag(\bV)$, with fundamental group isomorphic to $\Z$. This property enabled Leray to interpret the Maslov correction through generalizations of orientations on Lagrangian
subspaces in \cite[I\S2.10]{Leray}. 

Consider the connected double covering $\Lag(\bV;\mu_4) \to \Lag(\bV; \mu_2)$. As an element of $\Lag(\bV;\mu_2)$ is a Lagrangian subspace with a $1$-orientation, an element of
$\Lag(\bV;\mu_4)$ may be viewed as a Lagrangian subspace with a $2$-orientation. On this tower, a
loop $\lambda$ in $\Lag(\bV)$ lifts to a loop in $\Lag(\bV;\mu_4)$ if and only if its Maslov
index is divisible by four, equivalently if $\Mas(\lambda)/4$ is integral.

We may further justify Leray's interpretation by considering the stabilizer $\Sp_{\bF}(\bV)$ of
a fixed Lagrangian subspace $\bF\subseteq\bV$ in $\Sp(\bV)$, and its identity component
$\Sp^+_{\bF}(\bV)$. The inclusion $\Sp^+_{\bF}(\bV)\hookrightarrow\Sp(\bV)$ lifts to the Lie
group double covering $\ssigma:\Mp(\bV)\to\Sp(\bV)$ (Lemma~\ref{lem:lifts}), and we may express
the covering spaces as homogeneous spaces
\[ \begin{tikzcd}[row sep=2.2em, column sep=2.6em]
\Lag(\bV;\mu_4) \ar[r, "2{:}1"] \ar[d, "\cong"'] &
  \Lag(\bV;\mu_2) \ar[r, "2{:}1"] \ar[d, "\cong"'] &
  \Lag(\bV) \ar[d, "\cong"'] \\
\Mp(\bV)/\Sp^+_{\bF}(\bV) \ar[r] &
  \Sp(\bV)/\Sp^+_{\bF}(\bV) \ar[r] &
  \Sp(\bV)/\Sp_{\bF}(\bV) ,
\end{tikzcd} \]
and $\Sp^+_{\bF}(\bV)$ is genuinely the stabilizer of a Lagrangian subspace with an orientation.

On the other hand, Blattner--Kostant--Sternberg \cite{BKS} introduced half-forms as square roots
of volume forms on complex polarizations. The metalinear double covering
$\ML(\bF)\to\GL(\bF)$ has a character
$\chi_{\mathrm{BKS}}:\ML(\bF)\to\C^\times$ (\S\ref{subsec:characters}) which squares to the determinant of the
covering projection. Half-forms are defined as sections of the associated bundle for a metalinear frame bundle
\cite{BKS,Blattner,BlattnerNonReal}, and were introduced to compare Hilbert spaces associated to different polarizations (as well as to provide canonical integration measures).

The natural correspondence between orientations and normalized volume
forms extends to a correspondence between Leray's $2$-orientations and normalized BKS half-forms up to a coherent choice of primitive fourth root of unity. This can be stated as follows.

\begin{namedtheorem}[\ref{thm:threeways}]
For a coherent choice of primitive fourth root of unity, the associated bundles over the
Lagrangian Grassmannian $\Lag(\bV)$
\[ \Lag(\bV;\mu_4)\times_{\mu_4}\U(1)
   \;\cong\;\widetilde{\mathfrak B}(\tau_{\bV})\times_{\ML(\bF)}\U(1)
   \;\cong\;\Mp^c(\bV)\times_{\Mp^c_{\bF}(\bV)}\U(1) , \]
are isomorphic as flat $\U(1)$-bundles with holonomy a primitive fourth root of unity over the
Maslov loop (a generator of $\pi_1(\Lag(\bV))\cong\Z$).
\end{namedtheorem}

Here $\widetilde{\mathfrak B}(\tau_{\bV})$ is the bundle of metalinear frames of the tautological bundle $\tau_{\bV} \to \Lag(\bV)$ \cite{BKS,GuilleminSternberg}.

The idea of the proof is to view $\Mp(\bV)\to\Lag(\bV)\cong\Mp(\bV)/\Mp_{\bF}(\bV)$ as a principal
$\Mp_{\bF}(\bV)$-bundle, where $\Mp_{\bF}(\bV)$ is the preimage of $\Sp_{\bF}(\bV)$ in the
metaplectic covering $\ssigma:\Mp(\bV)\to\Sp(\bV)$. The component group
$\pi_0(\Mp_{\bF}(\bV))$ is cyclic of order four (Lemma~\ref{lem:lifts}), so we may choose an
identification $\widetilde{\ori}:\pi_0(\Mp_{\bF}(\bV))\xrightarrow{\ \cong\ }\mu_4$. This character
relates to the normalized Hess character $\hat\chi^c:\Mp^c_{\bF}(\bV)\to\U(1)$
\cite[Thm.~3.5]{Hess} and the normalized BKS character
$\hat\chi_{\mathrm{BKS}}:\ML(\bF)\to\U(1)$ \cite[Lem.~5.3(b)]{BKS},
\cite[(4.1), p.~253]{GuilleminSternberg} by
\begin{equation}\label{eq:chiagree}
  \widetilde{\ori}
  \;=\;\hat\chi_{\mathrm{BKS}}\circ\widetilde{\ind}
  \;=\;\hat\chi^c\big|_{\Mp_{\bF}(\bV)} .
\end{equation}
Equivalently, the diagram
\begin{equation}\label{eq:threecharacters}
\begin{tikzcd}[row sep=2.2em, column sep=5em, ampersand replacement=\&]
 \& \ML(\bF) \ar[dr, "\hat\chi_{\mathrm{BKS}}"] \& \\
\Mp_{\bF}(\bV) \ar[ur, "\widetilde{\ind}"] \ar[r, "\widetilde{\ori}"] \ar[dr, hook, "\inc"'] \&
  \mu_4 \ar[r, hook, "\inc"] \& \U(1) \\
 \& \Mp^c_{\bF}(\bV) \ar[ur, "\hat\chi^c"'] \&
\end{tikzcd}
\end{equation}
commutes (Lemmas~\ref{lem:chicrestricts} and \ref{lem:indtilde}). The theorem then follows
because the intermediate factor of a composition of associated bundle constructions is itself an
associated bundle, and the associated bundle of the composite character is the composition of
the operations of taking associated bundles.

Here $\Mp^c_{\bF}(\bV)$ is the preimage of $\Sp_{\bF}(\bV)$ in the metaplectic-c circle extension
$\ssigma^c:\Mp^c(\bV)\to\Sp(\bV)$. While the metaplectic group $\Mp(\bV)$ admits a lifting of
$\Sp^+_{\bF}(\bV)\hookrightarrow\Sp(\bV)$ but not of $\Sp_{\bF}(\bV)\hookrightarrow\Sp(\bV)$, the
metaplectic-c group $\Mp^c(\bV)$ admits liftings of both (Lemma~\ref{lem:lifts}). The geometry of
this situation is shown in Figure~\ref{fig:gssquare}. So a principal $\Sp_{\bF}(\bV)$-bundle need
not extend to a principal $\Mp_{\bF}(\bV)$-bundle, but always extends to a principal
$\Mp^c_{\bF}(\bV)$-bundle.

\subsection{Construction and factorization of the EBK holonomy}
\label{subsec:gauss}

A metaplectic-c prequantization $(P^c,\gamma)\to P\to M$ descends to a prequantization over the
Lagrangian Grassmannian bundle as follows. By \cite[Prop.~5.5]{KobayashiNomizu}, associated
bundles may be identified with the orbit space quotients, so
\[ \Lag(TM)\;:=\;P\times_{\Sp(\bV)}\Lag(\bV)
   \;=\;P\times_{\Sp(\bV)}(\Sp(\bV)/\Sp_{\bF}(\bV))
   \;=\;P/\Sp_{\bF}(\bV) , \]
and alternatively
\[ \Lag(TM)\;=\;P^c\times_{\Mp^c(\bV)}(\Mp^c(\bV)/\Mp^c_{\bF}(\bV))
   \;\cong\;P^c/\Mp^c_{\bF}(\bV). \]
Hence $P^c\to\Lag(TM)$ is a principal $\Mp^c_{\bF}(\bV)$-bundle. Then taking the bundle
associated by the normalized Hess character $\hat\chi^c$ of \S\ref{subsec:gssquare} gives
a principal $\U(1)$-bundle over $\Lag(TM)$,
\begin{equation}\label{eq:Yassoc}
  Y\;:=\;P^c\times_{\hat\chi^c}\U(1)\;\longrightarrow\;\Lag(TM).
\end{equation}
The central connection $\gamma$ descends to a principal circle connection $\tilde\gamma$ on
$Y\to\Lag(TM)$. There are two choices of Hess character corresponding to a choice of primitive fourth root of unity and the descent depends on this choice. 

\begin{namedproposition}[\ref{prop:gammadescends}]
Let $(P^c,\bu,\gamma)$ be a metaplectic-c prequantization of a symplectic manifold $(M,\omega)$,
and let $\pi:\Lag(TM)\to M$ be its Lagrangian Grassmannian bundle. Then for a choice of normalized Hess character $\hat{\chi}^c$, the central connection $\gamma$ descends to a
principal connection $\tilde\gamma$ on $Y\to\Lag(TM)$, whose curvature form is $\pi^\ast\omega$.
Moreover, $\tilde\gamma$ is flat on each fibre $\Lag(T_xM)$, with holonomy a primitive fourth
root of unity on the generators of $\pi_1(\Lag(T_xM))$.
\end{namedproposition}

The unique character $\eta:\Mp^c(\bV)\to\U(1)$ restricting to the squaring map on the central
circle \cite[\S43]{Weil-unitary} gives a principal circle bundle with connection $(Y_\eta,\gamma_\eta)$ over $M$, where
$Y_\eta:=P^c\times_\eta\U(1)\cong P^c/\Mp(\bV)$ and $\gamma_\eta$ descends from $2\gamma$.

For a Lagrangian submanifold $\iota_L:L\hookrightarrow M$, the prequantization over $\Lag(TM)$
pulls back by the Gauss section $s_L:L\to\Lag(TM)|_L$ with the following holonomies.

\begin{namedtheorem}[\ref{thm:ebk}]
Let $(M,\omega)$ be metaplectic-c prequantizable, let $(P^c,\gamma)$ be a metaplectic-c
prequantization of $(M,\omega)$, and let $(Y,\tilde\gamma)\to\Lag(TM)$ be its descent to the
Lagrangian Grassmannian bundle for a choice of normalized Hess character $\hat{\chi}^c$. Let $L\subseteq M$ be a Lagrangian submanifold and let
$\lambda\in H_1(L;\Z)$.
\begin{enumerate}[label=\textup{(\arabic*)}]
\item The holonomy of $s_L^\ast(Y,\tilde\gamma)$ on $\lambda$ satisfies
\[ \Hol_{s_L^\ast Y}(\lambda)^2
   \;=\;\exp\Big(2\pi i\oint_\lambda\tfrac{i\,s_\lambda^*\gamma_\eta}{h}\Big)
        \cdot(-1)^{\langle w_1(TL),\lambda\rangle} . \]
For the two choices of Hess character (\S\ref{subsec:gssquare}), the holonomies differ by
$(-1)^{\langle w_1(TL),\lambda\rangle}$.

\item If $w_2(TM)=0$, there exists a splitting
$(P^c,\gamma)\cong\tilde P\times_{\mu_2}(Y_{\mathrm{KS}},\gamma_{\mathrm{KS}})$
(\S\ref{subsec:splitting}), and for every Lagrangian submanifold $L$, the holonomy of
$s_L^\ast(Y,\tilde\gamma)$ on $\lambda$ splits as
\[ \Hol_{s_L^\ast Y}(\lambda)
   \;=\;\exp\Big(2\pi i\oint_\lambda\tfrac{i\,s_\lambda^*\gamma_{\mathrm{KS}}}{h}\Big)
        \cdot i^{\Mas_4(\lambda)} , \]
where $\Mas_4(\lambda)$ is the mod $4$ Maslov index of $\lambda$ with respect to the fourfold Maslov covering associated to $\tilde P$ (\S\ref{subsec:ebkcondition}).

\item If $c_1(TM)=0$, there exist a splitting
$(P^c,\gamma)\cong\tilde P\times_{\mu_2}(Y_{\mathrm{KS}},\gamma_{\mathrm{KS}})$ and a lifting of
its mod $4$ Maslov index $\Mas_4$ to an integer-valued Maslov index $\Mas:H_1(L;\Z)\to\Z$ such that
\[ \Hol_{s_L^\ast Y}(\lambda)
   \;=\;\exp\Big(2\pi i\Big[\oint_\lambda\tfrac{i\,s_\lambda^*\gamma_{\mathrm{KS}}}{h}
     +\tfrac{\Mas(\lambda)}{4}\Big]\Big) . \]

\item The splitting obstructions restrict to $L$ as
\[
  \iota_L^\ast w_2(TM)=w_1(TL)^2\ \text{ in } H^2(L;\Z_2),
  \qquad
  \iota_L^\ast c_1(TM)=\beta_2\big(w_1(TL)\big)\ \text{ in } H^2(L;\Z).
\]
Consequently, if $L$ is orientable, there
exist a splitting of $\iota_L^\ast(P^c,\gamma)$ and a lifting of its mod $4$ Maslov index to an
integer-valued Maslov index $\Mas:H_1(L;\Z)\to\Z$ such that
\[ \Hol_{s_L^\ast Y}(\lambda)
   \;=\;\exp\Big(2\pi i\Big[\oint_\lambda\tfrac{i\,s_\lambda^*\gamma_{\mathrm{KS}}}{h}
     +\tfrac{\Mas(\lambda)}{4}\Big]\Big) . \]
\end{enumerate}
\end{namedtheorem}
We call $L$ \emph{quantizable with respect to $(P^c,\gamma,\hat\chi^c)$} when this holonomy
character is trivial (Definition~\ref{def:quantizable}). By \textup{(1)}, the choice of
$\hat\chi^c$ matters only on nonorientable Lagrangian submanifolds; for
$\mathbb{RP}^2\subset\mathbb{CP}^2$ exactly one choice makes $\mathbb{RP}^2$ quantizable
(Example~\ref{ex:rp2}).

The holonomies can be compared with the holonomies of the WKB ansatz as follows. The Hess character satisfies
\begin{equation}\label{eq:chisquare1}
  \chi^c\big([\tilde g,\xi]\big)
  \;=\;\xi\,\chi(\tilde g) ,\qquad [\tilde g,\xi]\in\Mp^c_{\bF}(\bV) ,
\end{equation}
and admits the polar factorization
\begin{equation}\label{eq:chipolar}
  \chi^c
  \;=\;\big|\chi^c\big|\cdot\hat\chi^c .
\end{equation}
We can read the value in two different ways,
\[ \underbrace{\xi}_{\text{Kostant--Souriau phase}}
   \cdot\underbrace{\chi(\tilde g)}_{\text{half-form values}}
   \;=\;
   \underbrace{\big|\chi(\tilde g)\big|}_{\substack{\text{WKB amplitude}\\ \text{(half-densities)}}}
   \cdot\underbrace{\hat\chi(\tilde g)\,\xi}_{\text{WKB phase}} . \]
The representative $[\tilde g,\xi]$ is determined
only up to the simultaneous sign change $(\tilde g,\xi)\mapsto(\tilde g\beps,-\xi)$, so
$\xi$, $\chi(\tilde g)$ and $\hat\chi(\tilde g)$ are not
separately homomorphisms on $\Mp^c_{\bF}(\bV)$.

If $P^c$ is the metaplectic-c prequantization obtained from a metaplectic structure and a
Kostant--Souriau prequantization, the line bundle $s_L^\ast P^c\times_{\chi^c}\C$ factors into
the Kostant--Souriau prequantization with its half-form correction (restricted to $L$), and into
the tensor product of half-densities with the Maslov line bundle (with Kostant--Souriau phase)
\cite[(1.5.3)]{Duistermaat}, characterizing the WKB ansatz $\psi=a\,e^{iS/\hbar}$ on $L$.

\begin{table}[htbp]
\centering\small
\renewcommand{\arraystretch}{1.45}
\setlength{\aboverulesep}{0pt}\setlength{\belowrulesep}{0pt}
\setlength{\extrarowheight}{2.5pt}
\begin{tabular}{@{}>{\raggedright\arraybackslash}p{3.8cm}
                 |>{\raggedright\arraybackslash}p{4.0cm}
                 |>{\raggedright\arraybackslash}p{6.3cm}@{}}
\toprule
\emph{Integrality condition \newline (differential obstruction)} &
\emph{Topological obstruction} &
\emph{Construction when unobstructed} \\
\midrule[\heavyrulewidth]
Kostant--Weil \cite{Weil-connection,Kostant,Souriau}
  & \multicolumn{1}{>{\centering\arraybackslash}p{4.0cm}|}{---}
  & Kostant--Souriau prequantization on $M$ \\
\multicolumn{1}{@{}>{\centering\arraybackslash}p{3.8cm}|}{---}
  & Stiefel--Whitney $w_2(TM)$ \cite{KostantSpinors}
  & Metaplectic lifting on $M$, metalinear frame bundle on $\Lag(TM)$ \\
\multicolumn{1}{@{}>{\centering\arraybackslash}p{3.8cm}|}{---}
  & By polarization type \newline \cite{BKS,GuilleminSternberg,Sniatycki75}
  & Metalinear frame bundle on $(M,F)$ \cite{BKS} \\
\multicolumn{1}{@{}>{\centering\arraybackslash}p{3.8cm}|}{---}
  & $2c_1(TM)$ \cite{Seidel}
  & $\infty$-fold Maslov covering on $\Lag(TM)$ \cite{Kontsevich}
    (cotangent bundles \cite{GuilleminSternberg}) \\
\multicolumn{1}{@{}>{\centering\arraybackslash}p{3.8cm}|}{---}
  & $2c_1(TM)$ mod $N$ \cite{Seidel}
  & $N$-fold Maslov covering on $\Lag(TM)$ \cite{Seidel} \\
\hline\hline
Kostant--Weil
  & By polarization type
  & Half-form correction to Kostant--Souriau for fixed $F$ \\
Kostant--Weil
  & Stiefel--Whitney $w_2(TM)$
  & BKS pairings \cite{BKS}, fourfold Maslov coverings,
    prequantizations over $\Lag(TM)$ \\
\cmidrule[\cmidrulewidth]{1-3}
\multicolumn{2}{@{}>{\centering\arraybackslash}p{8.1cm}|}{%
   Metaplectic-c prequantizability
   \cite{Czyz,CzyzMaslov,KarasevMaslov,DazordPatissier,Hess,RobinsonRawnsley}}
  & $\Mp^c$ prequantization \cite{Hess}, prequantization over $\Lag(TM)$ \\
\hline\hline
Bohr--Sommerfeld
  & \multicolumn{1}{>{\centering\arraybackslash}p{4.0cm}|}{---}
  & Parallel section of $\iota_L^\ast Y_{\mathrm{KS}}$ \\
\multicolumn{1}{@{}>{\centering\arraybackslash}p{3.8cm}|}{---}
  & Maslov index \cite{Maslov,Arnold,Hormander}
  & Lifting of curves to the universal covering of $\Lag(\bV)$ \\
\multicolumn{1}{@{}>{\centering\arraybackslash}p{3.8cm}|}{---}
  & Maslov index mod $2q$ \cite{Gosson} (mod $4$ \cite{Hormander,Gosson94})
  & $q$-orientations \cite{Leray} (parallel sections of the Maslov line bundle
    \cite{Hormander}) \\
\multicolumn{1}{@{}>{\centering\arraybackslash}p{3.8cm}|}{---}
  & Maslov correction as holonomy \cite{Weinstein}
  & Lifting of curves to the fourfold covering of $\Lag(\bV)$ \\
\multicolumn{1}{@{}>{\centering\arraybackslash}p{3.8cm}|}{---}
  & Pullback of a Maslov covering \cite{Seidel}
  & Graded Lagrangian submanifolds \cite{Seidel} \\
\cmidrule[\cmidrulewidth]{1-3}
\multicolumn{2}{@{}>{\centering\arraybackslash}p{8.1cm}|}{%
   Bohr--Sommerfeld with constant shift \cite{Sniatycki75}}
  & --- \\
\cmidrule[\cmidrulewidth]{1-3}
\multicolumn{2}{@{}>{\centering\arraybackslash}p{8.1cm}|}{%
   Einstein--Brillouin--Keller \cite{Keller}}
  & Parallel section of $s_L^\ast Y$, canonical operators \cite{KarasevMaslov},
    global principal symbols on $L$ \cite{Duistermaat,GuilleminSternberg} \\
\bottomrule
\end{tabular}
\caption{Obstructions and constructions. The obstruction class for the existence of
metalinear frame bundles of a polarization $F$ lies in $H^2(M;\Z_2)$ \cite[\S5]{BKS} and
depends on the
polarization type of $F$: it is $w_1(F)^2$ for real $F$ \cite[V\S4]{GuilleminSternberg}
and $w_2(F)$ for complex $F$, denoted $[c]$ in \cite[Thm.~3.1]{Sniatycki75}. For a symplectic
manifold equipped with a polarization there is another, separate quantizability problem, in
which the cohomological degree itself encodes the polarization type
\cite{Sniatycki77,MirandaPresas,BlattnerRawnsley}.}
\label{tab:obstructions}
\end{table}
\FloatBarrier

\subsection{Notation}
\label{subsec:notation}

Throughout, $M$ is a paracompact smooth manifold.

\begin{enumerate}

    \item \textbf{Circles.} $\U(1)$ is the unit circle in $\C$, written multiplicatively, with
    Lie algebra $\u(1)\cong i\R$. For $z\in\C^\times$ we denote $\hat z:=z/|z|\in\U(1)$.
    Central connections are $\u(1)$-valued and curvature forms are
    real valued. We denote by $\Usm$ the sheaf of smooth $\U(1)$-valued functions on $M$ and by
    $\Uloc$ the sheaf of locally constant ones, so that
    $H^r(M;\Uloc)\cong H^r_{\mathrm{sing}}(M;\U(1))$ and $H^2(M;\Usm)\cong H^3(M;\Z)$. For two
    circle bundles with connection $Y_1$, $Y_2$, we denote by $Y_1\otimes Y_2$ the circle bundle
    with connection whose corresponding hermitian line bundle is the tensor product of the
    hermitian line bundles corresponding to $Y_1$, $Y_2$ (the Picard product, \S\ref{subsec:ebklag}).

    \item \textbf{Cyclic groups.} We view the group of $k$-th roots of unity
    $\mu_k\le\U(1)$ as a subgroup of the circle, written multiplicatively, and use it as the
    value group of a
    character, the fibre of a covering, or a holonomy. We denote by $\Z_k$ the abstract cyclic
    group, written additively and used for cohomology coefficients.

    \item \textbf{Normalizations.} Throughout, we use the normalization $h=2\pi\hbar=2\pi$,
    and only display $h$ and $\hbar$ in Section~\ref{sec:intro}. To restore the dependence
    on $h$, one views the radius of $\U(1)$, the connection, and curvature forms as quantities
    with dimensions of $\hbar$. Then the normalized quantities $\omega/h$ and $i\gamma/h$
    appearing in the integrality conditions are geometric, as are the differentials $\dd^\bullet$ of sheaf
    complexes.

    \item \textbf{Suppressions.} For \v{C}ech cochains we omit indices: $\gamma$ denotes
    $\{\gamma_j\}$ and $c$ denotes $\{c_{jk}\}$. We abbreviate hypercohomology
    $\HH^r(M;\mathcal F^\bullet)$ to $\HH^r(\mathcal F^\bullet)$ when $M$ is fixed.

    \item \textbf{Signs.} The sign in $\dd\gamma=\pm i\,(\rho^c)^\ast\omega$, and so in
    $\curv\gamma$ of Definition~\ref{def:centralconnection}, depends on the
    identification $\u(1)\cong i\R$. We follow the positive sign convention of
    \cite{Hess,GuilleminSternberg}; see \cite{RobinsonRawnsley} for the negative sign
    convention. Holonomy is defined by $\ell(1)=\ell(0)\cdot\Hol(\lambda)$ for a horizontal lift
    $\ell$ of a loop $\lambda$, and the monodromy of a covering likewise.
    If the curvature is left $\u(1)$-valued with no identification \cite{KobayashiNomizu}, or
    the circle is written additively as $\R/h\Z$ \cite{Kostant}, no sign choice arises. The sign
    of the Maslov correction likewise depends on the identification
    $\pi_0(\Mp_{\bF}(\bV))\cong\mu_4$ of \S\ref{subsec:gssquare}.
    The \v{C}ech differential is $\check\delta$, normalized on a
    $0$-cochain by $(\check\delta\gamma)_{jk}=\gamma_k-\gamma_j$, and transition functions are
    taken from local sections by $s_k=s_j\,c_{jk}$. Reversing either convention reverses the sign
    in the compatibility between local connection forms,
    $\check\delta\gamma=\dd\log c$.

    \item \textbf{Principal bundles.} We denote by $\mathbf P(H)$ the groupoid of principal
    $H$-bundles on $M$ and their isomorphisms, and $\mathbf P(H;\mathfrak h)$ when the bundles
    carry an $\mathfrak h$-valued connection; the subscript $0$ indicates flat connections, as in
    $\mathbf P_0(H;\mathfrak h)$. For abelian $H$ we denote by $\mathcal P(H;\mathfrak h)$ the
    abelian group $\pi_0\mathbf P(H;\mathfrak h)$ of isomorphism classes, and likewise for the
    decorated versions.

\end{enumerate}

\section{The homological algebra of metaplectic-c prequantizability}
\label{sec:complexes}

\subsection{Kostant--Weil integrality from truncated complexes}
\label{subsec:kostantweil}

Brylinski described the Kostant--Souriau prequantizability problem using
smooth Deligne complexes \cite[Thms.~2.2.12, 2.2.14]{Brylinski}. In this section we present
Hess's metaplectic-c prequantizability criterion \cite[Thm.~1.14]{Hess} with some modifications
to Brylinski's method.

We use the following standard facts about sheaf hypercohomology, following the notation of
\cite[\S1.2, \S1.3]{Brylinski}. For a bounded-below complex $\mathcal F^\bullet$ of sheaves of
abelian groups on $M$, $\HH^\bullet(\mathcal F^\bullet)$ is the cohomology of the derived global
sections functor $R\Gamma(-)$ applied to $\mathcal F^\bullet$. On a good cover $\mathfrak U$ that
is acyclic for $\mathcal F^\bullet$, it is computed by the total complex $(\mathcal K^\bullet,\check\dd)$ of the \v{C}ech double complex
$C^p(\mathfrak U;\mathcal F^q)$, with total differential $\check\dd=\check\delta+(-1)^p\dd$. A
quasi-isomorphism of complexes of sheaves induces isomorphisms on sheaf hypercohomology, so the
hypercohomology of a resolution of a sheaf $\underline A$ computes the ordinary sheaf cohomology
$H^\bullet(M;\underline A)$. A termwise short exact sequence of complexes of sheaves induces a
long exact sequence in sheaf hypercohomology. A morphism of complexes of sheaves induces a map
on sheaf hypercohomology.

Recall the statement of Kostant--Weil integrality.

\begin{theorem}[Th\'eor\`eme V.1 of \cite{Weil-connection}, Proposition 2.1.1 of \cite{Kostant}]
\label{thm:kostantweil}
Let $M$ be a smooth manifold and let $\omega$ be a closed $2$-form on $M$.
\begin{enumerate}[label=\textup{(\arabic*)}]
\item There exists a circle bundle with connection $\pi:(Y,\gamma)\to M$ with
$\dd\gamma=i\,\pi^\ast\omega$ if and only if $\omega$ is integral, that is,
$[\omega/2\pi]\in H^2(M;\R)$ lies in the image of $H^2(M;\Z)\to H^2(M;\R)$.
\item $H^1(M;\U(1))$ acts freely and transitively on the isomorphism classes of principal circle
bundles with connection whose curvature is $\omega$.
\end{enumerate}
\end{theorem}

Both statements can be encoded in the exact sequence
\begin{equation}\label{eq:kostantweil}
  0 \to \mathcal{P}_0(\U(1);\u(1)) \to \mathcal{P}(\U(1);\u(1)) \xrightarrow{\curv}
  Z^2_{\dR}(M) \xrightarrow{\ e^{i\int}\ } H^2(M;\Uloc),
\end{equation}
where $\mathcal{P}_0(\U(1);\u(1))$ and $\mathcal{P}(\U(1);\u(1))$ are the sets of isomorphism
classes of principal circle bundles with flat connection, and with connection, respectively, and
$Z^2_{\dR}(M)$ denotes the closed $2$-forms on $M$, with the additional observation that every
nonempty preimage of $\curv$ is a $\mathcal P_0(\U(1);\u(1))$-torsor. The map
$e^{i\int}:Z^2_{\dR}(M)\to H^2(M;\Uloc)$ is defined under the identification
\[ H^2(M;\Uloc)\cong\Hom_\Z\big(H_2(M;\Z),\U(1)\big), \]
provided by the universal coefficient theorem, as
\[ \omega\;\longmapsto\;\big(S\mapsto e^{i\int_S\omega}\big). \]
Under the correspondence between principal circle
bundles with connection and hermitian line bundles with connection, tensor product induces the
Picard product (\S\ref{subsec:ebklag}) making these sets abelian groups and \eqref{eq:kostantweil} an exact
sequence of abelian groups.

The integral forms are exactly $Z_{\dR}^2(M)_{\Z}=\ker e^{i\int}$, so the left-hand portion of
\eqref{eq:kostantweil} is the curvature exact sequence
\begin{equation} \label{eq:curvatureses}
  0 \to \mathcal{P}_0(\U(1);\u(1)) \to \mathcal{P}(\U(1);\u(1)) \xrightarrow{\curv}
  Z^2_{\dR}(M)_\Z\to 0 .
\end{equation}

\begin{figure}[htbp]
\centering
\begin{tikzpicture}[>=stealth,font=\footnotesize]
  \def\dx{3.9}\def\dy{2.1}
  % --- cells: only the cochains themselves
  \node (c02) at (0,2*\dy)      {$\dd\gamma$};
  \node (c01) at (0,\dy)        {$\gamma$};
  \node (c11) at (\dx,\dy)      {$\check\delta\gamma-\dd\log c$};
  \node (c10) at (\dx,0)        {$c$};
  \node (c20) at (2*\dx,0)      {$\check\delta c$};
  % --- differentials
  \draw[->] (c01) -- node[above,font=\scriptsize,pos=0.22,fill=white,inner sep=1.5pt]{$\check\delta$} (c11);
  \draw[->] (c10) -- node[above,font=\scriptsize,pos=0.28,fill=white,inner sep=1.5pt]{$\check\delta$} (c20);
  \draw[->] (c01) -- node[left,font=\scriptsize,fill=white,inner sep=1.5pt,pos=0.28]{$\dd$} (c02);
  \draw[->] (c10) -- node[right,font=\scriptsize,pos=0.72,fill=white,inner sep=1.5pt]{$\dd\log$} (c11);
  % --- the box the truncation deletes, with its reading
  \begin{scope}[on background layer]
    \node[fill=black!16,rounded corners,fit=(c02),inner sep=5pt](box){};
    \node[fill=black!6,rounded corners,fit=(c11),inner sep=5pt]{};
    \node[fill=black!6,rounded corners,fit=(c20),inner sep=5pt]{};
  \end{scope}
  \node[anchor=west,font=\scriptsize,align=left] at ($(box.east)+(0.35,0)$)
    {flatness condition in $\mathcal C^\bullet$;\\ empty condition in $\mathcal D^\bullet$};
  % --- readings of the two degree-one cochains
  \node[font=\scriptsize,anchor=north west] at ($(c01.south)+(-0.30,-0.12)$) {connection forms};
  \node[font=\scriptsize,anchor=north] at ($(c10.south)+(0,-0.12)$) {transition functions};
  \node[font=\scriptsize,anchor=west,align=left] at ($(c11.east)+(0.30,0)$)
    {transformation law\\ for connection forms};
  \node[font=\scriptsize,anchor=north] at ($(c20.south)+(0,-0.12)$) {bundle gluing};
  % --- axes
  \draw[-] (-1.0,-0.95) -- (2*\dx+1.0,-0.95);
  \draw[-] (-1.0,-0.95) -- (-1.0,2*\dy+0.7);
  \foreach \i/\lab in {0/{C^0},1/{C^1},2/{C^2}}
     \node[font=\scriptsize] at (\i*\dx,-1.3) {$\lab(\mathfrak U;-)$};
  \foreach \j/\lab in {0/{\Usm},1/{i\Omega^1},2/{i\Omega^2}}
     \node[font=\scriptsize,anchor=east] at (-1.15,\j*\dy) {$\lab$};
  % --- total-degree line: separates p+q=1 from p+q=2
  \draw[dashed,thick] (-0.20*\dx,1.70*\dy) -- (1.70*\dx,-0.20*\dy);
\end{tikzpicture}
\caption{Truncation changes the interpretation of closed total degree $1$ \v{C}ech--de Rham cochains.}
\label{fig:doublecomplex}
\end{figure}
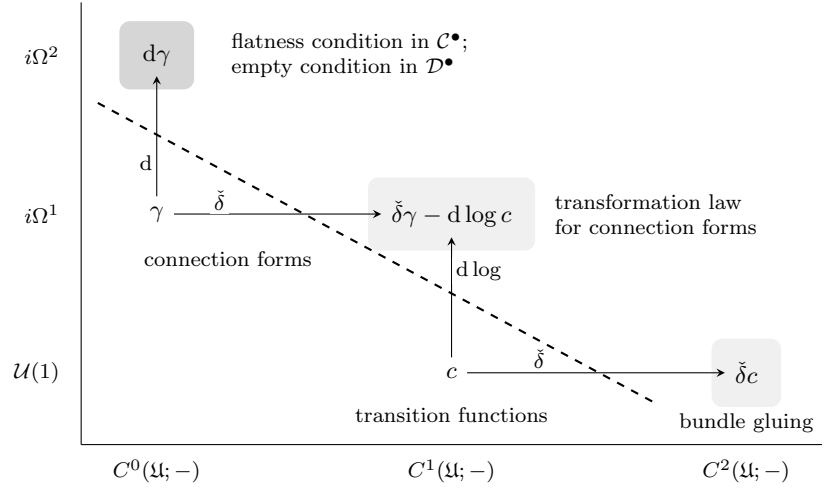

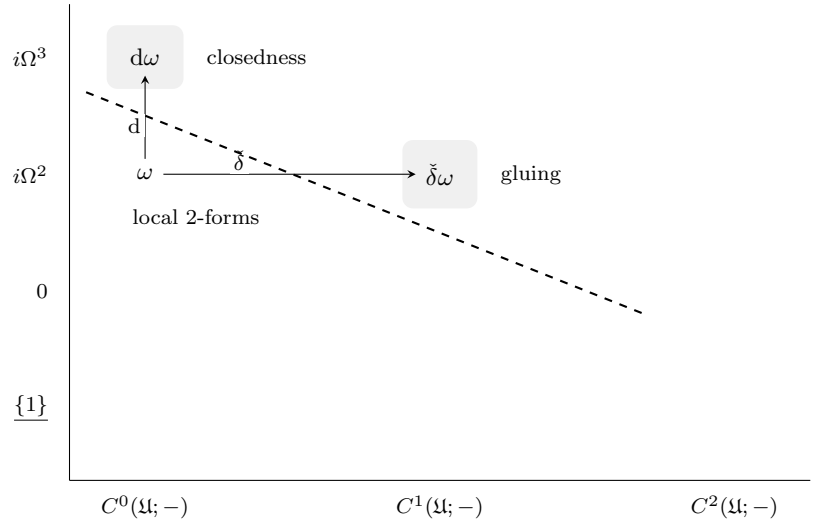
\begin{figure}[htbp]
\centering
\begin{tikzpicture}[>=stealth,font=\footnotesize]
  \def\dx{3.9}\def\dy{1.55}
  % --- cells
  \node (t03) at (0,3*\dy)   {$\dd\omega$};
  \node (t02) at (0,2*\dy)   {$\omega$};
  \node (t12) at (\dx,2*\dy) {$\check\delta\omega$};
  % --- differentials
  \draw[->] (t02) -- node[left,font=\scriptsize,fill=white,inner sep=1.5pt,pos=0.4]{$\dd$} (t03);
  \draw[->] (t02) -- node[above,font=\scriptsize,pos=0.30,fill=white,inner sep=1.5pt]{$\check\delta$} (t12);
  % --- shading of the total degree 3 components
  \begin{scope}[on background layer]
    \node[fill=black!6,rounded corners,fit=(t03),inner sep=5pt]{};
    \node[fill=black!6,rounded corners,fit=(t12),inner sep=5pt]{};
  \end{scope}
  % --- readings
  \node[font=\scriptsize,anchor=north west] at ($(t02.south)+(-0.30,-0.12)$) {local $2$-forms};
  \node[font=\scriptsize,anchor=west] at ($(t03.east)+(0.35,0)$) {closedness};
  \node[font=\scriptsize,anchor=west] at ($(t12.east)+(0.35,0)$) {gluing};
  % --- axes
  \draw[-] (-1.0,-0.95) -- (2*\dx+1.0,-0.95);
  \draw[-] (-1.0,-0.95) -- (-1.0,3*\dy+0.7);
  \foreach \i/\lab in {0/{C^0},1/{C^1},2/{C^2}}
     \node[font=\scriptsize] at (\i*\dx,-1.3) {$\lab(\mathfrak U;-)$};
  \foreach \j/\lab in {0/{\underline{\{1\}}},1/{0},2/{i\Omega^2},3/{i\Omega^3}}
     \node[font=\scriptsize,anchor=east] at (-1.15,\j*\dy) {$\lab$};
  % --- total-degree line: separates p+q=2 from p+q=3
  \draw[dashed,thick] (-0.20*\dx,2.70*\dy) -- (1.70*\dx,0.80*\dy);
\end{tikzpicture}
\caption{For $\mathcal T^\bullet$ the degree two cohomology of the total complex is the closed
$2$-forms.}
\label{fig:Tcomplex}
\end{figure}

We now realize \eqref{eq:kostantweil} as the long exact sequence of a short exact sequence of
complexes of sheaves. Consider the de Rham resolution of $\Uloc$,
\begin{equation}\label{eq:fullcomplex}
  (\mathcal{C}^\bullet, \dd^\bullet):= 0 \to \Usm \xrightarrow{\dd\log} i\Omega^1
  \xrightarrow{\dd} i\Omega^2 \xrightarrow{\dd} \cdots \xrightarrow{\dd} i\Omega^m \to 0,
\end{equation}
where $\Usm$ sits in degree $0$, $m=\dim_\R M$, and $i\Omega^p$ is the sheaf of imaginary valued
$p$-forms. We use its truncations
\begin{align*}
  (\mathcal{D}^\bullet, \dd_{\le 1}^\bullet)
    &:= 0 \to \Usm \xrightarrow{\dd\log} i\Omega^1 \to 0,\\
  (\mathcal{T}^\bullet, \dd_{\ge 2}^\bullet)
    &:= 0 \to \underline{\{1\}} \to 0 \to i\Omega^2 \xrightarrow{\dd} \cdots
       \xrightarrow{\dd} i\Omega^m \to 0.
\end{align*}

There is a short exact sequence of complexes of sheaves
\begin{equation} \label{eq:truncationses}
    0 \to \mathcal{T}^\bullet \to \mathcal{C}^\bullet \to \mathcal{D}^\bullet \to 0,
\end{equation}
whose long exact sequence in hypercohomology is
\[ \dots \to \HH^1(\mathcal{T}^\bullet) \to \HH^1(\mathcal{C}^\bullet) \to
   \HH^1(\mathcal{D}^\bullet) \to \HH^2(\mathcal{T}^\bullet) \to
   \HH^2(\mathcal{C}^\bullet) \to \dots \]

Then there is an identification of exact sequences
\begin{equation}\label{eq:kwidentification}
\begin{tikzcd}[row sep=3em, column sep=1.9em]
\HH^1(\mathcal{T}^\bullet) \ar[r] \ar[d] & \HH^1(\mathcal{C}^\bullet) \ar[r] \ar[d,"\circled{1}"] &  \HH^1(\mathcal{D}^\bullet) \ar[r, "\partial"] \ar[d,"\circled{2}"] &  \HH^2(\mathcal{T}^\bullet) \ar[r] \ar[d,"\circled{3}"] & \HH^2(\mathcal{C}^\bullet)\ar[d,"\circled{4}"]\\
0 \ar[r] &  \mathcal{P}_0(\U(1);\u(1)) \ar[r] & \mathcal{P}(\U(1);\u(1)) \ar[r, "\curv"'] & Z^2_{\dR}(M) \ar[r, "e^{i\int}"'] & H^2(M;\Uloc) \end{tikzcd}
\end{equation}
in which every square commutes and every vertical arrow is an isomorphism
(Figures~\ref{fig:doublecomplex} and~\ref{fig:Tcomplex}).

A cochain of the \v{C}ech--de Rham double complex with total degree $1$ is a pair $(\gamma,c)$
with $\gamma\in C^0(\mathfrak U;i\Omega^1)$ and $c\in C^1(\mathfrak U;\Usm)$: local connection
forms and transition functions. Its coboundary has one component in each of the three boxes of
total degree $2$, so closedness imposes one condition per box,
\begin{equation}\label{eq:circlelocal}
  \dd\gamma=0,
  \qquad
  \check\delta\gamma-\dd\log c=0,
  \qquad
  \check\delta c=1,
\end{equation}
or with indices,
\[
  \dd\gamma_j=0 \ \text{ on } U_j,
  \qquad
  \gamma_k-\gamma_j-\dd\log c_{jk}=0 \ \text{ on } U_{jk},
  \qquad
  c_{jk}\,c_{k\ell}\,c_{\ell j}=1 \ \text{ on } U_{jk\ell}.
\]
The last two say that $c$ glues a principal circle bundle and that the $\gamma$ assemble into a
connection on it. The first says that the connection is flat. For $\mathcal C^\bullet$ the first
condition is considered, so its degree-$1$ cocycles define circle bundles with flat connection,
giving \circled{1}. For $\mathcal D^\bullet$ the first condition is empty, so its degree-$1$
cocycles define circle bundles with \emph{arbitrary} connection. This is \circled{2}. The group
$\HH^2(\mathcal T^\bullet)$ is the closed $2$-forms, which is \circled{3}, and
$\HH^2(\mathcal C^\bullet)\cong H^2(M;\Uloc)$ by quasi-isomorphism, which is \circled{4}. The
connecting homomorphism sends $(\gamma,c)$ to $\check\dd(\gamma,c)=(\dd\gamma,0,1)$ and is
therefore the curvature map.

\subsection{The groupoid formulation of the classical lifting problem}
\label{subsec:liftings}

We recall the notions of groupoids, homotopy fibres, and exactness conditions for pointed
groupoids. We view a groupoid $\mathbf X$ as a category whose morphisms are
invertible, and a morphism of groupoids as a functor. The set of isomorphism classes of objects
of $\mathbf X$ is denoted by $\pi_0\mathbf X$, and the group of automorphisms of an object $x$
is denoted $\pi_1(\mathbf X,x)$. A groupoid is \textbf{discrete} if its only morphisms are
the identities. A functor $\mathbf X\to Y$ is understood as a functor to the discrete
groupoid associated to a set $Y$, necessarily factoring through
$\mathbf X\to\pi_0\mathbf X\to Y$.

Given a functor of groupoids $A:\mathbf X\to\mathbf Y$ and an object $y$ of $\mathbf Y$, the
\textbf{homotopy fibre} $\hofib{A}{y}$ is the groupoid whose objects are pairs
$(x,\hv y)$ with $x$ an object of $\mathbf X$ and
$\hv y:A(x)\xrightarrow{\ \cong\ }y$, and whose morphisms
$(x_1,\hv y_1)\to(x_2,\hv y_2)$ are the morphisms $\hv x:x_1\to x_2$ of $\mathbf X$ satisfying
$\hv y_2\circ A(\hv x)=\hv y_1$.

Let $\mathbf X\xrightarrow{A}\mathbf Y\xrightarrow{B}\mathbf Z$ be functors of groupoids, let
$z$ be an object of $\mathbf Z$ such that there exists a natural transformation $\nu$ from
$B\circ A$ to the constant functor valued at $z$. Then
$\nu_x:B(A(x))\xrightarrow{\ \cong\ }z$, so that $(A(x),\nu_x)$ is an object of
$\hofib{B}{z}$. There is therefore an induced functor
\[
  \mathbf X\longrightarrow\hofib{B}{z},
  \qquad
  x\longmapsto(Ax,\nu_x),
  \qquad
  \hv x\longmapsto A\hv x ,
\]
where $A\hv x$ is a morphism of $\hofib{B}{z}$ by naturality of $\nu$. We consider two
\textbf{exactness conditions}
for the triple $(A,B,z)$ \cite{Vitale,GrandisVitale,MantovaniMetereVitale}, depending on
whether we are viewing the image of objects up to isomorphism or the image of both of objects up
to isomorphism and of isomorphisms:
\begin{enumerate}
\item[$(\mathrm{E}1)$] There exists a natural transformation $\nu$ from $B\circ A$ to the
constant functor valued at $z$ such that the induced functor $\mathbf X\to\hofib{B}{z}$ is
essentially surjective.
\item[$(\mathrm{E}2)$] There exists a natural transformation $\nu$ from $B\circ A$ to the
constant functor valued at $z$ such that the induced functor $\mathbf X\to\hofib{B}{z}$ is full
and essentially surjective.
\end{enumerate}

The conditions depend only on the isomorphism class of $z$ in $\mathbf Z$. The induced functor
$\mathbf X\to\hofib{B}{z}$ depends on $\nu$, but the conditions do not, since they require only
that some $\nu$ exist. They descend to exactness
conditions for pointed sets and for groups
\cite[Ch.~V, 2.4, 2.4.1]{GabrielZisman}, \cite[Thm.~4.3]{Brown},
\cite[Prop.~3.1]{Vitale} for cat-groups,
\cite[Lem.~4.5, Prop.~3.6, Prop.~6.3]{MantovaniMetereVitale}.

\begin{proposition}\label{prop:piexact}
Let $\mathbf X\xrightarrow{A}\mathbf Y\xrightarrow{B}\mathbf Z$ be functors of groupoids and
$z$ be an object of $\mathbf Z$.
\begin{enumerate}[label=\textup{(\arabic*)}]
\item If $(\mathrm{E}1)$ holds, then for the sequence
\begin{equation}\label{eq:LEShomotopy}
  \pi_0\mathbf X\longrightarrow \pi_0\mathbf Y\longrightarrow\pi_0\mathbf Z
\end{equation}
the preimage of $[z]\in\pi_0\mathbf Z$ equals the image of $\pi_0\mathbf X$.

\item If $(\mathrm{E}2)$ holds and $x$ is an object of $\mathbf X$, then the six term sequence
\begin{equation}\label{eq:sixterm}
  \pi_1(\mathbf X,x)\to \pi_1(\mathbf Y,Ax)\to\pi_1(\mathbf Z,z)
  \to(\pi_0\mathbf X,[x])\to(\pi_0\mathbf Y,[Ax])\to(\pi_0\mathbf Z,[z])
\end{equation}
is exact as groups at $\pi_1(\mathbf Y,Ax)$ and exact as pointed sets
at $\pi_1(\mathbf Z,z)$, $(\pi_0\mathbf X,[x])$ and $(\pi_0\mathbf Y,[Ax])$.
\end{enumerate}
\end{proposition}

\begin{proof}
Assume $(\mathrm{E}1)$ holds. Then the induced functor $\mathbf X\to\hofib{B}{z}$ is
essentially surjective, so it descends to a surjection from $\pi_0\mathbf X$ to
$\pi_0\hofib{B}{z}$. Forgetting the second component maps $\pi_0\hofib{B}{z}$ onto the preimage
of $[z]$ in $\pi_0\mathbf Y$. Composing the two gives \textup{(1)}.

Assume $(\mathrm{E}2)$ holds. Then by \textup{(1)} the sequence is exact at
$(\pi_0\mathbf Y,[Ax])$. The induced functor $\mathbf X\to\hofib{B}{z}$ descends to a bijection
on isomorphism classes, so we can replace the fourth term by
$\big(\pi_0\hofib{B}{z},[(Ax,\nu_x)]\big)$. The map to $(\pi_0\mathbf Y,[Ax])$ sends
$[(y,B(y)\to z)]$ to the isomorphism class $[y]$ of the first component.

Since morphisms in the homotopy fibre precompose the second component, postcomposing the second
component of an object $(y,B(y)\to z)$ by an automorphism of $z$ may give nonisomorphic objects
of $\hofib{B}{z}$. The postcomposition gives a connecting map
$\pi_1(\mathbf Z,z)\to\pi_0\hofib{B}{z}$. An isomorphism class $[(y,B(y)\to z)]$ in the preimage
of $[Ax]$ has a representative
$\big(Ax,\ B(Ax)\xrightarrow{\ \nu_x\ }z\xrightarrow{\ \hv z\ }z\big)$ for an automorphism
$\hv z$ of $z$. So the sequence is exact at $(\pi_0\mathbf X,[x])$.

An automorphism $\hv z$ of $z$ is in the preimage of $[(Ax,\nu_x)]$ if there exists an
automorphism $\hv y$ of $Ax$ such that the composition
$B(Ax)\xrightarrow{\ B(\hv y)^{-1}\ }B(Ax)\xrightarrow{\ \nu_x\ }z\xrightarrow{\ \hv z\ }z$
equals $\nu_x$. This condition is equivalent to
$\hv z=\nu_x\circ B(\hv y)\circ\nu_x^{-1}$, that is, to $\hv z$ being the image of $\hv y$ under
$\pi_1(\mathbf Y,Ax)\to\pi_1(\mathbf Z,z)$. So the sequence is exact at $\pi_1(\mathbf Z,z)$.

An automorphism $\hv y$ of $Ax$ is in the preimage of the identity automorphism of $z$ if
$\nu_x\circ B(\hv y)\circ\nu_x^{-1}=\id_z$. Then $\hv y$ is an automorphism of the object
$(Ax,\nu_x)$ in the homotopy fibre $\hofib{B}{z}$, and by fullness it is the image of an
automorphism of $x$, so the sequence is exact at $\pi_1(\mathbf Y,Ax)$.
\end{proof}

Let $M$ be a smooth manifold, and let $G$ be a Lie group. Then let $\mathbf P(G)$ be the
groupoid whose objects are principal $G$-bundles over $M$ and whose morphisms are isomorphisms
of principal $G$-bundles.

\begin{definition}[Liftings]\label{def:lifting}
Let
\[ 1 \to C \to \tilde{G} \xrightarrow{\ \sigma\ } G \to 1 \]
be a central extension of Lie groups, and let $\rho:P\to M$ be a principal $G$-bundle. A
\textbf{$\sigma$-lifting} of $P$ is a pair $(\tilde P,\bu)$ consisting of a principal
$\tilde G$-bundle $\tilde\rho:\tilde P\to M$ with right action $\tilde R$ and a principal
bundle morphism $\bu:\tilde P\to P$ over $M$ equivariant along $\sigma$, in the sense that the
squares
\[ \begin{tikzcd}[row sep=2.8em, column sep=3.4em] \tilde P \ar[d, "\bu"] \ar[dd, bend right = 50, "\tilde\rho"'] & \tilde P \times \tilde G \ar[l, "\tilde R"'] \ar[d, "{(\bu,\sigma)}"]  \\ P \ar[d, "\rho"'] & P \times G \ar[l, "R"']\\  M& \end{tikzcd} \]
commute. Alternatively, a $\sigma$-lifting of $P$ is a principal $\tilde G$-bundle
$\tilde P\to M$ together with a choice of isomorphism $\hv u:\tilde P/C\cong P$ of principal
$G$-bundles, obtained by descending $\bu$.

Two $\sigma$-liftings $(\tilde P_1,\bu_1)$ and $(\tilde P_2,\bu_2)$ of the same $P$ are
\emph{isomorphic} if there is an isomorphism of principal $\tilde G$-bundles that covers the
\emph{identity automorphism} of $P$, or alternatively, such that the descended isomorphism
$\tilde P_1/C\xrightarrow{\ \cong\ }\tilde P_2/C$ agrees with the given isomorphisms
$\tilde P_1/C\xrightarrow{\ \hv u_1\ }P\xrightarrow{\ \hv u_2^{-1}\ }\tilde P_2/C$.
\end{definition}

The liftings of $P$ form a groupoid $\mathbf L(P,\sigma)$, that can be identified with the
homotopy fibres of the quotient-by-central-subgroup functor
$\,\cdot\,/C:\mathbf P(\tilde G)\to\mathbf P(G)$:
\begin{equation}\label{eq:liftinghofib}
  \mathbf L(P,\sigma)\;\cong\;\hofib{(\,\cdot\,/C)}{P}.
\end{equation}
For each nonempty homotopy fibre $\hofib{(\,\cdot\,/C)}{P}$, a choice of $\sigma$-lifting
determines an equivalence of groupoids $\mathbf P(C)\simeq\mathbf L(P,\sigma)$. Since
$\pi_0\mathbf P(C)\cong H^1(M;\mathcal C)$ is an abelian group, this equivalence makes
$\pi_0\mathbf L(P,\sigma)$ a torsor under $H^1(M;\mathcal C)$.

In classical obstruction theory \cite{Frenkel,HessKrausser,Brylinski} the lifting problem is
solved through the long exact sequence of pointed sets in sheaf cohomology
\begin{equation}\label{eq:classicalseq}
\begin{tikzcd}[row sep=3em, column sep=2.2em]
H^1(M;\mathcal C) \ar[r] \ar[d,"\cong"'] &
H^1(M;\tilde{\mathcal G}) \ar[r] \ar[d,"\cong"'] &
H^1(M;\mathcal G) \ar[r,"\delta"] \ar[d,"\cong"'] &
H^2(M;\mathcal C) \ar[d,"="] \\
\pi_0\mathbf P(C) \ar[r] &
\pi_0\mathbf P(\tilde G) \ar[r] &
\pi_0\mathbf P(G) \ar[r,"w"'] &
H^2(M;\mathcal C)
\end{tikzcd}
\end{equation}
associated to the short exact sequence of sheaves of group-valued functions
$1\to\mathcal C\to\tilde{\mathcal G}\to\mathcal G\to1$. We denote the connecting
homomorphism $\delta$ by
\[ w(\,\cdot\,,\sigma):H^1(M;\mathcal G)\longrightarrow H^2(M;\mathcal C), \]
so that $w(P,\sigma)$ is the obstruction class of
the lifting problem for $P$. Then $\mathbf L(P,\sigma)$ is nonempty if and only if
$w(P,\sigma)$ is trivial in $H^2(M;\mathcal C)$, and $H^1(M;\mathcal C)$ acts freely and
transitively on $\pi_0\mathbf L(P,\sigma)$.

At the level of groupoids there is a sequence of functors
\begin{equation}\label{eq:groupoidseq}
  \mathbf P(C)\xrightarrow{\ (\,\cdot\,)\times_C\tilde G\ }\mathbf P(\tilde G)
  \xrightarrow{\ \cdot\,/C\ }\mathbf P(G)
  \xrightarrow{\ w\ }H^2(M;\mathcal C).
\end{equation}
Because the snaking construction is not well defined on morphisms, we view
$H^2(M;\mathcal C)$ as a pointed discrete groupoid, whose objects form the abelian group
$H^2(M;\mathcal C)$. The isomorphism classes
$\pi_0\mathbf P(C)\cong H^1(M;\mathcal C)$ form an abelian group. The leftmost functor is
extension of structure group along $C\hookrightarrow\tilde G$, and lifts the map $H^1(M;\mathcal C)\to H^1(M;\tilde{\mathcal G})$
of \eqref{eq:classicalseq}. Since $C$ acts trivially on $G=\tilde G/C$ the composite
$\big(P_C\times_C\tilde G\big)/C\cong(P_C/C)\times G\cong M\times G$ is the trivial principal
$G$-bundle. This isomorphism provides the natural transformation $\nu$ at the trivial bundle
in $\mathbf P(G)$.

At $\mathbf P(G)$, the sequence \eqref{eq:groupoidseq} satisfies the exactness condition
$(\mathrm{E}1)$ over the identity element of $H^2(M;\mathcal C)$. It further satisfies the
condition that the homotopy fibre of $\,\cdot\,/C$ over $P$ is empty whenever $w(P,\sigma)$
is not equal to the identity element in $H^2(M;\mathcal C)$. At $\mathbf P(\tilde G)$, the
sequence \eqref{eq:groupoidseq} satisfies the exactness condition $(\mathrm{E}2)$ over the
trivial principal $G$-bundle. It further satisfies the property that every nonempty homotopy
fibre of $\,\cdot\,/C$ is equivalent as a groupoid to $\mathbf P(C)$. These recover the
classical obstruction and classification results on isomorphism classes.

The induced action on the isomorphism classes of the homotopy fibre differs from the action on
the fibre of the map on isomorphism classes. If the choice of identification $\hv u$ is
forgotten before defining the action, the induced isomorphisms of $\tilde P$ may descend to
nontrivial automorphisms of $P$,
resulting in an action that is transitive, but not necessarily free. This may be read from the
exact sequence
\[ \pi_1\big(\mathbf P(G),P\big)\to\pi_0\mathbf P(C)\to\pi_0\mathbf P(\tilde G)
   \to\pi_0\mathbf P(G) \]
of Proposition~\ref{prop:piexact}\textup{(2)}.

\subsection{The first hypercohomology groupoid}
\label{subsec:hypergroupoid}

We define the \emph{first hypercohomology groupoid} for a complex of sheaves of groups,
possibly nonabelian in degree zero, to complete the diagrams
\eqref{eq:kwidentification}, \eqref{eq:classicalseq} and \eqref{eq:groupoidseq}, and extend the
solution of the lifting problem to the lifting problem with prescribed central curvature.

Let $(\mathcal F^\bullet,\dd^\bullet)$ be a complex of sheaves such that $\mathcal F^0$ is a sheaf
of groups (written multiplicatively), $\mathcal F^1$ is a sheaf of abelian groups (written
additively), and the degree-zero differential $\dd^0:\mathcal F^0\to\mathcal F^1$ is a morphism
of sheaves of groups. Given a cover $\mathfrak U$, let
\[ \check{C}^0(\mathfrak{U}; \mathcal F^1) \oplus \check{C}^1(\mathfrak{U}; \mathcal F^0)
   \xrightarrow{\ D^1\ } \check{C}^0(\mathfrak{U}; \mathcal F^2)
   \oplus \check{C}^1(\mathfrak{U}; \mathcal F^1)
   \oplus \check{C}^2(\mathfrak{U}; \mathcal F^0),\]
\begin{equation}\label{eq:D1}
  D^1(\gamma, g)
  := \big(\dd^1\gamma,\ \check\delta\gamma-\dd^0 g,\ \check\delta g\big)
  = \Big(\big\{\dd^1\gamma_j\big\},\
    \big\{\gamma_k-\gamma_j-\dd^0(g_{jk})\big\},\
    \big\{g_{jk}\,g_{k\ell}\,g_{j\ell}^{-1}\big\}\Big),
\end{equation}
where the three families are indexed by $j$, by $jk$ and by $jk\ell$, and each term is
restricted to the intersection labelling it.

For $t \in \check{C}^0(\mathfrak{U}; \mathcal F^0)$, let
\begin{equation}\label{eq:D0}
  D^0\big(t, (\gamma, g)\big)
  := \Big(\big\{\gamma_j-\dd^0(t_j)\big\},\
     \big\{t_j\,g_{jk}\,t_k^{-1}\big\}\Big),
\end{equation}
with the same convention on restrictions. The map $D^0$ defines an action of $\check C^0(\mathfrak U;\mathcal F^0)$ on
$\check C^0(\mathfrak U;\mathcal F^1)\oplus\check C^1(\mathfrak U;\mathcal F^0)$. It preserves
the fibre $(D^1)^{-1}(0,0,1)$ because $\dd^1\dd^0=0$,
$\dd^0(t_j\,g_{jk}\,t_k^{-1})=\dd^0(g_{jk})+\dd^0(t_j)-\dd^0(t_k)$, and
\[ t_j\,g_{jk}\,t_k^{-1}\cdot t_k\,g_{k\ell}\,t_\ell^{-1}\cdot
   \big(t_j\,g_{j\ell}\,t_\ell^{-1}\big)^{-1}
   \;=\; t_j\,(\check\delta g)_{jk\ell}\,t_j^{-1} . \]
So the orbit set $(D^1)^{-1}(0,0,1)/\check C^0(\mathfrak U;\mathcal F^0)$ is well defined.

Let $\mathbf{H}^1_{\mathfrak U}(\mathcal F^\bullet)$ be the (small) action groupoid for this action. The
objects are the cocycles $(\gamma,g)\in(D^1)^{-1}(0,0,1)$, and the morphisms
$(\gamma,g)\to(\gamma',g')$ are $t\in\check C^0(\mathfrak U;\mathcal F^0)$ with
$D^0(t,(\gamma,g))=(\gamma',g')$.

To make this definition independent of the choice of cover $\mathfrak U$, we take the
pseudo-colimit in the $2$-category $\mathbf{Gpd}$ of small groupoids. Pseudo-colimits exist in
$\mathbf{Gpd}$ \cite[Prop.~3.1, (5.8), Prop.~5.2]{Kelly}, and coincide with \textbf{homotopy
colimits} \cite[\S6.1]{Gambino} in the homotopy $1$-category $\mathrm{Ho}(\mathbf{Gpd})$,
defined by the model structure whose weak equivalences are the categorical equivalences and
whose fibrations are the Grothendieck fibrations.

In a $2$-category we denote the \textbf{vertical composition} of $2$-morphisms by $\circ$, and
the \textbf{horizontal composition} of $2$-morphisms, as well as the horizontal composition of
a $2$-morphism with a $1$-morphism, by $\ast$.

Let $\operatorname{Cov}(M)$ be the small $2$-category of open covers of $M$ indexed by subsets
of its topology, with $1$-morphisms as refinement maps between indexed covers, and with a
single, invertible $2$-morphism between any two parallel refinement maps. A refinement map from $\mathfrak V=\{V_j\}_{j\in J}$ to
$\mathfrak U=\{U_k\}_{k\in K}$ is given by $r:J\to K$ with $V_j\subseteq U_{j_r}$, where
$j_r:=r(j)$, and we also denote it by $r:\mathfrak V\to\mathfrak U$. Let $\mathbf{Gpd}$ be the $2$-category of small groupoids.

Restriction along a refinement map $r$ commutes with $D^1$ and $D^0$, so
\[
  r^\ast(\gamma,g)
  :=\Big(\gamma_{j_r}\big|_{V_j},\,
         g_{j_r\,k_r}\big|_{V_{jk}}\Big),
  \qquad
  (r^\ast t)_j:=t_{j_r}\big|_{V_j}
\]
defines a functor ($1$-morphism)
$r^\ast:\mathbf H^1_{\mathfrak U}(\mathcal F^\bullet)\to\mathbf H^1_{\mathfrak V}(\mathcal F^\bullet)$.
Restriction along all refinement maps defines a strict $2$-functor
$\Phi:\operatorname{Cov}(M)^{\mathrm{op}}\to\mathbf{Gpd}$,
\[
  \Phi(\mathfrak U):=\mathbf H^1_{\mathfrak U}(\mathcal F^\bullet),
  \qquad \Phi(r):=r^\ast,
\]
and for parallel refinement maps $r,r':\mathfrak V\to\mathfrak U$, $\Phi$ sends the
$2$-morphism $r'\Rightarrow r$ to the natural isomorphism between restriction functors
\begin{equation}\label{eq:nudef}
  \nu_{r,r'}:(r')^\ast\Longrightarrow r^\ast ,
\end{equation}
where for each object $(\gamma,g)$ of $\mathbf H^1_{\mathfrak U}(\mathcal F^\bullet)$ the
morphism $\nu_{r,r'}(\gamma,g):(r')^\ast(\gamma,g)\to r^\ast(\gamma,g)$ is given by
\[
  \Big\{g_{j_r\,j_{r'}}\big|_{V_j}\Big\}_{j\in J}
  \in\check C^0(\mathfrak V;\mathcal F^0),
\]
over each $V_j\subseteq U_{j_r}\cap U_{j_{r'}}$. The natural isomorphisms $\nu$ satisfy the identities
\[
  \nu_{r,r}=\id,
  \qquad
  \nu_{r,r'}\circ\nu_{r',r''}=\nu_{r,r''},
  \qquad
  \nu_{r r_1,\,r' r_1}=r_1^\ast\ast\nu_{r,r'},
  \quad
  \nu_{r_2 r,\,r_2 r'}=\nu_{r,r'}\ast r_2^\ast ,
\]
for refinement maps $r_1$ and $r_2$ composable with $r$ and $r'$, because their components are
values of $g$ and
$g_{j_r\,j_{r'}}g_{j_{r'}\,j_{r''}}=g_{j_r\,j_{r''}}$.

\begin{lemma}\label{lem:refinementequiv}
Let $(\mathcal F^\bullet,\dd^\bullet)$ be a complex of sheaves such that $\mathcal F^0$ is the
sheaf of smooth functions with values in a Lie group $G$, $\mathcal F^1$ is a sheaf of abelian
groups, and $\dd^0:\mathcal F^0\to\mathcal F^1$ is a morphism of sheaves of groups. Let
$\mathfrak U$ and $\mathfrak V$ be good covers of $M$ where $\mathfrak V$ refines
$\mathfrak U$. Let $r:\mathfrak V\to\mathfrak U$ be a refinement map. Then
\[
  r^\ast:\mathbf H^1_{\mathfrak U}(\mathcal F^\bullet)\longrightarrow
  \mathbf H^1_{\mathfrak V}(\mathcal F^\bullet)
\]
is an equivalence of groupoids.
\end{lemma}

\begin{proof}
Let $(\gamma,g)$ be an object of $\mathbf H^1_{\mathfrak V}(\mathcal F^\bullet)$, and let $P$
be the principal $G$-bundle defined by $g$. Since $\mathfrak U$ and $\mathfrak V$ are good
covers, there exist local sections $s_j$ over $V_j$ and $\tilde s_k$ over $U_k$.

On $U_k\cap U_\ell$ the choice of local sections $\{\tilde s_k\}$ determines
$\tilde g_{k\ell}$ by $\tilde s_\ell=\tilde s_k\tilde g_{k\ell}$. On $U_k\cap V_j$, the choice
of sections $\{\tilde s_k\}$, $\{s_j\}$ determines $h_{kj}$ by $\tilde s_k=s_jh_{kj}$.

The identities $g_{j\ell}=h_{kj}h_{k\ell}^{-1}$ and $\gamma_\ell-\gamma_j=\dd^0(g_{j\ell})$ on
$U_k\cap V_j\cap V_\ell$ imply that the sections $\gamma_j+\dd^0(h_{kj})$ on $U_k\cap V_j$
agree on overlaps. So they
glue to a section $\tilde\gamma_k\in\mathcal F^1(U_k)$. Then $(\tilde\gamma,\tilde g)$ is an
object of $\mathbf H^1_{\mathfrak U}(\mathcal F^\bullet)$, as
$D^1(\tilde\gamma,\tilde g)=(0,0,1)$ by construction. On $V_j\subseteq U_{j_r}$, the
$0$-cochain $\{h_{j_rj}^{-1}\}$ gives a morphism
$(\gamma,g)\to r^\ast(\tilde\gamma,\tilde g)$ in
$\mathbf H^1_{\mathfrak V}(\mathcal F^\bullet)$. Thus $r^\ast$ is essentially surjective.

For fullness, let $(\tilde\gamma,\tilde g)$ and $(\tilde\gamma',\tilde g')$ be objects of
$\mathbf H^1_{\mathfrak U}(\mathcal F^\bullet)$ and let
$t:r^\ast(\tilde\gamma,\tilde g)\to r^\ast(\tilde\gamma',\tilde g')$ be a morphism in
$\mathbf H^1_{\mathfrak V}(\mathcal F^\bullet)$. Let $P$ and $P'$ be the principal $G$-bundles
defined by $\tilde g$ and $\tilde g'$, with local sections $\tilde s_k$ and $\tilde s'_k$ over
$U_k$, so that $\tilde s_m=\tilde s_k\tilde g_{km}$ and
$\tilde s'_m=\tilde s'_k\tilde g'_{km}$ on $U_k\cap U_m$. On $V_j$ write
$s_j:=\tilde s_{j_r}|_{V_j}$ and $s'_j:=\tilde s'_{j_r}|_{V_j}$, whose transition functions on
$V_j\cap V_\ell$ are $g_{j\ell}=(r^\ast\tilde g)_{j\ell}$ and
$g'_{j\ell}=(r^\ast\tilde g')_{j\ell}$. On $U_k\cap V_j$ let $h_{kj}$ and $h'_{kj}$ be defined
by $\tilde s_k=s_jh_{kj}$ and $\tilde s'_k=s'_jh'_{kj}$.

On $U_k\cap V_j\cap V_\ell$ we have $h_{kj}=g_{j\ell}h_{k\ell}$ and
$g'_{j\ell}=h'_{kj}(h'_{k\ell})^{-1}$, so with $t_jg_{j\ell}=g'_{j\ell}t_\ell$,
\[
  (h'_{kj})^{-1}t_jh_{kj}
  =(h'_{kj})^{-1}t_jg_{j\ell}h_{k\ell}
  =(h'_{kj})^{-1}g'_{j\ell}t_\ell h_{k\ell}
  =(h'_{k\ell})^{-1}t_\ell h_{k\ell}.
\]
Hence the sections $(h'_{kj})^{-1}t_jh_{kj}$ on $U_k\cap V_j$ agree on overlaps and glue to
$\tilde t_k\in\mathcal F^0(U_k)$. Then
$\tilde g'_{km}=\tilde t_k\tilde g_{km}\tilde t_m^{-1}$, and $\tilde t_k|_{V_j}=t_j$ whenever
$j_r=k$.

Moreover, on $U_k\cap V_j$, the identities $\tilde\gamma_k=\gamma_j+\dd^0(h_{kj})$,
$\tilde\gamma'_k=\gamma'_j+\dd^0(h'_{kj})$ and $\gamma'_j=\gamma_j-\dd^0(t_j)$ imply
\[
  \tilde\gamma_k-\dd^0(\tilde t_k)
  =\gamma_j+\dd^0(h_{kj})-\big(\dd^0(h_{kj})+\dd^0(t_j)-\dd^0(h'_{kj})\big)
  =\gamma'_j+\dd^0(h'_{kj})
  =\tilde\gamma'_k,
\]
so $\tilde t=\{\tilde t_k\}$ is a morphism
$(\tilde\gamma,\tilde g)\to(\tilde\gamma',\tilde g')$ such that $r^\ast\tilde t=t$. This proves
fullness.

Finally, a morphism over $\mathfrak U$ determines a bundle map. If two such morphisms have the
same restriction under $r^\ast$, their bundle maps agree over every $V_j$, hence over $M$.
Expressing that bundle map using the sections $\tilde s_k,\tilde s'_k$ shows that the original morphisms
agree. Thus $r^\ast$ is faithful.
\end{proof}

A \textbf{pseudo-cocone} of $\Phi$ \cite[\S7.6]{Borceux}, \cite[\S5]{Kelly} is a small
groupoid $\mathbf X$ together with functors
$\Edge_{\mathbf X}(\mathfrak U):\Phi(\mathfrak U)\to\mathbf X$ and natural isomorphisms
$\Face_{\mathbf X}(r):\Edge_{\mathbf X}(\mathfrak V)\circ\Phi(r)\Rightarrow
\Edge_{\mathbf X}(\mathfrak U)$, one for each refinement map $r:\mathfrak V\to\mathfrak U$,
such that
\begin{enumerate}[label=\textup{(\arabic*)}]
\item $\Face_{\mathbf X}(\id_{\mathfrak U})=\id_{\Edge_{\mathbf X}(\mathfrak U)}$;
\item for composable refinement maps
$\mathfrak W\xrightarrow{\ r_2\ }\mathfrak V\xrightarrow{\ r_1\ }\mathfrak U$,
\[
  \Face_{\mathbf X}(r_1)\circ\big(\Face_{\mathbf X}(r_2)\ast\Phi(r_1)\big)
  =\Face_{\mathbf X}(r_1r_2);
\]
\item for parallel refinement maps $r,r':\mathfrak V\to\mathfrak U$,
\[
  \Face_{\mathbf X}(r')=\Face_{\mathbf X}(r)\circ
  \big(\Edge_{\mathbf X}(\mathfrak V)\ast\Phi(r'\Rightarrow r)\big).
\]
\end{enumerate}
The pseudo-cocones of $\Phi$ with vertex $\mathbf X$
form a groupoid $\PsCocone(\mathbf X)$, where morphisms are natural transformations
$\Edge_{\mathbf X}(\mathfrak U)\Rightarrow\Edge'_{\mathbf X}(\mathfrak U)$ compatible with
$\Face_{\mathbf X}(r)$. Equivalently, a pseudo-cocone with vertex $\mathbf X$ is a
pseudo-natural transformation from $\Phi$ to the constant $2$-functor at $\mathbf X$.

A \textbf{pseudo-colimit} of $\Phi$ is a pseudo-cocone
$\big(\mathbf C,\Edge_{\mathbf C},\Face_{\mathbf C}\big)$ such that
\[
  \mathbf{Gpd}(\mathbf C,\mathbf X)\longrightarrow\PsCocone(\mathbf X),
  \qquad
  A\longmapsto\big(\mathbf X,\ A\circ\Edge_{\mathbf C},\ A\ast\Face_{\mathbf C}\big),
\]
is an isomorphism of groupoids for every small groupoid $\mathbf X$
\cite[Def.~7.6.1]{Borceux}, \cite[(5.4)]{Kelly}, \cite[Def.~4.2.13]{Pirashvili}. This isomorphism of
groupoids encodes the universal property, and the pseudo-colimit is unique up to isomorphism of
groupoids \cite[Prop.~7.6.2]{Borceux}, while the homotopy colimit is unique up to equivalence
of groupoids.

For two refinement maps $r,r'$ an ordinary colimit requires $r^\ast(\gamma,g)$ and
$(r')^\ast(\gamma,g)$ to have equal images, and views the natural isomorphisms
$\nu_{r,r'}$ between common refinements as identity morphisms of the colimit groupoid
$\mathbf C$. The pseudo-colimit keeps these natural isomorphisms as morphisms of the
pseudo-colimit groupoid $\mathbf C$.

\begin{definition}\label{def:HHgroupoid}
For the complex $\mathcal F^\bullet$ above, the first hypercohomology
groupoid is
\[
  \mathbf H^1(\mathcal F^\bullet)
  :=\operatorname*{ps\,colim}_{\mathfrak U\in\operatorname{Cov}(M)^{\mathrm{op}}}
    \mathbf H^1_{\mathfrak U}(\mathcal F^\bullet),
\]
where the pseudo-colimit is taken in $\mathbf{Gpd}$.
\end{definition}

\begin{proposition}\label{prop:goodcover}
Let $\mathcal F^\bullet$ satisfy the conditions of Lemma~\ref{lem:refinementequiv}, and let
$\mathfrak U$ be a good cover of $M$. Then
\[
  \Edge_{\mathbf H^1(\mathcal F^\bullet)}(\mathfrak U):
  \mathbf H^1_{\mathfrak U}(\mathcal F^\bullet)\longrightarrow\mathbf H^1(\mathcal F^\bullet)
\]
is an equivalence of groupoids.
\end{proposition}

\begin{proof}
To prove that the canonical edge functor $\Edge_{\mathbf H^1(\mathcal F^\bullet)}(\mathfrak U)$
is an equivalence, we construct a pseudo-cocone with vertex
$\mathbf H^1_{\mathfrak U}(\mathcal F^\bullet)$. By the universal property of the
pseudo-colimit, this induces a functor back from the pseudo-colimit to
$\mathbf H^1_{\mathfrak U}(\mathcal F^\bullet)$, which we show is its quasi-inverse.

To streamline the notation, we denote
$\mathbf H_{\mathfrak A}:=\mathbf H^1_{\mathfrak A}(\mathcal F^\bullet)$ for any cover
$\mathfrak A$, and the pseudo-colimit by $\mathbf C:=\mathbf H^1(\mathcal F^\bullet)$.
We construct a pseudo-cocone
$(\mathbf H_{\mathfrak U},\Edge_{\mathfrak U},\Face_{\mathfrak U})$.

\paragraph{Step 1: the faithfulness principle.}
The difficulty in constructing this pseudo-cocone is defining natural transformations between
composites of functors involving quasi-inverses. We resolve this by a lifting principle.

By Lemma~\ref{lem:refinementequiv}, if $\mathfrak Z$ is a good cover refining $\mathfrak U$,
the restriction functor
$r_{\mathfrak Z\mathfrak U}^\ast:\mathbf H_{\mathfrak U}\to\mathbf H_{\mathfrak Z}$ is an
equivalence of groupoids, in particular fully faithful. Consequently, for any two functors
$F_1,F_2:\mathbf X\to\mathbf H_{\mathfrak U}$, postcomposition with
$r_{\mathfrak Z\mathfrak U}^\ast$ induces a bijection on natural transformations:
\begin{equation}\label{eq:faithfulness}
  \operatorname{Nat}(F_1,F_2)\xrightarrow{\ \cong\ }
  \operatorname{Nat}\big(r_{\mathfrak Z\mathfrak U}^\ast\circ F_1,\,
  r_{\mathfrak Z\mathfrak U}^\ast\circ F_2\big).
\end{equation}
To define natural transformations targeting $\mathbf H_{\mathfrak U}$, or to verify equalities
between them, it is therefore necessary and sufficient to do so on a sufficiently fine good
cover $\mathfrak Z$.

\paragraph{Step 2: the edge functors.}
For each cover $\mathfrak W$ we define a functor
$\Edge_{\mathfrak U}(\mathfrak W):\mathbf H_{\mathfrak W}\to\mathbf H_{\mathfrak U}$. Choose a
good cover $\mathfrak V$ refining both $\mathfrak W$ and $\mathfrak U$, with refinement maps
$r_{\mathfrak V\mathfrak W}:\mathfrak V\to\mathfrak W$ and
$r_{\mathfrak V\mathfrak U}:\mathfrak V\to\mathfrak U$. Because $\mathfrak V$ and $\mathfrak U$
are good covers, $r_{\mathfrak V\mathfrak U}^\ast$ is an equivalence. Choose a quasi-inverse
$(r_{\mathfrak V\mathfrak U}^\ast)^{-1}$ and let
\[
  \Edge_{\mathfrak U}(\mathfrak W):=
  (r_{\mathfrak V\mathfrak U}^\ast)^{-1}\circ r_{\mathfrak V\mathfrak W}^\ast .
\]
For $\mathfrak W=\mathfrak U$ we choose the identity refinement $\mathfrak V=\mathfrak U$, so
that $\Edge_{\mathfrak U}(\mathfrak U)=\id_{\mathbf H_{\mathfrak U}}$.

\paragraph{Step 3: the faces.}
For a refinement map $r:\mathfrak W_2\to\mathfrak W_1$ we need a natural isomorphism
\[
  \Face_{\mathfrak U}(r):\Edge_{\mathfrak U}(\mathfrak W_2)\circ r^\ast
  \Longrightarrow\Edge_{\mathfrak U}(\mathfrak W_1).
\]
Let $\mathfrak V_1$ and $\mathfrak V_2$ be the good covers chosen in Step 2 for $\mathfrak W_1$
and $\mathfrak W_2$, and choose a good cover $\mathfrak Y$ refining $\mathfrak V_1$,
$\mathfrak V_2$ and $\mathfrak U$. We apply \eqref{eq:faithfulness} with
$r_{\mathfrak Y\mathfrak U}^\ast$.

Because $(r_{\mathfrak V_i\mathfrak U}^\ast)^{-1}$ is a quasi-inverse,
$r_{\mathfrak Y\mathfrak U}^\ast\circ(r_{\mathfrak V_i\mathfrak U}^\ast)^{-1}$ is naturally
isomorphic to the restriction $r_{\mathfrak Y\mathfrak V_i}^\ast$. On $\mathfrak Y$ the two
composites to be compared therefore simplify:
\begin{align*}
  r_{\mathfrak Y\mathfrak U}^\ast\circ
  \big(\Edge_{\mathfrak U}(\mathfrak W_2)\circ r^\ast\big)
  &\cong r_{\mathfrak Y\mathfrak V_2}^\ast\circ r_{\mathfrak V_2\mathfrak W_2}^\ast\circ
  r^\ast\cong r_{\mathfrak Y\mathfrak W_2}^\ast\circ r^\ast,\\
  r_{\mathfrak Y\mathfrak U}^\ast\circ\Edge_{\mathfrak U}(\mathfrak W_1)
  &\cong r_{\mathfrak Y\mathfrak V_1}^\ast\circ r_{\mathfrak V_1\mathfrak W_1}^\ast
  \cong r_{\mathfrak Y\mathfrak W_1}^\ast .
\end{align*}
Both $r_{\mathfrak Y\mathfrak W_2}^\ast\circ r^\ast$ and $r_{\mathfrak Y\mathfrak W_1}^\ast$
are restrictions along refinement maps $\mathfrak Y\to\mathfrak W_1$, so by the definition of
$\Phi$ on $2$-morphisms there is a canonical natural isomorphism $\nu_{\mathfrak Y}$ of
\eqref{eq:nudef} between them. By the bijection \eqref{eq:faithfulness} there is a unique
natural transformation targeting $\mathbf H_{\mathfrak U}$ corresponding, under the
isomorphisms displayed above, to $\nu_{\mathfrak Y}$, and we define it to be
$\Face_{\mathfrak U}(r)$. Since restriction along a refinement map between good covers is
faithful, this definition does not depend on the choice of $\mathfrak Y$.

\paragraph{Step 4: the pseudo-cocone conditions.}
By \eqref{eq:faithfulness}, two natural transformations targeting $\mathbf H_{\mathfrak U}$
are equal if and only if they are equal after restriction to a common good cover
$\mathfrak Z$. On $\mathfrak Z$ the quasi-inverses cancel up to natural isomorphism and the
faces correspond to the comparisons $\nu$, so the three conditions for $\Face_{\mathfrak U}$
reduce to the three identities for $\nu$ stated below \eqref{eq:nudef}:
\begin{enumerate}[label=\textup{(\arabic*)}]
\item $\nu_{r,r}=\id$ gives
$\Face_{\mathfrak U}(\id_{\mathfrak W})=\id_{\Edge_{\mathfrak U}(\mathfrak W)}$;
\item $\nu_{r,r'}\circ\nu_{r',r''}=\nu_{r,r''}$ gives
$\Face_{\mathfrak U}(r_1)\circ\big(\Face_{\mathfrak U}(r_2)\ast r_1^\ast\big)
=\Face_{\mathfrak U}(r_1r_2)$ for composable refinement maps $r_1,r_2$;
\item $\nu_{r\,r_1,\,r'r_1}=r_1^\ast\ast\nu_{r,r'}$ gives
$\Face_{\mathfrak U}(r')=\Face_{\mathfrak U}(r)\circ
\big(\Edge_{\mathfrak U}(\mathfrak W_2)\ast\nu_{r,r'}\big)$ for parallel refinement maps
$r,r':\mathfrak W_2\to\mathfrak W_1$.
\end{enumerate}
Thus $(\mathbf H_{\mathfrak U},\Edge_{\mathfrak U},\Face_{\mathfrak U})$ is a pseudo-cocone of
$\Phi$.

\paragraph{Step 5: the equivalence.}
Let $E:=\Edge_{\mathbf C}(\mathfrak U):\mathbf H_{\mathfrak U}\to\mathbf C$ be the canonical
edge functor. The universal property of $\mathbf C$ gives a functor
$A:\mathbf C\to\mathbf H_{\mathfrak U}$ and an isomorphism of pseudo-cocones
$A\circ\Edge_{\mathbf C}\cong\Edge_{\mathfrak U}$. Evaluating at $\mathfrak U$,
\[
  A\circ E\cong\Edge_{\mathfrak U}(\mathfrak U)=\id_{\mathbf H_{\mathfrak U}}.
\]

For the other direction, consider $E\circ A:\mathbf C\to\mathbf C$. For each cover
$\mathfrak W$ the faces $\Face_{\mathbf C}(r_{\mathfrak V\mathfrak U})^{-1}$ and
$\Face_{\mathbf C}(r_{\mathfrak V\mathfrak W})$ give an isomorphism of pseudo-cocones
\[
  E\circ\Edge_{\mathfrak U}\cong\Edge_{\mathbf C},
\]
so that
\[
  (E\circ A)\circ\Edge_{\mathbf C}\cong E\circ\Edge_{\mathfrak U}\cong\Edge_{\mathbf C},
\]
while $\id_{\mathbf C}$ induces $\Edge_{\mathbf C}$. The universal property makes
$\mathbf{Gpd}(\mathbf C,\mathbf C)\to\PsCocone(\mathbf C)$ an isomorphism of groupoids, in
particular fully faithful, so $E\circ A\cong\id_{\mathbf C}$. Hence $A$ is a quasi-inverse of
$E$, which is therefore an equivalence of groupoids.
\end{proof}

The $2$-functor $\Phi$ gives a category $F_\Phi$ fibered over $\operatorname{Cov}(M)$, whose
inductive limit over $\operatorname{Cov}(M)^{\mathrm{op}}$ exists by
\cite[Exp.~VI, 6.3]{SGA4} and is $\mathbf H^1(\mathcal F^\bullet)$. Moreover
\[
  F_{\mathbf H}(U):=\mathbf H^1_{\{U\}}(\mathcal F^\bullet|_U)
\]
defines a prestack on $\operatorname{Open}(M)$, since its morphism presheaves
$\{t\in\mathcal F^0(U):\dd^0t=\gamma-\gamma'\}$ are sheaves. Its stackification
$\widehat F_{\mathbf H}$ exists by \cite[II.2.1.3]{Giraud}, and
$\widehat F_{\mathbf H}(M)$ is the colimit of the descent categories over the covers of $M$
\cite[\S7]{Street}, so $\widehat F_{\mathbf H}(M)\simeq\mathbf H^1(\mathcal F^\bullet)$.
\begin{remark}\label{rem:naturality}
The construction is natural in $M$. Suppose that $\mathcal F^\bullet$ is defined on every
manifold and pulls back along smooth maps, as are complexes built from
sheaves of Lie group-valued functions, differential forms and locally constant functions. For a
smooth map $f:N\to M$ and a cover $\mathfrak U$ of $M$, pulling back cochains commutes with $D^0$
and $D^1$, and so gives a functor
\[
  f^\ast:\mathbf H^1_{\mathfrak U}(\mathcal F^\bullet_M)\longrightarrow
  \mathbf H^1_{f^{-1}\mathfrak U}(\mathcal F^\bullet_N),
  \qquad f^{-1}\mathfrak U:=\{f^{-1}U_j\}.
\]
These functors commute strictly with restriction along refinement maps, so they form a
pseudo-cocone and induce
\[
  f^\ast:\mathbf H^1(\mathcal F^\bullet_M)\longrightarrow\mathbf H^1(\mathcal F^\bullet_N),
\]
with $(f\circ g)^\ast\cong g^\ast f^\ast$. A morphism of complexes acts pointwise on cochains and
so commutes with $f^\ast$ on each cover. The functors it induces therefore commute with $f^\ast$
up to canonical natural isomorphism, and so do the connecting maps into $\HH^2$.
\end{remark}

\subsection{Circle extensions and central connections}
\label{subsec:centralconn}

Let $G$ be a Lie group, $\sigma:\tilde G\to G$ a Lie group double covering, and $G^c$ the
associated circle extension, or central product,
\begin{equation}\label{eq:centralext}
  1 \to \U(1) \to G^c \xrightarrow{\ \sigma^c\ } G \to 1,
  \qquad G^c\;\cong\;(\tilde G\times\U(1))/\Delta
\end{equation}
with $\Delta\cong\Z_2$ the antidiagonal.

Circle extensions of this type arise from the axioms of quantum theory and complexification of
coefficients. Since pure states are
rays \cite{Wigner}, symmetries act projectively, and we take lifts of the symmetry groups to
recover the action on the vector space of wave functions (cf.\ \cite{Bargmann}). Given a real
projective representation $\varrho:G\to\PO(V)$ whose pulled-back double covering
$G\times_{\PO(V)}\Or(V)\to G$ is isomorphic to $\sigma:\tilde G\to G$, the circle extension
\eqref{eq:centralext} can be read off the pullback squares
\begin{equation}\label{eq:pullbackGc}
\begin{tikzcd}[row sep=2.4em, column sep=2.4em]
G^c \ar[r] \ar[d,"\sigma^c"'] \ar[dr, phantom, "\lrcorner", very near start] &
  \big(\Or(V)\times\U(1)\big)/\Delta \ar[r] \ar[d,"\Ad"] \ar[dr, phantom, "\lrcorner", very near start] &
  \U(V\otimes_\R\C) \ar[d,"\Ad"] \\
G \ar[r,"\varrho"'] & \PO(V) \ar[r] & \PU(V\otimes_\R\C)
\end{tikzcd}
\end{equation}
In the symplectic case $G=\Sp(\bV)$, the metaplectic-c group $\Mp^c(\bV)$ is constructed by
pulling back the projective representation
\[ \Sp(\bV)\;\longrightarrow\;\Aut_{\U(1)}\big(\Heis(\bV;\U(1))\big)\;\xrightarrow{\ \cong\ }\;
   PN_{\U(\mathcal H)}\big(\varrho_{\mathrm{Sch}}(\Heis(\bV;\U(1)))\big)\;\hookrightarrow\;\PU(\mathcal H) \]
into the projectivization of the normalizer of the image of the Schr\"odinger representation
$\varrho_{\mathrm{Sch}}:\Heis(\bV;\U(1))\to\U(\mathcal H)$ of the Heisenberg group, where
$\Aut_{\U(1)}$ denotes the automorphisms fixing the center pointwise. It is then shown to be a central product of $\Mp(\bV)$ and $\U(1)$
\cite{Shale,Weil-unitary}, and the metaplectic representation is constructed as a lift.

In the special orthogonal case $G=\SO(V)$, the $\Spin^c(V)$ representation on the complexified
spinor space is constructed directly from Clifford modules, and descends to a projective unitary
representation of $\SO(V)$ on the projective spinor space.

The exceptional double coverings $\sigma:\Sp(\R^4)\to\SO^+(2,3)$ in conformal geometry and
$\sq:\SO(2)\to\SO(2)$ for semions in two dimensions, and the electroweak gauge group
$\U(2)\cong(\SU(2)\times\U(1))/\Z_2$, may be of particular interest.

\begin{definition}\label{def:eta}
The map
\begin{equation}\label{eq:etadef}
  \eta:G^c\longrightarrow\U(1),\qquad \eta[\tilde g,\xi]:=\xi^2
\end{equation}
is well defined on $\Delta$-cosets and gives a character
\begin{equation}\label{eq:etaseq}
  1 \to \tilde G \to G^c \xrightarrow{\ \eta\ } \U(1) \to 1
\end{equation}
restricting to the squaring map on the central $\U(1)$.
\end{definition}

Since $\eta|_{\U(1)}=\sq$, the differential at the identity of $G^c$ is
$\dd_1\eta|_{\u(1)}=2$, so $\ker\dd_1\eta$ is an ideal meeting $\u(1)$ trivially and mapping
isomorphically to $\g$. Thus $\eta$ determines a splitting $\g^c=\u(1)\oplus\g$ of Lie algebras,
and the projection onto the first summand
\begin{equation}\label{eq:mudef}
  \tfrac12\,\dd_1\eta\;:\;\g^c\longrightarrow\u(1)
\end{equation}
has a factor $\tfrac12$.

The characters of $G^c$ squaring the central circle form a $\Hom(G,\U(1))$-torsor, since for two
such characters $\eta,\eta'$ the ratio $\eta'\eta^{-1}$ descends to a character of $G$. If $\g$
is semisimple and $G$ is connected, as for the symplectic group, then $\dd_1(\eta'\eta^{-1})=0$
and the choice of $\eta$ is unique \cite[Thm.~1.11]{Hess}. Circle extensions not associated to
double coverings may admit no squaring character, for instance the Heisenberg group.

Denote by $\g^c$, $\g$ and $\u(1)\cong i\R$ the Lie algebras of $G^c$, $G$ and $\U(1)$.

\begin{lemma}\label{lem:sumMC}
Let $\Theta_G\in\Omega^1(G;\g)$, $\Theta_{G^c}\in\Omega^1(G^c;\g^c)$ and
$\Theta_{\U(1)}\in\Omega^1(\U(1);\u(1))$ be the Maurer--Cartan forms of $G$, $G^c$ and $\U(1)$.
Then
\begin{equation} \label{eq:sumMC}
  \Theta_{G^c} = \tfrac12\,\eta^\ast \Theta_{\U(1)} + (\sigma^c)^\ast \Theta_G .
\end{equation}
\end{lemma}

\begin{proof}
Let $\lambda_g:G^c\to G^c$ denote left multiplication by $g\in G^c$. From
$\eta(g^{-1}h)=\eta(g^{-1})\eta(h)$ we get
$\lambda_{\eta(g)^{-1}}\circ\eta=\eta\circ\lambda_{g^{-1}}$, so
\[(\eta^\ast\Theta_{\U(1)})(g) = \dd_{\eta(g)}\lambda_{\eta(g)^{-1}}\circ \dd_g \eta
  = \dd_1\eta\circ \dd_g\lambda_{g^{-1}} = (\dd_1\eta\circ\Theta_{G^c})(g)
  = 2\big(\tfrac12\dd_1\eta\big)\circ\Theta_{G^c}(g) .\]
The first and third equalities are the definitions of pullback and of Maurer--Cartan forms, the
second equality is the chain rule for differentiation at each $g\in G^c$, and the last equality
expresses $\dd_1\eta$ in terms of the
projection $\tfrac12\dd_1\eta:\g^c\to\u(1)$. Similarly
$(\sigma^c)^\ast\Theta_G=\dd_1\sigma^c\circ\Theta_{G^c}$, with $\dd_1\sigma^c$ the differential
of $\sigma^c$ at the identity and the projection $\g^c\to\g$. Then $\Theta_{G^c}$ is the sum of
the projections to each summand.
\end{proof}

Central connections are objects on principal $G^c$-bundles whose evaluation on vertical
directions is modelled by the projection to the central direction
$\tfrac12\,\eta^\ast\Theta_{\U(1)}$ (for closely related notions, see
$\mu$-pseudoconnection \cite[Def.~1.12]{Hess}, metaplectic-c prequantization
\cite[Def.~3.2]{VaughanQ}, splittings of a connective structure
\cite[Def.~3.5, (6)]{Gomi}, compatible connection and scalar curvature
\cite[\S2, (2.8)]{Waldorf}).

\begin{definition}[Central connection and curvature]\label{def:centralconnection}
Let $(\rho^c:P^c \to M, R^c)$ be a principal $G^c$-bundle over $M$. A $\u(1)$-valued $1$-form
$\gamma \in \Omega^1(P^c; \u(1))$ is a \textbf{central connection} if
\begin{enumerate}
\item $\gamma \circ \dd R^c_{g} = \Ad(g^{-1}) \circ \gamma$ for all $g \in G^c$; and
\item $\gamma(X_{\bxi}) = \tfrac12\dd_1\eta(\bxi)$ for all $\bxi \in \g^c$, where
$X_{\bxi}(p) = \tfrac{\dd}{\dd t}\big|_{t = 0} R^c_{\exp (t\bxi)}(p)$.
\end{enumerate}
The $\u(1)$-valued $2$-form
$\dd\gamma + \tfrac{1}{2}[\gamma, \gamma] = \dd\gamma$ is basic and descends to $M$. A $2$-form
$\omega$ on $M$ is the \textbf{curvature form} of $\gamma$ if
$\dd\gamma = i (\rho^c)^\ast \omega$, and we denote $\curv\gamma:=\omega$; the identification
$\u(1)\cong i\R$ is used here once and for all.
\end{definition}

Two principal $G^c$-bundles with central connections are \emph{isomorphic} if an isomorphism of
principal $G^c$-bundles pulls one central connection back to the other. Let
$\mathbf P(G^c;\u(1))$ denote the groupoid whose objects are the principal $G^c$-bundles with
central connection and whose morphisms are these isomorphisms.

Taking curvature forms is a functor
$\curv:\mathbf P(G^c;\u(1))\to Z_{\dR}^2(M)$ to the discrete groupoid of closed $2$-forms.
Since a connection-preserving isomorphism $\phi$ satisfies $\phi^\ast\gamma'=\gamma$,
$(\rho^c)^\ast\omega=(\rho^c)^\ast\omega'$, and since $\rho^c$ is a surjective submersion,
$\omega=\omega'$. Thus $\curv$ maps isomorphisms to identity morphisms.

For a principal $G$-bundle $P$ and a closed $2$-form $\omega$ on $M$, let
$\mathbf L(P,\sigma^c,\omega)$ denote the groupoid of $\sigma^c$-liftings of $P$, in the sense
of Definition~\ref{def:lifting}, whose central connection has curvature form $\omega$; its
morphisms are the isomorphisms of Definition~\ref{def:lifting}, that is, the
connection-preserving isomorphisms covering the identity of $P$.

The transformation law for a central connection resembles that for principal circle connections
\eqref{eq:circlelocal}. Let $(P^c,\gamma)$ be a principal $G^c$-bundle with central connection,
and let $s:U\to P^c$ and $g:U\to G^c$ on an open $U\subseteq M$. Then by the Leibniz rule
\cite[Ch.~I, Prop.~1.4]{KobayashiNomizu},
\begin{align*}
  \dd(s \cdot g) = \dd\big(R^c \circ (s,g)\big)
  &= \dd R^c\circ(\dd s,0) + \dd R^c\circ(0,\dd g) \\
  &= \dd_s R^c_g \circ \dd s + X_{\Theta_{G^c}(\dd g)}(s \cdot g) ,
\end{align*}
where the second term is the fundamental vector field of the principal action. From its defining
conditions in Definition~\ref{def:centralconnection}, the central connection evaluates to
\begin{align*}
  \gamma\big(\dd_s R^c_g \circ \dd s\big)
  &= \Ad(g^{-1})\circ s^\ast\gamma = s^\ast \gamma , \\
  \gamma\big(X_{\Theta_{G^c}(\dd g)}(s \cdot g)\big)
  &= \tfrac12\,\dd_1\eta\big(\Theta_{G^c}(\dd g)\big)
   = \tfrac12\,g^\ast \eta^\ast \Theta_{\U(1)}
   = \tfrac12\,\eta(g)^{-1}\dd\,\eta(g)
   = \tfrac12\,\dd\log \eta(g) ,
\end{align*}
on each term, by the centrality of $\u(1)$ in $\g^c$ and
Lemma~\ref{lem:sumMC}. So the transformation law for central connections is given by
\begin{equation} \label{eq:centralconnectionlocal}
  (s \cdot g)^\ast \gamma = \Ad(g^{-1})\circ s^\ast \gamma
     + \tfrac12\,\dd_1\eta\big(\Theta_{G^c}(\dd g)\big)
  = s^\ast \gamma + \tfrac12\,\dd\log \eta(g).
\end{equation}
So given a principal $G^c$-bundle with central connection, a choice of local sections $s_j$ on
$U_j$ with $s_k=s_j\cdot g^c_{jk}$ on $U_{jk}$ defines descent data
$(\{\gamma_j\},\{g^c_{jk}\})$ satisfying
\begin{align}
  g^c_{jk}\,g^c_{k\ell}\,(g^c_{j\ell})^{-1}&=1
    &&\text{ on } U_{jk\ell}, \label{eq:transitionmaps} \\
  \gamma_k-\gamma_j-\tfrac12\dd\log\eta(g^c_{jk})&=0
    &&\text{ on } U_{jk}, \label{eq:centralconnection}
\end{align}
where $\gamma_j=s_j^\ast\gamma$. Conversely, given descent data
$(\{\gamma_j\},\{g^c_{jk}\})$ satisfying \eqref{eq:transitionmaps} and
\eqref{eq:centralconnection}, we can construct a principal $G^c$-bundle with transition
functions $\{g^c_{jk}\}$ and a central connection $\gamma$ having $\{\gamma_j\}$ as its local
connection forms.

\subsection{The lifting problem with prescribed central curvature}
\label{subsec:prescribedcurv}

Let $\sigma:\tilde G\to G$ be a Lie group double covering, and let $M$ be a smooth manifold.
Consider the sequence of functors of groupoids
\begin{multline}\label{eq:bundleseq}
  \mathbf P(\mu_2)
  \xrightarrow{\ (\times_{\mu_2}\tilde G,\ \times_{\mu_2}\U(1))\ }
  \mathbf P(\tilde G)\times\mathbf P_0(\U(1);\u(1))
  \xrightarrow{\ \times_M/\Delta\ }\\
  \mathbf P(G^c;\u(1))
  \xrightarrow{\ (w(\,\cdot\,/\U(1),\sigma),\ \curv)\ }
  H^2(M;\Z_2)\oplus Z^2_{\dR}(M).
\end{multline}
Here $(\times_{\mu_2}\tilde G,\ \times_{\mu_2}\U(1))$ is a pair of extensions of structure
groups by $\mu_2\cong\ker\sigma\hookrightarrow\tilde G$ and $\mu_2\hookrightarrow\U(1)$,
\[ T \longmapsto\big(T\times_{\mu_2}\tilde G,\ T\times_{\mu_2}\U(1)\big). \]
The functor $\times_M/\Delta$ is the central product along the antidiagonal
$\mu_2\cong\Delta$,
\[ (\tilde P,Y)\longmapsto \big(\tilde P\times_M Y\big)/\Delta. \]
The bundle $\big(\tilde P\times_M Y\big)/\Delta$ is a principal $G^c$-bundle with a central
connection induced from the principal connection of $Y$. The third functor is
\[ (P^c,\gamma)\longmapsto\big(w(P^c/\U(1),\sigma),\ \curv\gamma\big), \]
where $w(P,\sigma)\in H^2(M;\Z_2)$ is the obstruction class of \eqref{eq:classicalseq} for
$\mathcal C=\underline{\Z_2}$, and $\curv\gamma$ is the central curvature of $\gamma$.

We now construct a short exact sequence of complexes of sheaves analogous to
\eqref{eq:truncationses}, that encodes \eqref{eq:transitionmaps} and
\eqref{eq:centralconnection} into its total differential. Let $\mathcal G$,
$\tilde{\mathcal G}$ and $\mathcal G^c$ be the sheaves of smooth $G$-, $\tilde G$- and
$G^c$-valued functions on $M$. Let
\begin{align}
  \Delta^\bullet
    &:= \ \underline{\Delta}\ \to\ 0\ \to\ i\Omega^2 \xrightarrow{\dd}\cdots
        \xrightarrow{\dd} i\Omega^m \to 0, \label{eq:Deltacomplex}\\
  (\tilde{\mathcal G}\times\Usm)^\bullet
    &:= \ \tilde{\mathcal G}\times\Usm \xrightarrow{\frac12\dd\log\eta'} i\Omega^1
        \xrightarrow{\dd} i\Omega^2 \xrightarrow{\dd}\cdots \xrightarrow{\dd} i\Omega^m \to 0,
        \nonumber\\
  (\mathcal G^c)^\bullet
    &:= \ \mathcal{G}^c \xrightarrow{\frac12\dd\log\eta} i\Omega^1 \to 0,
        \nonumber
\end{align}
where $\eta'(\tilde g,\xi)=\xi^2$, so that $\tfrac12\dd\log\eta'=\dd\log\xi$.

This gives a short exact sequence of complexes of sheaves,
\begin{equation}\label{eq:Deltases}
  1 \to \Delta^\bullet \to (\tilde{\mathcal G}\times\Usm)^\bullet
    \to (\mathcal G^c)^\bullet \to 1 .
\end{equation}
We compute the hypercohomology groupoids in the following statement
(Figures~\ref{fig:Ecomplex} and~\ref{fig:Deltacomplex}), and identify them with the sequence
\eqref{eq:bundleseq}.
\begin{figure}[htbp]
\centering
\begin{tikzpicture}[>=stealth,font=\footnotesize]
  \def\dx{4.1}\def\dy{1.9}
  \node (e02) at (0,2*\dy)      {\phantom{$\dd\gamma$}};
  \node (e01) at (0,\dy)        {$\gamma$};
  \node (e11) at (\dx,\dy)      {$\check\delta\gamma-\tfrac12\dd\log\eta(g^c)$};
  \node (e10) at (\dx,0)        {$g^c$};
  \node (e20) at (2*\dx,0)      {$\check\delta g^c$};
  \draw[->] (e01) -- node[above,font=\scriptsize,pos=0.20,fill=white,inner sep=1.5pt]{$\check\delta$} (e11);
  \draw[->] (e10) -- node[above,font=\scriptsize,pos=0.28,fill=white,inner sep=1.5pt]{$\check\delta$} (e20);
  \draw[->,densely dotted] (e01) -- node[left,font=\scriptsize,fill=white,inner sep=1.5pt,pos=0.28]{$0$} (e02);
  \draw[->] (e10) -- node[right,font=\scriptsize,pos=0.72,fill=white,inner sep=1.5pt]{$\tfrac12\dd\log\eta$} (e11);
  \begin{scope}[on background layer]
    \node[fill=black!16,rounded corners,fit=(e02),inner sep=5pt](ebox){};
    \node[fill=black!6,rounded corners,fit=(e11),inner sep=5pt]{};
    \node[fill=black!6,rounded corners,fit=(e20),inner sep=5pt]{};
  \end{scope}
  \node[anchor=west,font=\scriptsize,align=left] at ($(ebox.east)+(0.35,0)$)
    {empty condition in $(\mathcal G^c)^\bullet$;\\ the curvature is unconstrained};
  \node[font=\scriptsize,anchor=north west] at ($(e01.south)+(-0.30,-0.12)$) {central connection forms};
  \node[font=\scriptsize,anchor=north] at ($(e10.south)+(0,-0.12)$) {transition functions};
  \node[font=\scriptsize,anchor=north] at ($(e20.south)+(0,-0.12)$) {bundle gluing};
  \node[font=\scriptsize,anchor=west,align=left] at ($(e11.east)+(0.30,0)$)
    {transformation law for\\ central connection forms};
  \draw[-] (-1.0,-0.95) -- (2*\dx+1.0,-0.95);
  \draw[-] (-1.0,-0.95) -- (-1.0,2*\dy+0.7);
  \foreach \i/\lab in {0/{C^0},1/{C^1},2/{C^2}}
     \node[font=\scriptsize] at (\i*\dx,-1.3) {$\lab(\mathfrak U;-)$};
  \foreach \j/\lab in {0/{\mathcal G^c},1/{i\Omega^1},2/{0}}
     \node[font=\scriptsize,anchor=east] at (-1.15,\j*\dy) {$\lab$};
  \draw[dashed,thick] (-0.20*\dx,1.70*\dy) -- (1.70*\dx,-0.20*\dy);
\end{tikzpicture}
\caption{The nonabelian analogue of Figure~\ref{fig:doublecomplex}.}
\label{fig:Ecomplex}
\end{figure}
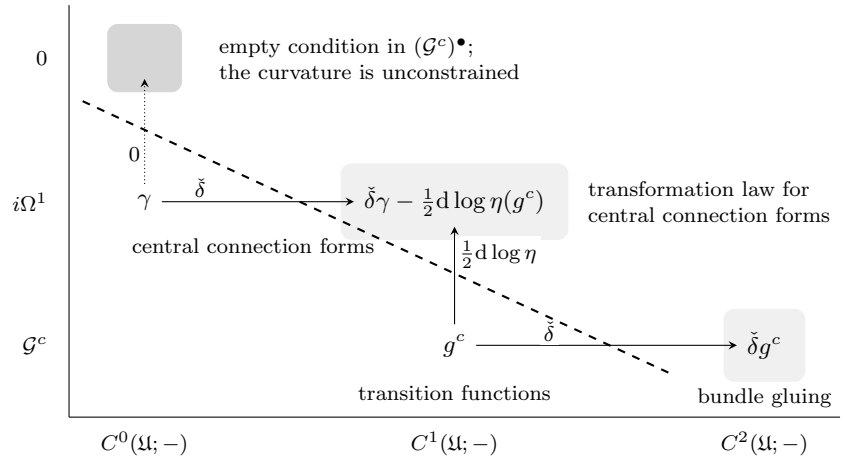

\begin{figure}[htbp]
\centering
\begin{tikzpicture}[>=stealth,font=\footnotesize]
  \def\dx{3.1}\def\dy{1.35}
  \node (d03) at (0,3*\dy)      {$i\,\dd\omega$};
  \node (d02) at (0,2*\dy)      {$i\,\omega$};
  \node (d12) at (\dx,2*\dy)    {$i\,\check\delta\omega$};
  \node (d20) at (2*\dx,0)      {$\varepsilon$};
  \node (d30) at (3*\dx,0)      {$\check\delta\varepsilon$};
  \draw[->] (d02) -- node[left,font=\scriptsize,fill=white,inner sep=1.5pt,pos=0.4]{$\dd$} (d03);
  \draw[->] (d02) -- node[above,font=\scriptsize,pos=0.30,fill=white,inner sep=1.5pt]{$\check\delta$} (d12);
  \draw[->] (d20) -- node[above,font=\scriptsize,pos=0.30,fill=white,inner sep=1.5pt]{$\check\delta$} (d30);
  \begin{scope}[on background layer]
    \node[fill=black!6,rounded corners,fit=(d03),inner sep=5pt]{};
    \node[fill=black!6,rounded corners,fit=(d12),inner sep=5pt]{};
    \node[fill=black!6,rounded corners,fit=(d30),inner sep=5pt]{};
  \end{scope}
  \node[font=\scriptsize,anchor=north west] at ($(d02.south)+(-0.30,-0.12)$) {local $2$-forms};
  \node[font=\scriptsize,anchor=north] at ($(d20.south)+(0,-0.12)$) {Stiefel--Whitney cochain};
  \node[font=\scriptsize,anchor=west] at ($(d03.east)+(0.30,0)$) {closedness};
  \node[font=\scriptsize,anchor=west] at ($(d12.east)+(0.30,0)$) {gluing};
  \node[font=\scriptsize,anchor=south] at ($(d30.north)+(0,0.14)$) {cocycle};
  \draw[-] (-1.0,-0.95) -- (3*\dx+1.0,-0.95);
  \draw[-] (-1.0,-0.95) -- (-1.0,3*\dy+0.7);
  \foreach \i/\lab in {0/{C^0},1/{C^1},2/{C^2},3/{C^3}}
     \node[font=\scriptsize] at (\i*\dx,-1.3) {$\lab(\mathfrak U;-)$};
  \foreach \j/\lab in {0/{\underline\Delta},1/{0},2/{i\Omega^2},3/{i\Omega^3}}
     \node[font=\scriptsize,anchor=east] at (-1.15,\j*\dy) {$\lab$};
  \draw[dashed,thick] (-0.20*\dx,2.70*\dy) -- (2.70*\dx,-0.20*\dy);
\end{tikzpicture}
\caption{The nonabelian analogue of Figure~\ref{fig:Tcomplex}. Because
$\Delta^\bullet$ retains $\underline\Delta$ in degree zero, a closed total degree $2$
cochain has two independent components, and the degree two cohomology of the total complex is
$H^2(M;\Z_2)\oplus i\,Z^2_{\dR}(M)$.}
\label{fig:Deltacomplex}
\end{figure}
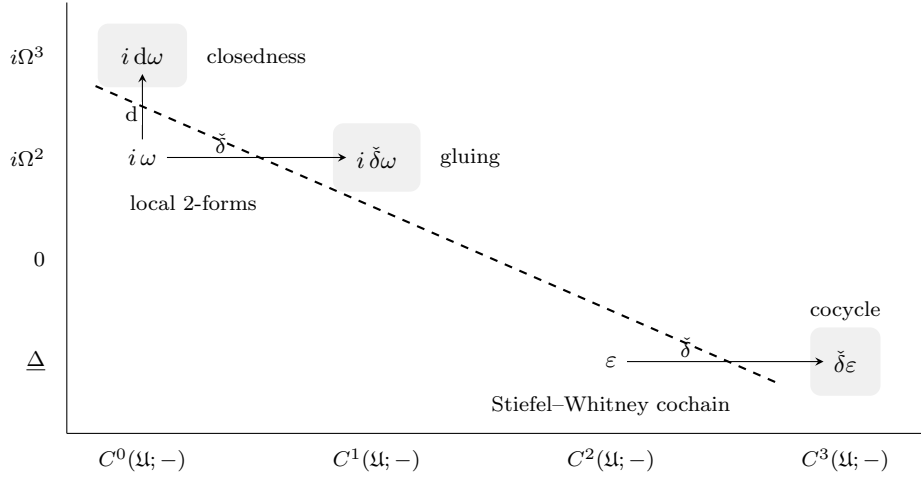

\begin{lemma}\label{lem:gpdLES}
Let $M$ be a smooth manifold, let $\sigma:\tilde G\to G$ be a Lie group double covering, and
let $\sigma^c:G^c\to G$ be the associated circle extension \eqref{eq:centralext}.

There is a diagram of functors of groupoids
\[
\begin{tikzcd}[row sep=3em, column sep=2.6em]
\mathbf{H}^1(\Delta^\bullet)
  \ar[r] \ar[d,"\simeq"']
&
\mathbf{H}^1\big((\tilde{\mathcal G}\times\Usm)^\bullet\big)
  \ar[r] \ar[d,"\simeq"']
&
\mathbf{H}^1\big((\mathcal G^c)^\bullet\big)
  \ar[r,"\partial"] \ar[d,"\simeq"']
&
\HH^2(\Delta^\bullet)
  \ar[d,"\cong"]
\\
\mathbf P(\mu_2)
  \ar[r]
&
\mathbf P(\tilde G)\times\mathbf P_0(\U(1);\u(1))
  \ar[r]
&
\mathbf P(G^c;\u(1))
  \ar[r]
&
H^2(M;\Z_2)\oplus Z^2_{\dR}(M).
\end{tikzcd}
\]
such that the bottom row is \eqref{eq:bundleseq}, the first three vertical arrows are
equivalences of groupoids given by gluing transition functions and local connection forms, the
first two squares commute up to canonical natural isomorphism, and the last square commutes
strictly. The groups $\HH^2(\Delta^\bullet)$ and $H^2(M;\Z_2)\oplus Z^2_{\dR}(M)$ are viewed
as discrete groupoids, and are identified up to isomorphism with a choice of $\pm i$.
\end{lemma}

\begin{proof}
Fix a good cover $\mathfrak U=\{U_j\}$ of $M$. We establish the diagram's equivalence by
evaluating the hypercohomology groupoids on $\mathfrak U$.

The degree-one cocycles of the three complexes on $\mathfrak U$ are exactly the descent data
defining the respective principal bundles.

For $\Delta^\bullet$, a cocycle consists of locally constant functions
$\varepsilon_{jk}:U_{jk}\to\Delta$ satisfying
$\varepsilon_{jk}\varepsilon_{k\ell}=\varepsilon_{j\ell}$, which define a principal
$\mu_2$-bundle.

For $(\tilde{\mathcal G}\times\Usm)^\bullet$, a cocycle is a tuple
$(\gamma_j,(\tilde g_{jk},\xi_{jk}))$ satisfying $\tilde g_{jk}\tilde g_{k\ell}=\tilde
g_{j\ell}$, $\xi_{jk}\xi_{k\ell}=\xi_{j\ell}$, $\gamma_k-\gamma_j=\dd\log\xi_{jk}$, and
$\dd\gamma_j=0$. These define a principal $\tilde G$-bundle alongside a principal
$\U(1)$-bundle equipped with a flat connection.

For $(\mathcal G^c)^\bullet$, a cocycle is a pair $(\gamma_j,g^c_{jk})$ satisfying
$g^c_{jk}g^c_{k\ell}=g^c_{j\ell}$ and $\gamma_k-\gamma_j=\tfrac12\dd\log\eta(g^c_{jk})$. These
define a principal $G^c$-bundle with a central connection.

Because local isomorphisms that agree on overlaps glue uniquely to global bundle isomorphisms,
the groupoids of \v{C}ech cocycles on $\mathfrak U$ are equivalent to the groupoids of global
principal bundles. Since $\mathfrak U$ is a good cover, Proposition~\ref{prop:goodcover}
identifies these cocycle groupoids with the hypercohomology pseudo-colimits. The gluing
functors representing the first three vertical arrows are therefore equivalences of groupoids.
Extending the structure group along $\Delta\hookrightarrow\tilde G\times\U(1)$ and taking the
central quotient to $G^c$ map these cocycle conditions along the chain maps of
\eqref{eq:Deltases}, so the first two squares commute up to canonical natural isomorphism.

For the last square, the fourth vertical arrow is the canonical isomorphism
$\HH^2(\Delta^\bullet)\cong H^2(M;\Z_2)\oplus i\,Z^2_{\dR}(M)$. Let $(P^c,\gamma)$ be a
principal $G^c$-bundle with central connection, represented by descent data
$(\gamma_j,g^c_{jk})$ on $\mathfrak U$. Choose smooth local lifts $(\tilde g_{jk},\xi_{jk})$ in
$\tilde G\times\U(1)$. Their failure to form a strict cocycle is measured by
\[
  \varepsilon_{jk\ell}
  :=\big(\tilde g_{jk}\tilde g_{k\ell}\tilde g_{j\ell}^{-1},\,
  \xi_{jk}\xi_{k\ell}\xi_{j\ell}^{-1}\big)\in\Delta,
\]
which is a \v{C}ech $2$-cocycle since $\Delta$ is central. Because
$\tfrac12\dd\log\eta([\tilde g_{jk},\xi_{jk}])=\dd\log\xi_{jk}$, the differential of this
lifted cochain evaluates to
\[
  D^1\big(\gamma,(\tilde g,\xi)\big)=(\dd\gamma,0,\varepsilon).
\]
Under $\Delta\cong\ker\sigma$, the first component of $\varepsilon$ is formed by lifting the
transition functions of $P^c/\U(1)$, so it represents $w(P^c/\U(1),\sigma)\in H^2(M;\Z_2)$. The
forms $\dd\gamma_j$ glue to $i\,\curv\gamma$, representing $\curv\gamma\in Z^2_{\dR}(M)$. The
connecting class therefore maps to $\big(w(P^c/\U(1),\sigma),\,\curv\gamma\big)$, and the last
square commutes strictly.
\end{proof}

As in differential cohomology \cite[(1.5)]{Freed}, the second Stiefel--Whitney class
$w(P,\sigma)$ of a principal $G$-bundle $P$ in
$H^2(M;\Z_2)$ can be compared with a closed $2$-form $\omega\in Z^2_{\dR}(M)$ in
$H^2(M;\Uloc)$ using the maps $(-1)^\cdot:H^2(M;\Z_2)\to H^2(M;\Uloc)$ and
$e^{i\int}:Z^2_{\dR}(M)\to H^2(M;\Uloc)$ of \eqref{eq:kostantweil}, through the square
\begin{equation}\label{eq:swcurvsquare}
\begin{tikzcd}[row sep=3em, column sep=3.4em]
\mathbf P(G^c;\u(1)) \ar[r, "\curv"] \ar[d, "{\cdot/\U(1)}"'] &
  Z_{\dR}^2(M) \ar[d, "e^{i\int}"] \\
\mathbf P(G)\ar[r, "{(-1)^{w(\,\cdot\,,\sigma)}}"'] & H^2(M;\Uloc)
\end{tikzcd}
\end{equation}
The map $(-1)^\cdot$ factors through the maps provided by the universal coefficient theorem,
\[ H^2(M; \Z_2) \twoheadrightarrow \Hom_\Z(H_2(M; \Z), \Z_2)
   \longrightarrow \Hom_\Z(H_2(M; \Z), \U(1)) \xrightarrow{\cong} H^2(M;\Uloc) , \]
so it is trivial on $\mathrm{Ext}\big(H_1(M;\Z),\Z_2\big)$.

\begin{lemma}\label{lem:u1agreement}
Let $M$ be a smooth manifold, let $\sigma:\tilde G\to G$ be a Lie group double covering, and
let $\sigma^c:G^c\to G$ be the associated circle extension \eqref{eq:centralext}. Let
$(P^c,\gamma)$ be a principal $G^c$-bundle with central connection, and let
$P=P^c/\U(1)$. Then
\[
  (-1)^{w(P,\sigma)}\cdot e^{i\int \curv\gamma}
\]
is the identity element of $H^2(M;\Uloc)$.
\end{lemma}

\begin{proof}
Let $P=P^c/\U(1)$ and $\omega=\curv\gamma$. For a choice of local sections on a good cover
$\mathfrak U$, there are local connection forms $\gamma_j$ and transition functions
$g^c_{jk}$. By choosing representatives at each value of $g^c_{jk}$ we obtain
$(\tilde g_{jk},\xi_{jk})$ in $\check C^1(\mathfrak U;\tilde{\mathcal G}\times\Usm)$. The
coboundary $\check\delta(\tilde g,\xi)$ projects to the trivial coboundary
$\check\delta g^c=1$, so it takes values in the kernel $\Delta$ of
$\tilde G\times\U(1)\to G^c$. The two components agree,
\[
  \tilde g_{jk}\tilde g_{k\ell}\tilde g_{j\ell}^{-1}
  =\xi_{jk}\xi_{k\ell}\xi_{j\ell}^{-1}
  =(-1)^{\varepsilon_{jk\ell}} ,
\]
and define a cocycle $\varepsilon\in\check C^2(\mathfrak U;\underline{\Z_2})$. Since
$\sigma^c[\tilde g,\xi]=\sigma(\tilde g)$, the $\tilde g_{jk}$ lift the transition functions of
$P$, so $\varepsilon$ represents $w(P,\sigma)$.

On the other hand, forgetting $\tilde g$, we obtain a total degree one cochain $(\gamma,\xi)$
in $\mathcal C^\bullet$. By \eqref{eq:circlelocal} its coboundary is
\[
  D^1(\gamma,\xi)
  =\big(\dd\gamma,\ \gamma_k-\gamma_j-\dd\log\xi_{jk},\ \check\delta\xi\big)
  =\big(i\,\omega,\,0,\,(-1)^{\varepsilon}\big).
\]
The second component vanishes because $\gamma$ is a central connection
\eqref{eq:centralconnection}, and the third was computed above.

The class of $\big(i\,\omega,0,(-1)^\varepsilon\big)$ in $\HH^2(\mathcal C^\bullet)$ maps to
$e^{i\int \omega}\cdot(-1)^{w(P,\sigma)}$ under the \v{C}ech--de Rham isomorphism
$\HH^2(\mathcal C^\bullet)\cong H^2(M;\Uloc)$ of \eqref{eq:kwidentification}. Since it is a
coboundary, it is trivial, so $(-1)^{w(P,\sigma)}\cdot e^{i\int \omega}$ is the identity
element.
\end{proof}

For a principal $G$-bundle $P$ and a closed $2$-form $\omega$, let
\begin{equation}\label{eq:wclassdef}
  w(P,\sigma^c,\omega)\;:=\;(-1)^{w(P,\sigma)}\cdot e^{i\int \omega}\;\in\;H^2(M;\Uloc)
\end{equation}
as in \eqref{eq:mainclass}. By Lemma~\ref{lem:u1agreement} the composite
\[ \mathbf P(G^c;\u(1))\xrightarrow{\ (\,\cdot\,/\U(1),\,\curv)\ }
   \mathbf P(G)\times Z^2_{\dR}(M)
   \xrightarrow{\ w(\,\cdot\,,\sigma^c,\,\cdot\,)\ }H^2(M;\Uloc) \]
is constant at the identity element. Equivalently, $w(P,\sigma^c,\omega)$ is trivial if
$(P,\omega)$ admits a $\sigma^c$-lifting with central connection $(P^c,\gamma)$ in the sense of
Definition~\ref{def:lifting}. Since the homotopy fibre of $(\,\cdot\,/\U(1),\curv)$ over
$(P,\omega)$ is the groupoid $\mathbf L(P,\sigma^c,\omega)$ of $\sigma^c$-liftings of $P$ with
central curvature $\omega$, this groupoid is nonempty only when $w(P,\sigma^c,\omega)$ is
trivial.

\begin{theorem}\label{thm:Pc}
Let $M$ be a smooth manifold, let $\sigma:\tilde G\to G$ be a Lie group double covering, and
let $\sigma^c:G^c\to G$ be the associated circle extension \eqref{eq:centralext}. Consider the sequence
\begin{equation}\label{eq:fibseq2}
  \mathbf P_0(\U(1);\u(1))\to\mathbf P(G^c;\u(1))
  \xrightarrow{(\,\cdot\,/\U(1),\,\curv)}\mathbf P(G)\times Z^2_{\dR}(M)
  \xrightarrow{w(\,\cdot\,,\sigma^c,\,\cdot\,)}H^2(M;\Uloc),
\end{equation}
where $Z^2_{\dR}(M)$ and $H^2(M;\Uloc)$ are viewed as discrete groupoids. Let $P$ be a
principal $G$-bundle and let $\omega$ be a closed $2$-form on $M$.
\begin{enumerate}[label=\textup{(\arabic*)}]
\item The sequence \eqref{eq:fibseq2} satisfies the exactness condition $(\mathrm{E}1)$ at
$\mathbf P(G)\times Z^2_{\dR}(M)$ for $w\in H^2(M;\Uloc)$ if and only if $w$ is the identity
element.

\item The sequence \eqref{eq:fibseq2} satisfies the exactness condition $(\mathrm{E}2)$ at
$\mathbf P(G^c;\u(1))$ for $(P,\omega)\in\mathbf P(G)\times Z^2_{\dR}(M)$ if and only if
$(P,\omega)$ is isomorphic to the trivial principal $G$-bundle with the zero form. In
particular, $\mathbf P_0(\U(1);\u(1))$ is equivalent as a groupoid to the homotopy fibre
$\mathbf L(M\times G,\sigma^c,0)$.

\item The homotopy fibre $\mathbf L(P,\sigma^c,\omega)$ of $(\,\cdot\,/\U(1),\curv)$ at
$(P,\omega)$ is nonempty if and only if $w(P,\sigma^c,\omega)$ is the identity element, and
all nonempty homotopy fibres of $(\,\cdot\,/\U(1),\curv)$ are equivalent as groupoids.
\end{enumerate}
\end{theorem}

\begin{proof}
\textup{(1)} By Lemma~\ref{lem:u1agreement} the composite
$w(\,\cdot\,,\sigma^c,\,\cdot\,)\circ(\,\cdot\,/\U(1),\curv)$ is constant at the identity
element of $H^2(M;\Uloc)$. Since $H^2(M;\Uloc)$ is discrete, a natural transformation from this
composite to the constant functor at $w$ exists if and only if $w$ is isomorphic, here equal,
to the identity element. So $(\mathrm{E}1)$ fails for every other $w$, and for the identity
element it remains to show that the induced functor
\[
  \mathbf P(G^c;\u(1))\longrightarrow
  \hofib{w(\,\cdot\,,\sigma^c,\,\cdot\,)}{\,\id\,}
\]
is essentially surjective. Since the objects of the homotopy fibre are the pairs $(P,\omega)$
with $w(P,\sigma^c,\omega)$ the identity element, it suffices to construct a $\sigma^c$-lifting
of $P$ with central curvature $\omega$ for each such pair.

We reverse the argument of Lemma~\ref{lem:u1agreement}. Take a good cover $\mathfrak U$ of $M$,
so that $(P,\omega)$ has descent data consisting of transition functions $g_{jk}$ with
$\check\delta g=1$ and the local $2$-forms $i\,\omega|_{U_j}$. Choose lifts
$\tilde g_{jk}:U_{jk}\to\tilde G$ of $g_{jk}$ on each $U_{jk}$. Their coboundary
$\check\delta\tilde g$ takes values in $\ker\sigma\cong\mu_2$, and represents
$(-1)^{w(P,\sigma)}$ in $H^2(M;\Uloc)$.

Applying the \v{C}ech--de Rham zigzag to $\omega$ gives a total degree one cochain
$(\gamma,\xi)$ of $\mathcal C^\bullet$ with
\[
  D^1(\gamma,\xi)
  =\big(\dd\gamma,\ \check\delta\gamma-\dd\log\xi,\ \check\delta\xi\big)
  =\big(i\,\omega,\ 0,\ \check\delta\xi\big):
\]
choose $\gamma_j\in i\Omega^1(U_j)$ with $\dd\gamma_j=i\,\omega|_{U_j}$, and then
$\xi_{jk}=e^{f_{jk}}$ for functions $f_{jk}$ with $\dd f_{jk}=\gamma_k-\gamma_j$, which exist
because $U_{jk}$ is contractible. The coboundary $\check\delta\xi$ is locally constant, since
$\dd\log\check\delta\xi=\check\delta\,\dd\log\xi=\check\delta\,\check\delta\gamma=0$ and the
$U_{jk\ell}$ are connected. Since $(i\,\omega,0,1)$ is itself a total degree two cocycle of
$\mathcal C^\bullet$, with class $e^{i\int \omega}$ by \eqref{eq:kwidentification}, and the
displayed triple is a coboundary, $\check\delta\xi$ represents $e^{-i\int \omega}$ in
$H^2(M;\Uloc)$.

Both $\check\delta\tilde g$ and $\check\delta\xi$ are locally constant, so their product is a
cocycle in $\check C^2(\mathfrak U;\Uloc)$ whose class is
$(-1)^{w(P,\sigma)}\cdot e^{-i\int \omega}=w(P,\sigma^c,\omega)^{-1}$ by
\eqref{eq:wclassdef}. Since $w(P,\sigma^c,\omega)$ is the identity element, so is its inverse,
and the product is a coboundary in $\Uloc$, so there is a locally constant
$u\in\check C^1(\mathfrak U;\Uloc)$ with
\[
  \check\delta\tilde g\cdot\check\delta\xi\cdot\check\delta u^{-1}=1 .
\]
Put $\xi':=\xi u^{-1}$. Then $\dd\log\xi'=\dd\log\xi$ because $u$ is locally constant, so
$(\gamma,\xi')$ still satisfies $\check\delta\gamma=\dd\log\xi'$, and
$\check\delta\xi'=\check\delta\tilde g$ because the elements of $\mu_2$ are their own inverses.
The two coboundaries of $(\tilde g,\xi')$ therefore agree in $\mu_2$, so
$\check\delta(\tilde g,\xi')$ takes values in the antidiagonal $\Delta$ and
$\check\delta[\tilde g,\xi']=1$ in $\mathcal G^c$. Since
$\tfrac12\dd\log\eta[\tilde g,\xi']=\dd\log\xi'$, the pair $\big(\gamma,[\tilde g,\xi']\big)$ is
a total degree one cocycle of $(\mathcal G^c)^\bullet$, and $\sigma^c[\tilde g,\xi']=g$ with
$\dd\gamma_j=i\,\omega|_{U_j}$. It is therefore descent data for a $\sigma^c$-lifting of $P$
with central curvature $\omega$.

\textup{(2)} The first functor of \eqref{eq:fibseq2} sends a circle bundle with flat connection
$(Y,\gamma_Y)$ to the extension of structure group $Y\times_{\U(1)}G^c$, whose central
connection is induced by $\gamma_Y$. The central $\U(1)$ acts trivially on $G^c/\U(1)\cong G$,
so the quotient is the trivial principal $G$-bundle and the induced central connection has
curvature $\curv\gamma_Y=0$. The composite is therefore constant at $(M\times G,0)$, so a
natural transformation to the constant functor at $(P,\omega)$ exists exactly when $(P,\omega)$
is isomorphic to $(M\times G,0)$, and $(\mathrm{E}2)$ fails for every other $(P,\omega)$. For
$(P,\omega)=(M\times G,0)$ the composite provides the natural transformation, and the induced
functor $\mathbf P_0(\U(1);\u(1))\to\mathbf L(M\times G,\sigma^c,0)$ is the equivalence of
\textup{(3)} below, hence full and essentially surjective.

\textup{(3)} The homotopy fibre $\mathbf L(P,\sigma^c,\omega)$ is nonempty exactly when
$w(P,\sigma^c,\omega)$ is the identity element, by \textup{(1)}. Suppose it is, and let
$(P^c_1,\gamma_1)$ and $(P^c_2,\gamma_2)$ be two of its objects. Over a good cover
$\mathfrak U$ they have descent data $(\gamma_1,g^c_1)$ and $(\gamma_2,g^c_2)$ above the same
transition functions $g$ of $P$. Since $\sigma^c(g^c_1)=\sigma^c(g^c_2)=g$, the ratio
\[
  \xi_{12}:=g^c_2\,(g^c_1)^{-1}
\]
takes values in the central $\U(1)$, and $\gamma_{12}:=\gamma_2-\gamma_1$ satisfies
$\check\delta\gamma_{12}=\dd\log\xi_{12}$ by \eqref{eq:centralconnection} and
$\dd\gamma_{12}=0$ because both central connections have curvature $\omega$. So
$(\gamma_{12},\xi_{12})$ is a degree one cocycle of $\mathcal C^\bullet$, that is, descent data
for a principal circle bundle with flat connection $(Y_{12},\gamma_{12})$ on $M$. Viewing
$P^c_1$ and $P^c_2$ as principal circle bundles over $P$ through $\bu$, their descent data
differ by $(\gamma_{12},\xi_{12})$, so
\[
  \big(P^c_1\otimes\rho^\ast Y_{12},\ \gamma_1+\rho^\ast\gamma_{12}\big)
  \;\cong\;(P^c_2,\gamma_2)
\]
as principal circle bundles with connection over $P$.

A choice of object $(P^c_1,\gamma_1)$ therefore gives a functor
\[
  \mathbf L(P,\sigma^c,\omega)\longrightarrow\mathbf P_0(\U(1);\u(1)),
  \qquad
  (P^c_2,\gamma_2)\longmapsto(Y_{12},\gamma_{12}),
\]
which is essentially surjective because tensoring $(P^c_1,\gamma_1)$ with a flat circle bundle
again lies in $\mathbf L(P,\sigma^c,\omega)$, and fully faithful because a morphism of
$\mathbf L(P,\sigma^c,\omega)$ covers the identity of $P$ and so is given by a $\U(1)$-valued
function. So every nonempty homotopy fibre of $(\,\cdot\,/\U(1),\curv)$ is equivalent to
$\mathbf P_0(\U(1);\u(1))$, and in particular any two of them are equivalent, whatever their
underlying $(P,\omega)$.
\end{proof}

On isomorphism classes, Theorem~\ref{thm:Pc} descends to a lifting criterion that generalizes
the metaplectic-c prequantizability criterion of \cite[Thm.~1.14]{Hess} and
\cite[(7.3)]{RobinsonRawnsley}.

\begin{corollary}[Liftings with prescribed central curvature]\label{cor:Pc}
Let $M$ be a smooth manifold, let $\sigma:\tilde G\to G$ be a Lie group double covering, and
let $\sigma^c:G^c\to G$ be the associated circle extension \eqref{eq:centralext}. Let $P$ be a
principal $G$-bundle over $M$ and let $\omega$ be a closed $2$-form on $M$.
\begin{enumerate}[label=\textup{(\arabic*)}]
\item There exists a $\sigma^c$-lifting $(P^c,\gamma)$ of $P$ with $\curv\gamma=\omega$ if and
only if $w(P,\sigma^c,\omega)$ is the identity element of $H^2(M;\Uloc)$.
\item $H^1(M;\U(1))$ acts freely and transitively on the isomorphism classes of such liftings,
by tensoring with flat circle bundles, and every such lifting has automorphism group
$H^0(M;\U(1))$.
\end{enumerate}
\end{corollary}

\begin{proof}
The $\sigma^c$-liftings of $P$ with central curvature $\omega$ are the objects of
$\mathbf L(P,\sigma^c,\omega)$, the homotopy fibre of $(\,\cdot\,/\U(1),\curv)$ over
$(P,\omega)$. It is nonempty exactly when
$w(P,\sigma^c,\omega)$ is the identity element, by Theorem~\ref{thm:Pc}(3). When it is
nonempty it is equivalent to $\mathbf L(M\times G,\sigma^c,0)$ by the same item, hence to
$\mathbf P_0(\U(1);\u(1))$ by Theorem~\ref{thm:Pc}(2). So
$\pi_0\mathbf L(P,\sigma^c,\omega)$ is a torsor under
$\pi_0\mathbf P_0(\U(1);\u(1))\cong H^1(M;\U(1))$, and the automorphism group of every object
is $\pi_1\mathbf P_0(\U(1);\u(1))\cong H^0(M;\U(1))$.
\end{proof}

\subsection{The splitting problem with prescribed central connection}
\label{subsec:splitting}

Truncating the first two complexes of \eqref{eq:Deltases} at degree one gives
\begin{align*}
  \Delta^\bullet_{\le1} &:= \ \underline{\Delta}\ \to\ 0,\\
  (\tilde{\mathcal G}\times\Usm)^\bullet_{\le1}
    &:= \ \tilde{\mathcal G}\times\Usm \xrightarrow{\frac12\dd\log\eta'} i\Omega^1 \to 0,
\end{align*}
and a short exact sequence of complexes of sheaves
\begin{equation}\label{eq:Deltasestrunc}
  1 \to \Delta^\bullet_{\le1} \to (\tilde{\mathcal G}\times\Usm)^\bullet_{\le1}
    \to (\mathcal G^c)^\bullet \to 1 .
\end{equation}
Removing the terms $i\Omega^{\ge2}$ removes the condition $\dd\gamma_j=0$, so the degree one
cocycles of $(\tilde{\mathcal G}\times\Usm)^\bullet_{\le1}$ glue to pairs consisting of a
principal $\tilde G$-bundle and a principal circle bundle with an arbitrary principal
connection, and $\HH^2(\Delta^\bullet_{\le1})=H^2(M;\Z_2)$ because $\Delta^\bullet_{\le1}$ is
concentrated in degree zero. The proof of Lemma~\ref{lem:gpdLES} applies verbatim to
\eqref{eq:Deltasestrunc} and gives the diagram
\[
\begin{tikzcd}[row sep=3em, column sep=2.6em]
\mathbf H^1(\Delta^\bullet_{\le1})
  \ar[r] \ar[d,"\simeq"']
&
\mathbf H^1\big((\tilde{\mathcal G}\times\Usm)^\bullet_{\le1}\big)
  \ar[r] \ar[d,"\simeq"']
&
\mathbf H^1\big((\mathcal G^c)^\bullet\big)
  \ar[r,"\partial"] \ar[d,"\simeq"']
&
\HH^2(\Delta^\bullet_{\le1})
  \ar[d,"\cong"]
\\
\mathbf P(\mu_2)
  \ar[r]
&
\mathbf P(\tilde G)\times\mathbf P(\U(1);\u(1))
  \ar[r]
&
\mathbf P(G^c;\u(1))
  \ar[r]
&
H^2(M;\Z_2)
\end{tikzcd}
\]
in which the bottom row is \eqref{eq:splitseq} up to a redundant second component
(Lemma~\ref{lem:u1agreement}), the first three vertical arrows are
equivalences of groupoids given by gluing, the first two squares commute up to canonical
natural isomorphism, and the last commutes strictly.

By Lemma~\ref{lem:u1agreement} the class $w(P^c/\U(1),\sigma)$ determines
$e^{i\int \curv\gamma}$, so the obstruction is the pair
$\big(w(\,\cdot\,/\U(1),\sigma),\,e^{i\int \curv}\big)$ in the sequence
\begin{multline}\label{eq:splitseq}
  \mathbf P(\mu_2)
  \xrightarrow{\ (\times_{\mu_2}\tilde G,\ \times_{\mu_2}\U(1))\ }
  \mathbf P(\tilde G)\times\mathbf P(\U(1);\u(1))
  \xrightarrow{\ \times_M/\Delta\ }\\
  \mathbf P(G^c;\u(1))
  \xrightarrow{\ (w(\,\cdot\,/\U(1),\sigma),\ e^{i\int \curv})\ }
  H^2(M;\Z_2)\times H^2(M;\Uloc).
\end{multline}

Let $\mathbf L(P^c,\,\cdot\,/\Delta,\gamma)$ denote the homotopy fibre of $\times_M/\Delta$ in
\eqref{eq:splitseq} over $(P^c,\gamma)$. Its objects are triples
\[
  \Big(\tilde P,\ (Y,\gamma_Y),\
  \hv u:(P^c,\gamma)\xrightarrow{\ \cong\ }\tilde P\times_{\mu_2}(Y,\gamma_Y)\Big),
\]
where $(\tilde P,Y)$ is a lifting of $P^c$ for the Lie group double covering
$\,\cdot\,/\Delta:\tilde G\times\U(1)\to G^c$, where $(Y,\gamma_Y)$ is a principal circle
bundle with connection, and where $\hv u$ is an isomorphism of principal $G^c$-bundles with central
connection. We view these objects as splittings of $(P^c,\gamma)$.

\begin{proposition}\label{prop:splittings}
Let $M$ be a smooth manifold, let $\sigma:\tilde G\to G$ be a Lie group double covering, and
let $\sigma^c:G^c\to G$ be the associated circle extension \eqref{eq:centralext}. Consider the
sequence \eqref{eq:splitseq},
\begin{multline*}
  \mathbf P(\mu_2)
  \xrightarrow{\ (\times_{\mu_2}\tilde G,\ \times_{\mu_2}\U(1))\ }
  \mathbf P(\tilde G)\times\mathbf P(\U(1);\u(1))
  \xrightarrow{\ \times_M/\Delta\ }\\
  \mathbf P(G^c;\u(1))
  \xrightarrow{\ (w(\,\cdot\,/\U(1),\sigma),\ e^{i\int \curv})\ }
  H^2(M;\Z_2)\times H^2(M;\Uloc).
\end{multline*}
Let $(P^c,\gamma)$ be a principal $G^c$-bundle with central connection.
\begin{enumerate}[label=\textup{(\arabic*)}]
\item The sequence \eqref{eq:splitseq} satisfies the exactness condition $(\mathrm{E}1)$ at
$\mathbf P(G^c;\u(1))$ for $(v,w)\in H^2(M;\Z_2)\times H^2(M;\Uloc)$ if and only if $(v,w)$ is
the identity element.

\item The sequence \eqref{eq:splitseq} satisfies the exactness condition $(\mathrm{E}2)$ at
$\mathbf P(\tilde G)\times\mathbf P(\U(1);\u(1))$ for $(P^c,\gamma)$ if and only if
$(P^c,\gamma)$ is isomorphic to the trivial principal $G^c$-bundle with its trivial central
connection. In particular, $\mathbf P(\mu_2)$ is equivalent as a groupoid to the homotopy fibre
$\mathbf L(M\times G^c,\,\cdot\,/\Delta,0)$.

\item The homotopy fibre $\mathbf L(P^c,\,\cdot\,/\Delta,\gamma)$ of $\times_M/\Delta$ at
$(P^c,\gamma)$ is nonempty if and only if $w(P^c/\U(1),\sigma)=0$, and all nonempty homotopy
fibres of $\times_M/\Delta$ are equivalent as groupoids.
\end{enumerate}
\end{proposition}

\begin{proof}
\textup{(1)} The composite
$\big(w(\,\cdot\,/\U(1),\sigma),\,e^{i\int \curv}\big)\circ(\times_M/\Delta)$ sends a pair
$\big(\tilde P,(Y,\gamma_Y)\big)$ to the identity element because $\tilde P$ is a
$\sigma$-lifting of $\big(\tilde P\times_{\mu_2}Y\big)/\U(1)\cong\tilde P/\mu_2$, and the
induced central connection has curvature equal to $\curv\gamma_Y$, which is integral by
Theorem~\ref{thm:kostantweil}. Since $H^2(M;\Z_2)\times H^2(M;\Uloc)$ is discrete, a natural
transformation from this composite to the constant functor at $(v,w)$ exists if and only if
$(v,w)$ is the identity element, so $(\mathrm{E}1)$ fails for every other $(v,w)$. It remains
to show that the induced functor
\[
  \mathbf P(\tilde G)\times\mathbf P(\U(1);\u(1))\longrightarrow
  \hofib{\big(w(\,\cdot\,/\U(1),\sigma),\,e^{i\int \curv}\big)}{\,\id\,}
\]
is essentially surjective. Since the objects of that homotopy fibre are the pairs
$(P^c,\gamma)$ whose obstruction is the identity element, it suffices to construct a splitting
of each such $(P^c,\gamma)$.

Take a good cover $\mathfrak U$ of $M$, so that $(P^c,\gamma)$ has descent data
$(\gamma,g^c)$. Choose representatives $(\tilde g_{jk},\xi_{jk})$ of $g^c_{jk}$ in
$\tilde G\times\U(1)$. Their coboundary projects to $\check\delta g^c=1$, so it takes values in
$\Delta$, and as in Lemma~\ref{lem:u1agreement} its two components agree and define a cocycle
$\varepsilon\in\check C^2(\mathfrak U;\underline{\Z_2})$ representing $w(P^c/\U(1),\sigma)$.
Since that class vanishes, $\varepsilon=\check\delta u$ for some
$u\in\check C^1(\mathfrak U;\underline{\Z_2})$, and replacing
$(\tilde g_{jk},\xi_{jk})$ by $(\tilde g_{jk},\xi_{jk})\,(-1)^{u_{jk}}$ makes the
coboundary trivial. The modified pair still satisfies
$\tfrac12\dd\log\eta'(\tilde g,\xi)=\dd\log\xi=\check\delta\gamma$, because $(-1)^u$ is
locally constant and $\tfrac12\dd\log\eta(g^c)=\check\delta\gamma$ by
\eqref{eq:centralconnection}. So $\big(\gamma,(\tilde g,\xi)\big)$ is a degree one cocycle of
$(\tilde{\mathcal G}\times\Usm)^\bullet_{\le1}$, that is, descent data for a principal
$\tilde G$-bundle $\tilde P$ and a principal circle bundle with connection $(Y,\gamma_Y)$ whose
central product is $(P^c,\gamma)$.

\textup{(2)} The first functor of \eqref{eq:splitseq} sends a principal $\mu_2$-bundle $T$,
with locally constant transition functions $\varepsilon_{jk}$, to the pair
$\big(T\times_{\mu_2}\tilde G,\ T\times_{\mu_2}\U(1)\big)$, whose transition functions are the
pairs $(\varepsilon_{jk},\varepsilon_{jk})$ and whose connection forms are zero. These pairs
lie in $\Delta$, so the central product has transition functions
$[\varepsilon_{jk},\varepsilon_{jk}]=[1,1]$ and zero connection forms. The composite is
therefore constant at the trivial principal $G^c$-bundle with its trivial central connection,
so a natural transformation to the constant functor at $(P^c,\gamma)$ exists exactly when
$(P^c,\gamma)$ is isomorphic to it, and $(\mathrm{E}2)$ fails for every other $(P^c,\gamma)$.
For the trivial bundle the composite provides the natural transformation, and the induced
functor $\mathbf P(\mu_2)\to\mathbf L(M\times G^c,\,\cdot\,/\Delta,0)$ is the equivalence of
\textup{(3)} below, hence full and essentially surjective.

\textup{(3)} The homotopy fibre $\mathbf L(P^c,\,\cdot\,/\Delta,\gamma)$ is nonempty exactly
when $w(P^c/\U(1),\sigma)=0$, by \textup{(1)}. Suppose it is, and let
$\big(\tilde P_1,(Y_1,\gamma_1)\big)$ and $\big(\tilde P_2,(Y_2,\gamma_2)\big)$ be two of its
objects. Over a good cover $\mathfrak U$ they have descent data
$\big(\gamma_1,(\tilde g_1,\xi_1)\big)$ and $\big(\gamma_2,(\tilde g_2,\xi_2)\big)$ above the
same transition functions $g^c$ of $P^c$, so the ratio
\[
  \epsilon_{12}:=\big(\tilde g_2\,\tilde g_1^{-1},\ \xi_2\,\xi_1^{-1}\big)
\]
takes values in $\Delta\cong\mu_2$. It is locally constant because $\Delta$ is discrete, and a
cocycle because $(\tilde g_1,\xi_1)$ and $(\tilde g_2,\xi_2)$ are, so it is descent data for a
principal $\mu_2$-bundle $T_{12}$. Both objects lie above the same descent data
$(\gamma,g^c)$ of $(P^c,\gamma)$, and the central product does not change $\u(1)$-valued local
connection forms, so $\gamma_2=\gamma_1=\gamma$ as local forms. These forms are compatible with
$T_{12}$ since $\dd\log(-1)^{\epsilon_{12}}=0$, and tensoring gives
\[
  \big(\tilde P_1,(Y_1,\gamma_1)\big)\otimes T_{12}
  \;\cong\;\big(\tilde P_2,(Y_2,\gamma_2)\big).
\]

A choice of object $\big(\tilde P_1,(Y_1,\gamma_1)\big)$ therefore gives a functor
\[
  \mathbf L(P^c,\,\cdot\,/\Delta,\gamma)\longrightarrow\mathbf P(\mu_2),
  \qquad
  \big(\tilde P_2,(Y_2,\gamma_2)\big)\longmapsto T_{12},
\]
which is essentially surjective because tensoring $\big(\tilde P_1,(Y_1,\gamma_1)\big)$ with a
principal $\mu_2$-bundle again lies in $\mathbf L(P^c,\,\cdot\,/\Delta,\gamma)$, and fully
faithful because a morphism of $\mathbf L(P^c,\,\cdot\,/\Delta,\gamma)$ covers the identity of
$P^c$ and so is given by a $\Delta$-valued function. So every nonempty homotopy fibre of
$\times_M/\Delta$ is equivalent to $\mathbf P(\mu_2)$, and in particular any two of them are
equivalent.
\end{proof}

\begin{corollary}[Splittings of the structure group]\label{cor:splittings}
Let $M$ be a smooth manifold, let $\sigma:\tilde G\to G$ be a Lie group double covering, and
let $\sigma^c:G^c\to G$ be the associated circle extension \eqref{eq:centralext}. Let
$(P^c,\gamma)$ be a principal $G^c$-bundle with central connection.
\begin{enumerate}[label=\textup{(\arabic*)}]
\item There exist a principal $\tilde G$-bundle $\tilde P$, a principal circle bundle with
connection $(Y,\gamma_Y)$, and an isomorphism of principal $G^c$-bundles with central
connection
\[
  (P^c,\gamma)\;\cong\;\tilde P\times_{\mu_2}(Y,\gamma_Y)
\]
if and only if $w(P^c/\U(1),\sigma)=0$. In that case $\curv\gamma$ is integral.
\item $H^1(M;\Z_2)$ acts freely and transitively on the isomorphism classes of such data.
\end{enumerate}
\end{corollary}

\begin{proof}
The splittings of $(P^c,\gamma)$ are the objects of
$\mathbf L(P^c,\,\cdot\,/\Delta,\gamma)$, the homotopy fibre of $\times_M/\Delta$ at
$(P^c,\gamma)$. It is nonempty exactly when $w(P^c/\U(1),\sigma)=0$, by
Proposition~\ref{prop:splittings}(3). When it is nonempty it is equivalent to
$\mathbf L(M\times G^c,\,\cdot\,/\Delta,0)$ by the same item, hence to $\mathbf P(\mu_2)$ by
Proposition~\ref{prop:splittings}(2). So $\pi_0\mathbf L(P^c,\,\cdot\,/\Delta,\gamma)$ is a
torsor under $\pi_0\mathbf P(\mu_2)\cong H^1(M;\Z_2)$. When such data exist, $\curv\gamma$ is
the curvature of the principal connection $\gamma_Y$, hence integral by
Theorem~\ref{thm:kostantweil}.
\end{proof}

\subsection{The Bockstein relation}
\label{subsec:bockstein}

Let $\sigma:\tilde G\to G$ be a Lie group double covering and let $\sigma^c:G^c\to G$ be its
associated circle extension \eqref{eq:centralext}. We now relate the obstruction classes for
the existence of $\sigma$-liftings, of $\sigma^c$-liftings, and of $\sigma^c$-liftings with
prescribed central curvature $\omega$. Here we view the first two classes as the obstruction
classes of \eqref{eq:classicalseq}, under the identifications
\[
  w(P,\sigma)\in H^2(M;\underline{\mu_2})\cong H^2(M;\Z_2),
  \qquad
  w(P,\sigma^c)\in H^2(M;\Usm)\cong H^3(M;\Z),
\]
and the third is $w(P,\sigma^c,\omega)\in H^2(M;\Uloc)$ of \eqref{eq:wclassdef}. They are
related by the Bockstein homomorphisms for the exact sequences of abelian groups
\begin{equation}\label{eq:coeffmorphism}
\begin{tikzcd}[row sep=2.4em, column sep=3.2em]
0 \ar[r] & \Z \ar[r,hook,"2"] \ar[d,equal] & \Z \ar[r] \ar[d,hook,"{\frac12}"]
  & \Z_2 \ar[r] \ar[d,hook,"{(-1)^\cdot}"] & 0\\
0 \ar[r] & \Z \ar[r,hook] & \R \ar[r,"{\exp(2\pi i\,\cdot\,)}"'] & \U(1) \ar[r] & 1 .
\end{tikzcd}
\end{equation}

\begin{theorem}[Bockstein relations]\label{thm:bockstein}
Let $M$ be a smooth manifold, let $\sigma:\tilde G\to G$ be a Lie group double covering, and
let $\sigma^c:G^c\to G$ be the associated circle extension \eqref{eq:centralext}.
Let $P$ be a principal $G$-bundle over $M$ and let $\omega$ be a closed $2$-form on $M$. Let
\[
  \beta_2:H^k(M;\Z_2)\to H^{k+1}(M;\Z)
  \qquad\text{and}\qquad
  \beta:H^k(M;\Uloc)\to H^{k+1}(M;\Z)
\]
be the Bockstein homomorphisms for the sequences in \eqref{eq:coeffmorphism}. Then under the
identification $H^2(M;\Usm)\cong H^3(M;\Z)$,
\[
  w(P,\sigma^c)\;=\;\beta\big(w(P,\sigma^c,\omega)\big)\;=\;\beta_2\big(w(P,\sigma)\big).
\]
\end{theorem}

\begin{proof}
Since $e^{i\int \omega}$ lifts to the real class $[\tfrac1{2\pi}\omega]\in H^2(M;\R)$, its
Bockstein vanishes, so by \eqref{eq:wclassdef}
\[
  \beta\big(w(P,\sigma^c,\omega)\big)
  =\beta\big((-1)^{w(P,\sigma)}\cdot e^{i\int \omega}\big)
  =\beta\big((-1)^{w(P,\sigma)}\big)
  =\beta_2\big(w(P,\sigma)\big),
\]
the last equality by the commutativity of \eqref{eq:coeffmorphism} and naturality of connecting
homomorphisms. It remains to identify $\beta_2\big(w(P,\sigma)\big)$ with $w(P,\sigma^c)$.

Fix a good cover $\mathfrak U$ of $M$, transition functions $g_{jk}$ for $P$, and lifts
$\tilde g_{jk}$ of $g_{jk}$ to $\tilde G$, so that
$\check\delta\tilde g=(-1)^\varepsilon$ with
$\varepsilon\in\check C^2(\mathfrak U;\underline{\Z_2})$ a cocycle representing $w(P,\sigma)$.
Choose $\tilde\varepsilon\in\check C^2(\mathfrak U;\underline\Z)$ reducing to $\varepsilon$ mod
$2$, which exists because $\mathfrak U$ is a good cover. Then
\begin{equation}\label{eq:bocksteinlift}
  \check\delta\tilde g
  =(-1)^{\varepsilon}
  =(-1)^{\tilde\varepsilon}
  =\exp\big(2\pi i\,\tfrac12\tilde\varepsilon\big).
\end{equation}

Viewing $\tilde\varepsilon$ as an integral lift of $\varepsilon$,
\[
  \Big[\tfrac12\,\check\delta\tilde\varepsilon\Big]=\beta_2\big(w(P,\sigma)\big) .
\]

Viewing $\tfrac12\tilde\varepsilon\in\check C^2(\mathfrak U;\mathcal R)$ as a real lift of
$[1,(-1)^\varepsilon]=\check\delta[\tilde g,1]$ by \eqref{eq:bocksteinlift}, the class
$\big[\tfrac12\,\check\delta\tilde\varepsilon\big]$ in $H^3(M;\Z)$ is the image of
$w(P,\sigma^c)=\big[\check\delta[\tilde g,1]\big]$ under the connecting homomorphism for the
exponential sheaf sequence $0\to\underline\Z\to\mathcal R\to\Usm\to1$.
\end{proof}

If $E$ is a symplectic vector bundle over $M$ with symplectic frame bundle $P$, then
$w(P,\ssigma)=w_2(E)$ is the mod $2$ reduction of $c_1(E)$, so
$w(P,\ssigma^c)=\beta_2\big(w(P,\ssigma)\big)$ vanishes by exactness of
$H^2(M;\Z)\xrightarrow{r}H^2(M;\Z_2)\xrightarrow{\beta_2}H^3(M;\Z)$, and $P$ always admits a
metaplectic-c lifting. In particular every symplectic manifold admits a metaplectic-c structure
\cite[Thm.~6.2]{RobinsonRawnsley}, \cite{ForgerHess}, \cite{Plymen}.

\subsection{Torsion independence of the shifted integrality condition}
\label{subsec:torsion}

When $P$ admits a $\sigma^c$-lifting, there is a real coefficient form of
Corollary~\ref{cor:Pc} that takes the form of a half-integral shifted integrality condition.
Since $w(P,\sigma^c)=\beta_2\big(w(P,\sigma)\big)=0$, there is an
integral class $\widetilde w(P,\sigma)\in H^2(M;\Z)$ whose mod $2$ reduction is $w(P,\sigma)$.
Then
\begin{equation}\label{eq:w2c1pairing}
  (-1)^{\langle w(P,\sigma),S\rangle}
  =(-1)^{\langle \widetilde w(P,\sigma),S\rangle}
  =(-1)^{\langle \widetilde w(P,\sigma)_\R,S_\R\rangle_\R}
  \qquad\text{for all } S\in H_2(M;\Z),
\end{equation}
because the $\Z$-valued evaluation pairing evaluates trivially on the torsion elements of
$H^2(M;\Z)$ and of $H_2(M;\Z)$.

\begin{proposition}\label{prop:hessform}
If $w(P,\sigma)$ is the mod $2$ reduction of an integral cohomology class
$\widetilde w(P,\sigma)\in H^2(M;\Z)$, then $w(P,\sigma^c,\omega)$ is the identity element of
$H^2(M;\Uloc)$ if and only if
\begin{equation}\label{eq:hessform}
  \tfrac{1}{2}\widetilde w(P,\sigma)_{\R}+\Big[\tfrac{1}{2\pi}\,\omega\Big]
  \ \in\ H^2(M;\Z)_\R .
\end{equation}
\end{proposition}

\begin{proof}
By \eqref{eq:wclassdef} and \eqref{eq:w2c1pairing},
\[
  \big\langle w(P,\sigma^c,\omega),S\big\rangle
  = e^{i\int_S\omega}\,(-1)^{\langle w(P,\sigma),S\rangle}
  = \exp\Big(2\pi i\Big[\tfrac12\big\langle \widetilde w(P,\sigma)_\R,S_\R\big\rangle_\R
      +\tfrac{1}{2\pi}\textstyle\int_S\omega\Big]\Big)
\]
for every $S\in H_2(M;\Z)$, so $w(P,\sigma^c,\omega)$ is the identity element if and only if
the bracket lies in $\Z$ for every $S$, which is \eqref{eq:hessform}.
\end{proof}

Any two integral lifts of $w(P,\sigma)$ differ by an element of $2H^2(M;\Z)$, so the
coset
\begin{equation}\label{eq:shiftedlattice}
  \Lambda(P,\sigma):=\tfrac12\widetilde w(P,\sigma)_\R+H^2(M;\Z)_\R\ \subseteq\ H^2_{\dR}(M;\R)
\end{equation}
does not depend on the lift. The condition \eqref{eq:hessform} is equivalent to
$[\tfrac1{2\pi}\omega]\in\Lambda(P,\sigma)$.

\begin{example}[An Enriques surface]\label{ex:enriques}
Let $X$ be an Enriques surface. Its first Chern class $c_1(X)=-K_X$ is nontrivial and
$2$-torsion, so $X$ admits no metaplectic structure. However $c_1(X)_\R=0$, so the
metaplectic-c prequantizability criterion reduces to the Kostant--Weil integrality condition.
\end{example}

\subsection{Unobstructedness of metaplectic-c prequantizability}
\label{subsec:unobstructed}

While every symplectic manifold admits a metaplectic-c structure, the Bockstein relation
(Theorem~\ref{thm:bockstein}) shows that the obstruction classes are different. Thus it remains
to show that every symplectic manifold admits a symplectic form for which it is metaplectic-c
prequantizable.

Recall that the \textbf{symplectic cone} $\mathcal C_M$ in $H^2_{\dR}(M;\R)$ is the subset of
cohomology classes represented by symplectic forms \cite[\S1]{McDuffIsotopy}.

\begin{proposition}\label{prop:universal}
Let $M$ be a symplectic manifold without boundary, with symplectic frame bundle $P$. Then a
cohomology class $a\in H^2_{\dR}(M;\R)$ is represented by $\tfrac1{2\pi}\omega$ for a symplectic
form $\omega$ with $(P,\omega)$ metaplectic-c prequantizable if and only if
$a\in\Lambda(P,\ssigma)\cap\mathcal C_M$. The intersection
$\Lambda(P,\ssigma)\cap\mathcal C_M$ is nonempty for every symplectic manifold $M$, and equals
$\Lambda(P,\ssigma)$ when $M$ is open.
\end{proposition}

\begin{proof}
Treating connected components separately, assume $M$ connected. For any almost complex
structure $J$, $c_1(TM,J)$ is an integral lift of $w(P,\ssigma)=w_2(TM)$, so
$\Lambda(P,\ssigma)$ is the coset $\tfrac12c_1(TM,J)_\R+H^2(M;\Z)_\R$.

By Proposition~\ref{prop:hessform} for $\ssigma^c$, a symplectic form $\omega$ makes
$(P,\omega)$ metaplectic-c prequantizable if and only if
$[\tfrac1{2\pi}\omega]\in\Lambda(P,\ssigma)$. This gives the stated criterion.

If $M$ is closed, $H^2(M;\Z)_\R$ is a full-rank lattice in the finite-dimensional space
$H^2_{\dR}(M;\R)$, since $H^2(M;\Z)/\mathrm{tors}\otimes\R\cong H^2(M;\R)$. The symplectic cone
$\mathcal C_M$ is open, since nondegeneracy is pointwise open and $M$ is compact, and it is
scale invariant. An open scale-invariant subset of $H^2_{\dR}(M;\R)$ meets
every coset of a full-rank lattice, so the intersection is nonempty.

If $M$ is open, Gromov's h-principle carries a nondegenerate $2$-form through nondegenerate
$2$-forms to a symplectic form representing any prescribed class in $H^2_{\dR}(M;\R)$
\cite[18.2.2]{EliashbergMishachev}, \cite[Thm.~7.34]{McDuffSalamon}, so
$\mathcal C_M=H^2_{\dR}(M;\R)$ and the intersection is $\Lambda(P,\ssigma)$, which contains $\tfrac12c_1(TM,J)_\R$.
\end{proof}

The Moser trick compares different representative symplectic forms in the same cohomology
class; for prequantum systems see \cite{MirandaWeitsman}.

\section{EBK quantizability on Lagrangian submanifolds}
\label{sec:ebk}

\subsection{Hess characters and the Guillemin--Sternberg square}
\label{subsec:gssquare}

Let $\bF\subseteq\bV$ be a linear Lagrangian subspace of a symplectic vector space $(\bV,\omega_{\bV})$, and let
$\Sp_{\bF}(\bV)\le\Sp(\bV)$ be the subgroup of linear symplectomorphisms of $\bV$ preserving
$\bF$. For $g\in\Sp_{\bF}(\bV)$ the restricted map $g|_{\bF}:\bF\to\bF$ is linear and invertible,
and restriction preserves composition, so there is a group homomorphism
\begin{equation}\label{eq:inddef}
  \ind:\Sp_{\bF}(\bV)\longrightarrow\GL(\bF),\qquad \ind(g):=g|_{\bF} .
\end{equation}

Since $\omega_{\bV}$ identifies $\bF^\ast$ with $\bV/\bF$, we may choose a splitting
$\bV\cong\bF\oplus\bF^\ast$. If $g\in\ker\ind$, it is a shearing transform
$g(v):=v+g_{\bF\bF^\ast}(v_{\bF^\ast})$ for the splitting $\bF\oplus\bF^\ast$. The restriction
$g|_{\bF}$ is the identity map, and the self-dual map $g_{\bF\bF^\ast}:\bF^\ast\to\bF$ identifies with a
quadratic form in $\Sym^2\bF$. The composition in $\ker\ind$ identifies with the addition of
self-dual maps, $(gh)_{\bF\bF^\ast}=g_{\bF\bF^\ast}+h_{\bF\bF^\ast}$, and the conjugation in $\Sp_{\bF}(\bV)$ normalizing
$\ker\ind$ descends to an action of the linear isomorphisms in $\GL(\bF)$ on $\Sym^2\bF$,
described by the semidirect product decomposition
\begin{equation}\label{eq:spfsemidirect}
  \Sp_{\bF}(\bV)\;\cong\;\GL(\bF)\ltimes\Sym^2\bF ,
  \qquad \ind(g)\cdot h_{\bF\bF^\ast}=\ind(g)\circ h_{\bF\bF^\ast}\circ\ind(g)^\ast=(ghg^{-1})_{\bF\bF^\ast}.
\end{equation}
This decomposition depends on the choice of splitting $\bV\cong\bF\oplus\bF^\ast$. The kernel $\ker\ind$ is
contractible, so $\Sp_{\bF}(\bV)$ is homotopy equivalent to the orthogonal group $\Or(n)$,
where $n=\dim\bF$.
Taking the sign of $\det\circ\,\ind$ gives a map
\begin{equation}\label{eq:sgndet}
  \ori(\,\cdot\,):=\operatorname{sgn}\det(\,\cdot\,|_{\bF}):\Sp_{\bF}(\bV)\longrightarrow\mu_2
\end{equation}
that is trivial on the identity component $\Sp^+_{\bF}(\bV)$, and its induced isomorphism on the
component groups factors through $\pi_0(\Or(n))$ as
$\pi_0(\Sp_{\bF}(\bV))\cong\pi_0(\Or(n))\cong\mu_2$.

Let $\ssigma:\Mp(\bV)\to\Sp(\bV)$ be the connected Lie group double covering with kernel
$\ker\ssigma:=\{1_{\Mp(\bV)},\beps\}$. As in \eqref{eq:centralext}, let
$\Mp^c(\bV):=(\Mp(\bV)\times\U(1))/\Delta$, with associated circle extension
$\ssigma^c:\Mp^c(\bV)\to\Sp(\bV)$ and squaring character $\eta:\Mp^c(\bV)\to\U(1)$. Let
$\Mp_{\bF}(\bV):=\ssigma^{-1}(\Sp_{\bF}(\bV))$ and
$\Mp^c_{\bF}(\bV):=(\ssigma^c)^{-1}(\Sp_{\bF}(\bV))$ be the preimages of $\Sp_{\bF}(\bV)$.

\begin{lemma}[{cf.\ \cite[Lem.~5.3]{BKS}}]\label{lem:lifts}\,
\begin{enumerate}
\item\label{it:liftplus} The inclusion $\Sp^+_{\bF}(\bV)\hookrightarrow\Sp(\bV)$ lifts to a
  unique Lie group homomorphism $\Sp^+_{\bF}(\bV)\hookrightarrow\Mp(\bV)$. Its image is the
  identity component of $\Mp_{\bF}(\bV)$.
\item\label{it:noliftmp} The component group $\pi_0(\Mp_{\bF}(\bV))$ is cyclic with four
  elements, and the inclusion $\Sp_{\bF}(\bV)\hookrightarrow\Sp(\bV)$ does not lift to a Lie
  group homomorphism into $\Mp(\bV)$.
\item\label{it:liftmpc} The inclusion $\Sp_{\bF}(\bV)\hookrightarrow\Sp(\bV)$ lifts to a Lie
  group homomorphism
  $\iota^c:\Sp_{\bF}(\bV)\hookrightarrow\ker\eta^2\le\Mp^c(\bV)$ in exactly two ways.
\end{enumerate}
\end{lemma}

\begin{proof}
(\ref{it:liftplus}) Choose a compatible linear complex structure $J$ on $\bV$. Then $\Sp(\bV)$
deformation retracts to the unitary group $\U(\bV)$ of the induced Hermitian inner product, and
$\Sp^+_{\bF}(\bV)$ deformation retracts onto $\SO(\bF)=\Sp^+_{\bF}(\bV)\cap\U(\bV)$. The complex
determinant on $\U(\bV)$ induces an isomorphism
$\pi_1(\U(\bV))\to\pi_1(\U(1))$ and restricts to the trivial constant map on $\SO(\bF)$, so the
inclusion $\Sp^+_{\bF}(\bV)\hookrightarrow\Sp(\bV)$ induces the trivial map on fundamental
groups. There is therefore a unique lift $\iota:\Sp^+_{\bF}(\bV)\to\Mp(\bV)$ mapping the
identity to the identity. The continuous maps $g\mapsto\iota(g)\iota(g^{-1})$ and
$(g,h)\mapsto\iota(g)\iota(h)\iota((gh)^{-1})$ take values in the discrete group $\ker\ssigma$,
so they are constant on the unique connected component of their domains. Since they evaluate to
$1_{\Mp(\bV)}$ on the identity elements, they evaluate to $1_{\Mp(\bV)}$ everywhere. Thus $\iota$ is a Lie group
homomorphism. Its image and the translate $\beps\,\iota(\Sp^+_{\bF}(\bV))$ are disjoint, open and
connected with union $\ssigma^{-1}(\Sp^+_{\bF}(\bV))$, so each is a connected component, and the
first contains the identity.

(\ref{it:noliftmp}) Let
$\epsilon\in\Or(\bF)\setminus\SO(\bF)\subseteq\Sp_{\bF}(\bV)\setminus\Sp^+_{\bF}(\bV)$ with
$\epsilon^2=1$, so that $\det_\C(\epsilon)=-1$. Take $l:[0,1]\to\U(\bV)$ to be a path with
$l(0)=1$ and $l(1)=\epsilon$ such that $\det_\C(l(t))=e^{i\pi t}$. Then
$\det_\C\big((l\ast(\epsilon l))(t)\big)=e^{2\pi i t}$, so the concatenation $l\ast(\epsilon l)$
generates $\pi_1(\U(\bV))\cong\pi_1(\Sp(\bV))\cong\Z$. Any
$\tilde\epsilon\in\ssigma^{-1}(\epsilon)$ squares to the endpoint of the concatenation of the
lifts of $l$ and $\epsilon l$. Since $\ssigma$ is a connected double covering, that endpoint is
the nontrivial element $\beps$ of $\ker\ssigma$, so $\tilde\epsilon^2=\beps$ and
$\tilde\epsilon$ has order four in $\Mp(\bV)$. By (\ref{it:liftplus}), the covering $\ssigma$
restricts to a trivial Lie group covering over $\Sp^+_{\bF}(\bV)$, and
\[ \pi_0(\Mp_{\bF}(\bV))\;\cong\;\Mp_{\bF}(\bV)/\Sp^+_{\bF}(\bV)
   \;=\;\big\{\Sp^+_{\bF}(\bV),\ \beps\Sp^+_{\bF}(\bV),\
   \tilde\epsilon\Sp^+_{\bF}(\bV),\ \beps\tilde\epsilon\Sp^+_{\bF}(\bV)\big\} \]
is cyclic of order four. A lift $\iota:\Sp_{\bF}(\bV)\to\Mp(\bV)$ that is also a Lie group
homomorphism would send the involution $\epsilon$ to an element of $\ssigma^{-1}(\epsilon)$, all
of which have order four, so no such lifting exists.

(\ref{it:liftmpc}) Consider
\[ \ker\eta^2\;=\;\eta^{-1}(1)\sqcup\eta^{-1}(-1)
   \;=\;\big\{[\tilde g,\xi]\in\Mp^c(\bV)\;:\;\xi\in\mu_4\big\}, \]
on which $\ssigma^c$ restricts to a fourfold Lie group covering of $\Sp(\bV)$. The
nonidentity component $\eta^{-1}(-1)$ consists of elements not in $\Mp(\bV)$. The element
$[\tilde\epsilon,\pm i]$ of $\eta^{-1}(-1)$ is an involution because
$[\tilde\epsilon,\pm i]^2=[\beps,-1]=[1,1]$. Take the lift
$\iota^c:\Sp_{\bF}(\bV)\to\ker\eta^2$ mapping the identity to $1_{\Mp^c(\bV)}$ and $\epsilon$ to
$[\tilde\epsilon,\pm i]$. As in (\ref{it:liftplus}) the maps
$g\mapsto\iota^c(g)\iota^c(g^{-1})$ and
$(g,h)\mapsto\iota^c(g)\iota^c(h)\iota^c((gh)^{-1})$ take values in the discrete kernel of
$\ssigma^c|_{\ker\eta^2}$ and so are constant on connected components, and their values at $1$ and
$\epsilon$, and at $(1,1)$, $(\epsilon,1)$, $(1,\epsilon)$ and $(\epsilon,\epsilon)$, are
$1_{\ker\eta^2}$. Hence $\iota^c$ is a Lie group homomorphism. Two lifts differ by a
$\mu_4$-valued character of $\Sp_{\bF}(\bV)$, which is either trivial or
$\operatorname{sgn}\det(\,\cdot\,|_{\bF})$.
\end{proof}
There exists a \textbf{Hess character} $\chi$ making the square
\begin{equation}\label{eq:chihessdiagram}
\begin{tikzcd}[row sep=2.6em, column sep=3.2em, ampersand replacement=\&]
\Mp_{\bF}(\bV) \ar[r,"\chi"] \ar[d,"\ssigma"']
  \ar[dr, phantom, "\lrcorner", very near start]
  \& \R^\times\sqcup i\R^\times \ar[d,"\sq"] \\
\Sp_{\bF}(\bV) \ar[r,"\det(\,\cdot\,|_{\bF})"'] \& \R^\times
\end{tikzcd}
\end{equation}
a pullback of $\Z_2$-extensions \cite[Thm.~3.5]{Hess}, so that
\begin{equation}\label{eq:chihesssq}
  \chi(\tilde g)^2=\det\big(\ssigma(\tilde g)|_{\bF}\big),
  \qquad
  \chi(\beps)=-1 .
\end{equation}
The choice of $\chi$ is unique up to complex conjugation, induced by composition
with the nontrivial automorphism of either the Lie group double covering
$\Mp_{\bF}(\bV)\to\Sp_{\bF}(\bV)$ or of $\R^\times\sqcup i\R^\times\to\R^\times$. The modulus
$|\chi|=|\det(\,\cdot\,|_{\bF})|^{1/2}$ is the transition factor for
half-densities. The normalization
$\hat\chi:=\chi/|\chi|$ takes values in $\mu_4$ and
$\hat\chi(\tilde g)^2=\operatorname{sgn}\det\big(\ssigma(\tilde g)|_{\bF}\big)=\ori(\ssigma(\tilde g))$.

For the associated circle extension, there is also a \textbf{Hess character}
\begin{equation}\label{eq:chicdef}
  \chi^c([\tilde g,\xi])\;:=\;\xi\,\chi(\tilde g),
\end{equation}
making the square
\begin{equation}\label{eq:chichessdiagram}
\begin{tikzcd}[row sep=2.8em, column sep=8em, ampersand replacement=\&]
\Mp^c_{\bF}(\bV) \ar[r,"\chi^c"] \ar[d,"{(\eta,\,\ssigma^c)}"']
  \ar[dr, phantom, "\lrcorner", very near start]
  \& \C^\times \ar[d,"\sq"] \\
\U(1)\times\Sp_{\bF}(\bV) \ar[r,"{(z,g)\,\mapsto\, z\det(g|_{\bF})}"'] \& \C^\times
\end{tikzcd}
\end{equation}
a pullback of $\Z_2$-extensions \cite[Thm.~3.5, (89), (90)]{Hess}. This Hess character
satisfies
\begin{equation}\label{eq:hess89}
  \chi^c([\tilde g,\xi])^2 \;=\; \xi^2\det\big(\ssigma(\tilde g)|_{\bF}\big),
  \qquad
  \big|\chi^c([\tilde g,\xi])\big|^2 \;=\; \big|\det\big(\ssigma(\tilde g)|_{\bF}\big)\big|.
\end{equation}

The choice of lift in \eqref{eq:chichessdiagram} is unique up to precomposition with the
unique nontrivial
involutive automorphism $\alpha$ of the Lie group double covering $(\eta,\ssigma^c)$,
\begin{equation}\label{eq:gsflip}
  \alpha([\tilde g,\xi])\;:=\;\big[\tilde g,\;\ori\big(\ssigma(\tilde g)\big)\,\xi\big],
\end{equation}
given by the unique nontrivial Lie group homomorphism
$\ori\circ\pr_2:\U(1)\times\Sp_{\bF}(\bV)\to\mu_2$.

The polar decomposition gives the identities
\begin{equation}\label{eq:chisplit}
  \big|\chi^c([\tilde g,\xi])\big| = \big|\det\big(\ssigma(\tilde g)|_{\bF}\big)\big|^{1/2},
  \qquad
  \hat\chi^c([\tilde g,\xi])^2 = \xi^2\,\ori\big(\ssigma(\tilde g)\big).
\end{equation}

Over $\Sp_{\bF}(\bV)$ the circle extension $\ssigma^c:\Mp^c_{\bF}(\bV)\to\Sp_{\bF}(\bV)$ is
trivializable, because
\begin{equation}\label{eq:mpcfsplit}
  \big(\ssigma^c,\,\hat\chi^c\big):
  \Mp^c_{\bF}(\bV)\;\xrightarrow{\ \cong\ }\;\Sp_{\bF}(\bV)\times\U(1)
\end{equation}
is an isomorphism of Lie groups. The normalized Hess character $\hat\chi^c$ is
the projection to the second factor after the trivialization, so its kernel is isomorphic to
$\Sp_{\bF}(\bV)$. Since $\ker\hat\chi^c\le\Mp^c_{\bF}(\bV)\cap\ker\eta^2$, the
kernel $\ker\hat\chi^c$ identifies with a specific choice of lift of
$\Sp_{\bF}(\bV)$ in $\Mp^c(\bV)$ in Lemma~\ref{lem:lifts}(\ref{it:liftmpc}). The two choices of
Hess character, $\hat\chi^c$ and $\hat\chi^c\circ\alpha$, give
two kernels, each identifying with a lift of $\Sp_{\bF}(\bV)$ to $\Mp^c(\bV)$ in
Lemma~\ref{lem:lifts}(\ref{it:liftmpc}). The intersection of $\Mp_{\bF}(\bV)$ with either kernel
is the unique lift of $\Sp^+_{\bF}(\bV)$ in $\Mp_{\bF}(\bV)$ of
Lemma~\ref{lem:lifts}(\ref{it:liftplus}), so that each $\ker\hat\chi^c$ coset in
$\Mp^c_{\bF}(\bV)$ contains at most one connected component of $\Mp_{\bF}(\bV)$.

The isomorphism \eqref{eq:mpcfsplit} descends to an isomorphism of Lie groups
\begin{equation}\label{eq:mpcfdescend}
  \Mp^c_{\bF}(\bV)/\Sp^+_{\bF}(\bV)\;\cong\;
  \big(\big(\Mp_{\bF}(\bV)/\Sp^+_{\bF}(\bV)\big)\times\U(1)\big)\big/\Delta
  \;\cong\;(\mu_4\times\U(1))/\Delta\;\cong\;\mu_2\times\U(1),
\end{equation}
where the last isomorphism maps $(\xi,z)\mapsto(\xi^2,\xi z)$. The first component of the image
in $\mu_2\times\U(1)$ is
$\hat\chi(\tilde g)^2=\operatorname{sgn}\det(\ssigma(\tilde g)|_{\bF})$, and
reads whether $[\tilde g,\xi]$ lies over the orientation preserving or the orientation
reversing component of $\Sp_{\bF}(\bV)/\Sp^+_{\bF}(\bV)$. The automorphism $\alpha$ descends to
an automorphism of the circle extension
$\Mp^c_{\bF}(\bV)/\Sp^+_{\bF}(\bV)\to\Sp_{\bF}(\bV)/\Sp^+_{\bF}(\bV)$ that fixes the circle over
the identity element and rotates the circle over the nonidentity element by $\pi$.

The isomorphism \eqref{eq:mpcfdescend} further restricts to
\begin{align*}
  \Mp_{\bF}(\bV)/\Sp^+_{\bF}(\bV)
  &\;\cong\;\big(\big(\Mp_{\bF}(\bV)/\Sp^+_{\bF}(\bV)\big)\times\mu_2\big)\big/\Delta
   \;\cong\;(\mu_4\times\mu_2)/\Delta \\
  &\;\cong\;\{(1,1),\,(1,-1),\,(-1,i),\,(-1,-i)\}\;\le\;\mu_2\times\U(1) ,
\end{align*}
and the automorphism $\alpha$ descended to \eqref{eq:mpcfdescend} extends the nontrivial
automorphism of the Lie group double covering $\Mp_{\bF}(\bV)\to\Sp_{\bF}(\bV)$.

The normalized Hess character $\hat\chi^c$ descends to a Lie group double
covering $\Mp^c_{\bF}(\bV)/\Sp^+_{\bF}(\bV)\to\U(1)$, which factors through the quotient
$\Mp^c_{\bF}(\bV)/\Sp^+_{\bF}(\bV)\to
\Mp^c_{\bF}(\bV)/\ker\hat\chi^c\cong\U(1)$. This double covering maps
$\Mp_{\bF}(\bV)/\Sp^+_{\bF}(\bV)$ injectively, because each
$\ker\hat\chi^c$ coset in $\Mp^c_{\bF}(\bV)$ contains at most one connected
component of $\Mp_{\bF}(\bV)$. The image of $\Mp_{\bF}(\bV)/\Sp^+_{\bF}(\bV)$ can be viewed as
the \textbf{Guillemin--Sternberg square} \cite[V\S4, p.~260]{GuilleminSternberg}, with the four
vertices
corresponding to the four connected components of $\Mp_{\bF}(\bV)/\Sp^+_{\bF}(\bV)$, inscribed
in $\Mp^c_{\bF}(\bV)/\ker\hat\chi^c$. The choice of Hess character determines
two circles $\Mp^c_{\bF}(\bV)/\ker\hat\chi^c$ and
$\Mp^c_{\bF}(\bV)/\ker(\hat\chi^c\circ\alpha)$ inscribing the
Guillemin--Sternberg square.

\begin{figure}[htbp]\centering
\begin{tikzpicture}[>=stealth,font=\scriptsize,
   dot/.style={circle,fill,inner sep=1.3pt},
   odot/.style={circle,fill,inner sep=1.3pt},
   ell/.style={black},
   starmark/.style={star,star points=5,star point ratio=2.1,fill,minimum size=6pt,
            inner sep=0pt,outer sep=0pt},
   sel/.style={gray,dashed}]
  \def\cx{4.8}\def\ry{3.3}
  % ---------- row 1
  \begin{scope}
    \draw[ell] (0,0) ellipse [x radius=0.85, y radius=0.42];
    \node[dot,label={[inner sep=2pt]left:$1$}] at (0.85,0) {}; \node[dot] at (-0.85,0) {};
  \end{scope}
  \begin{scope}[shift={(\cx,0)}]
    \draw[ell] (0,0) ellipse [x radius=0.42, y radius=0.85];
    \node[dot] at (0,0.85) {}; \node[dot] at (0,-0.85) {};
    \node[starmark] at (0.42,0) {};
    \draw[->,gray] (-0.241,0.924)
      arc[start angle=115, end angle=147, x radius=0.57, y radius=1.02];
    \node[gray,font=\scriptsize] at (-0.604,1.132) {$\pi$};
    \draw[->,gray] (0.241,-0.924)
      arc[start angle=295, end angle=327, x radius=0.57, y radius=1.02];
    \node[gray,font=\scriptsize] at (0.604,-1.132) {$\pi$};
  \end{scope}
  \draw[sel] (0.85,0) -- node[below,pos=0.4,text=black,inner sep=3pt]
    {$\cong\Sp_{\bF}(\bV)/\Sp^+_{\bF}(\bV)$} (\cx+0.42,0);
  \node[draw,rounded corners,inner sep=10pt,fit={(-1.3,-1.35) (\cx+1.05,1.35)},
        label={above:$\Mp^c_{\bF}(\bV)/\Sp^+_{\bF}(\bV)$}] (topbox) {};
  \node at ($(topbox.south)+(0,0.34)$) {$\Mp_{\bF}(\bV)/\Sp^+_{\bF}(\bV)$};
  \begin{scope}[shift={(2*\cx+1.1,0)}]
    \fill[black!8] (1.05,0) -- (0,1.05) -- (-1.05,0) -- (0,-1.05) -- cycle;
    \draw[ell] (0,0) circle [radius=1.05];
    \node[dot] at (0,1.05) {};  \node[dot] at (0,-1.05) {};
    \node[dot,label={[inner sep=2pt]right:$1$}] at (1.05,0) {};  \node[dot] at (-1.05,0) {};
  \end{scope}
  \draw[->] ($(topbox.east)+(0.15,0)$) -- node[above]{$\hat\chi^c$}
       node[below]{$2{:}1$} (2*\cx+1.1-1.4,0);
  % ---------- row 2
  \node[odot,label={[inner sep=2pt]left:$1$}] (b1) at (0,-\ry) {};
    \node[odot] (b2) at (\cx,-\ry) {};
    \node[below=6pt] (l1) at (b1) {orientation preserving};
    \node[below=6pt] (l2) at (b2) {orientation reversing};
    \node[draw,rounded corners,inner sep=5pt,
      fit={([yshift=11pt]topbox.west |- b1) ([yshift=11pt]topbox.east |- b1)
           (l1) (l2)}] (botbox) {};
  \node[below=3pt] at (botbox.south) {$\Sp_{\bF}(\bV)/\Sp^+_{\bF}(\bV)$};
  \node[align=center] at (2*\cx+1.1,-1.85)
    {The Guillemin--Sternberg square\\ inscribed in
     $\Mp^c_{\bF}(\bV)/\Sp_{\bF}(\bV)\cong\U(1)$};
  % ---------- arrows
  \draw[sel] (b1) -- (b2);
  \draw[->,shorten <=5pt,shorten >=5pt] (topbox.south -| b1) -- (botbox.north -| b1);
  \draw[->,shorten <=5pt,shorten >=5pt] (topbox.south -| b2) -- (botbox.north -| b2);
\end{tikzpicture}
\caption{Inscribing the Guillemin--Sternberg square. The quotient group
$\Mp^c_{\bF}(\bV)/\Sp^+_{\bF}(\bV)\cong\Z_2\times\U(1)$ consists of two circles, one for each
connected component of $\Sp_{\bF}(\bV)$. The four black dots in the upper box represent the
four connected components of $\Mp_{\bF}(\bV)$. The character $\hat\chi^c$
descends to $\Mp^c_{\bF}(\bV)/\Sp^+_{\bF}(\bV)$, and its descent factors through the
two-sheeted covering
$\Mp^c_{\bF}(\bV)/\Sp^+_{\bF}(\bV)\to
\Mp^c_{\bF}(\bV)/\ker\hat\chi^c\xrightarrow{\,\cong\,}\U(1)$,
mapping the subgroup $\Mp_{\bF}(\bV)/\Sp^+_{\bF}(\bV)$ injectively. The kernel
$\ker\hat\chi^c$ is isomorphic to $\Sp_{\bF}(\bV)$, and its two components are
joined by a dashed line, where the star marks the component not in $\Mp_{\bF}(\bV)$. The
involutive automorphism of \eqref{eq:gsflip} descends to an order two rotation of the
nonidentity component of $\Mp^c_{\bF}(\bV)/\Sp^+_{\bF}(\bV)$, which restricts to the
nontrivial automorphism of $\Mp_{\bF}(\bV)/\Sp^+_{\bF}(\bV)$, and exchanges the two
isomorphisms $\pi_0(\Mp_{\bF}(\bV))\cong\mu_4$, the two normalized Hess characters, the two
lifts of $\Sp_{\bF}(\bV)$ to $\Mp^c(\bV)$, the two orientations of the Guillemin--Sternberg
square \cite[V\S4, p.~260]{GuilleminSternberg}, and
the two choices of circles that inscribe it.}
\label{fig:gssquare}
\end{figure}
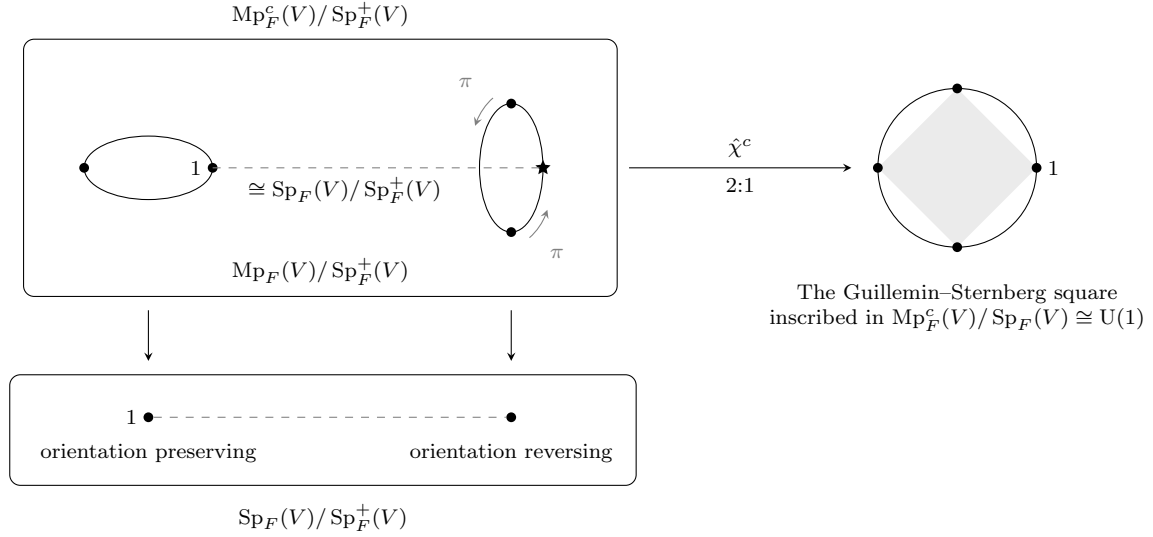
\subsection{Character identities}
\label{subsec:characters}

From Lemma~\ref{lem:lifts} we know $\pi_0(\Mp_{\bF}(\bV))\cong\mu_4$. Viewing $\mu_4$ as a
subgroup of $\U(1)$, we denote a choice of character
\[ \Mp_{\bF}(\bV)\longrightarrow\Mp_{\bF}(\bV)/\Sp^+_{\bF}(\bV)
   \xrightarrow{\ \cong\ }\pi_0(\Mp_{\bF}(\bV))\xrightarrow{\ \cong\ }\mu_4 \]
by $\widetilde{\ori}$, because the square
\begin{equation}\label{eq:sqrtordiagram}
\begin{tikzcd}[row sep=2.6em, column sep=3.2em, ampersand replacement=\&]
\Mp_{\bF}(\bV) \ar[r,"\widetilde{\ori}"] \ar[d,"\ssigma"']
  \ar[dr, phantom, "\lrcorner", very near start]
  \& \mu_4 \ar[d,"\sq"] \\
\Sp_{\bF}(\bV) \ar[r,"\ori"'] \& \mu_2
\end{tikzcd}
\end{equation}
commutes and is a pullback of $\Z_2$-extensions.

Normalizing by the modulus gives a canonical bijection between the two lifts of
$\det(\,\cdot\,|_{\bF})$ in \eqref{eq:chihessdiagram} and the two lifts of $\ori$ in
\eqref{eq:sqrtordiagram}. We denote the corresponding lifts so that
\begin{equation}\label{eq:chisqrtor}
  \hat\chi\;=\;\widetilde{\ori} .
\end{equation}

\begin{lemma}\label{lem:chicrestricts}
The normalized Hess character $\hat\chi^c$ restricts to $\hat\chi=\widetilde{\ori}$ on
$\ker\eta|_{\Mp^c_{\bF}(\bV)}\cong\Mp_{\bF}(\bV)$.
\end{lemma}

\begin{proof}
Equation \eqref{eq:chicdef} gives $\hat\chi^c([\tilde g,1])=\hat\chi(\tilde g)$, and
$\ker\eta|_{\Mp^c_{\bF}(\bV)}=\{[\tilde g,\xi]:\tilde g\in\Mp_{\bF}(\bV),\ \xi\in\mu_2\}
=\{[\tilde g,1]:\tilde g\in\Mp_{\bF}(\bV)\}\cong\Mp_{\bF}(\bV)$, since
$[\tilde g,-1]=[\tilde g\beps,1]$.
\end{proof}

In \cite{BKS,GuilleminSternberg} a character for the metalinear group is constructed from the
fact that $\GL(\bF\otimes\C)$ is connected with $\pi_1(\GL(\bF\otimes\C))\cong\Z$. There is
therefore a unique connected Lie group double covering $\ML(\bF\otimes\C)$ of
$\GL(\bF\otimes\C)$, realized as the fibre product of $\det_\C$ with $\sq:\C^\times\to\C^\times$,
and it comes with a unique lift $\det_\C^{1/2}:\ML(\bF\otimes\C)\to\C^\times$. The metalinear
group $\ML(\bF)$ is the preimage of $\GL(\bF)$ in $\ML(\bF\otimes\C)$, and
$\chi_{\mathrm{BKS}}$ is the restriction of $\det_\C^{1/2}$ to it. It fits in the pullback
\begin{equation}\label{eq:mldiagram}
\begin{tikzcd}[row sep=2.6em, column sep=3.2em, ampersand replacement=\&]
\ML(\bF) \ar[r,"\chi_{\mathrm{BKS}}"] \ar[d]
  \ar[dr, phantom, "\lrcorner", very near start]
  \& \R^\times\sqcup i\R^\times \ar[d,"\sq"] \\
\GL(\bF) \ar[r,"\det"'] \& \R^\times
\end{tikzcd}
\end{equation}
so it can also be described as the fibre product
\[ \ML(\bF)\;=\;\{(g,z)\in\GL(\bF)\times\C^\times:\ z^2=\det g\},
   \qquad \chi_{\mathrm{BKS}}(g,z)=z. \]
The condition $z^2=\det g$ places $z$ in $\R^\times\sqcup i\R^\times$, so that $\ML(\bF)$ has
four connected components \cite[p.~253]{GuilleminSternberg}.

\begin{lemma}\label{lem:indtilde}
There exists a homomorphism $\widetilde{\ind}$ making
\begin{equation}\label{eq:indlift}
\begin{tikzcd}[row sep=2.6em, column sep=3.2em, ampersand replacement=\&]
\Mp_{\bF}(\bV) \ar[r,"\widetilde{\ind}"] \ar[d,"\ssigma"']
  \& \ML(\bF) \ar[d] \\
\Sp_{\bF}(\bV) \ar[r,"\ind"'] \& \GL(\bF)
\end{tikzcd}
\end{equation}
commute, with $\chi_{\mathrm{BKS}}\circ\widetilde{\ind}=\chi$. The normalized characters satisfy
$\hat\chi_{\mathrm{BKS}}\circ\widetilde{\ind}=\widetilde{\ori}$.
\end{lemma}

\begin{proof}
Consider the map
\[ \widetilde{\ind}(\tilde g)\;:=\;\big(\ind(\ssigma(\tilde g)),\,\chi(\tilde g)\big). \]
By \eqref{eq:chihesssq} its components satisfy
$\chi(\tilde g)^2=\det\big(\ind(\ssigma(\tilde g))\big)$, so the image of $\widetilde{\ind}$
lies in $\ML(\bF)$. Both components are homomorphisms, the square commutes by construction, and
$\chi_{\mathrm{BKS}}\circ\widetilde{\ind}=\chi$ because $\chi_{\mathrm{BKS}}(g,z)=z$. Taking
normalizations and applying Lemma~\ref{lem:chicrestricts} gives the last identity.
\end{proof}

Let
\begin{equation}\label{eq:tautbundle}
  \tau_{\bV}\;:=\;\{(v,\Lambda):v\in\Lambda\}\;\subseteq\;\bV\times\Lag(\bV)
\end{equation}
be the tautological bundle over the Lagrangian Grassmannian, and let
$\mathfrak B(\tau_{\bV})$ be its frame bundle, a principal $\GL(\bF)$-bundle. Since $\omega_{\bV}$
identifies $\bV/\tau_{\bV}$ with $\tau_{\bV}^\ast$, we have
$\tau_{\bV}\oplus\tau_{\bV}^\ast\cong\Lag(\bV)\times\bV$, and the degree two part of
$w(\tau_{\bV})^2=w(\Lag(\bV)\times\bV)=1$ gives $w_1(\tau_{\bV})^2=0$. So
$\mathfrak B(\tau_{\bV})$ lifts to a principal $\ML(\bF)$-bundle
$\widetilde{\mathfrak B}(\tau_{\bV})$ \cite[p.~259]{GuilleminSternberg}, the bundle of
\emph{metalinear frames} \cite{BKS,Blattner,BlattnerNonReal}.

\begin{theorem}\label{thm:threeways}
For a consistent choice of primitive fourth root of unity in \eqref{eq:mldiagram} and
\eqref{eq:chicdef}, there are isomorphisms of flat principal $\U(1)$-bundles over the
Lagrangian Grassmannian $\Lag(\bV)$,
\begin{equation}\label{eq:threeways}
  \Lag(\bV;\mu_4)\times_{\mu_4}\U(1)
  \;\cong\;\widetilde{\mathfrak B}(\tau_{\bV})\times_{\ML(\bF)}\U(1)
  \;\cong\;\Mp^c(\bV)\times_{\Mp^c_{\bF}(\bV)}\U(1),
\end{equation}
the last two associated along $\hat\chi_{\mathrm{BKS}}$ and $\hat\chi^c$. Their holonomy over
the generator of $\pi_1(\Lag(\bV))$ is a primitive fourth root of unity.
\end{theorem}

\begin{proof}
Take the principal $\Mp_{\bF}(\bV)$-bundle $\Mp(\bV)\to\Lag(\bV)\cong\Mp(\bV)/\Mp_{\bF}(\bV)$. A
composition of associated bundle constructions is the associated bundle of the composition, so
the three factorizations of \eqref{eq:threecharacters} give
\begin{align*}
  \Mp(\bV)\times_{\Mp_{\bF}(\bV)}\U(1)
  &\;\cong\;\big(\Mp(\bV)\times_{\widetilde{\ori}}\mu_4\big)\times_{\mu_4}\U(1) \\
  &\;\cong\;\big(\Mp(\bV)\times_{\widetilde{\ind}}\ML(\bF)\big)
     \times_{\hat\chi_{\mathrm{BKS}}}\U(1) \\
  &\;\cong\;\big(\Mp(\bV)\times_{\Mp_{\bF}(\bV)}\Mp^c_{\bF}(\bV)\big)
     \times_{\hat\chi^c}\U(1) .
\end{align*}
For a principal $G$-bundle, taking the associated bundle along a character is the same as taking
the orbit space by its kernel \cite[Prop.~5.5]{KobayashiNomizu}. The intermediate factor of the
first line is therefore
$\Mp(\bV)\times_{\widetilde{\ori}}\mu_4\cong\Mp(\bV)/\Sp^+_{\bF}(\bV)\cong\Lag(\bV;\mu_4)$, since the
kernel of $\widetilde{\ori}$ is the identity component of $\Mp_{\bF}(\bV)$, which by
Lemma~\ref{lem:lifts}(\ref{it:liftplus}) is the lift of $\Sp^+_{\bF}(\bV)$. The intermediate
factors of the second and third lines are
\[ \Mp(\bV)\times_{\widetilde{\ind}}\ML(\bF)\;\cong\;\widetilde{\mathfrak B}(\tau_{\bV}),
   \qquad
   \Mp(\bV)\times_{\Mp_{\bF}(\bV)}\Mp^c_{\bF}(\bV)\;\cong\;\Mp^c(\bV), \]
by extension of structure groups.

They are flat because $\widetilde{\ori}$ is constant on connected components, so the transition
functions are locally constant. The holonomy along a loop is the monodromy of the fourfold
covering $\Lag(\bV;\mu_4)\to\Lag(\bV)$ along that loop. Since this covering is connected, the
monodromy $\pi_1(\Lag(\bV))\cong\Z\to\mu_4$ is surjective, so it sends a generator of
$\pi_1(\Lag(\bV))$ to a generator of $\mu_4$, which is a primitive fourth root of unity.
\end{proof}

\subsection{Circle bundles over the Lagrangian Grassmannian bundle}
\label{subsec:ebklag}

Let $(M,\omega)$ be a metaplectic-c prequantizable symplectic manifold
(Corollary~\ref{cor:Pc}), and let $(P^c,\bu,\gamma)$ be a metaplectic-c prequantization
of $M$. Choose a normalized Hess character $\hat\chi^c$ and identify $\Sp_{\bF}(\bV)$ with a subgroup of
$\Mp^c_{\bF}(\bV)$ by its kernel, using Lemma~\ref{lem:lifts}(\ref{it:liftmpc}). Then we can take
associated bundles
\begin{equation}\label{eq:globaltower}
\begin{aligned}
  \Lag(TM) \;&:=\;
     P^c\times_{\Mp^c(\bV)}\big(\Mp^c(\bV)/\Mp^c_{\bF}(\bV)\big)
    &&\cong\; P\times_{\Sp(\bV)}\Lag(\bV), \\
  Y \;&:=\; P^c\times_{\Mp^c_{\bF}(\bV)}\big(\Mp^c_{\bF}(\bV)/\Sp_{\bF}(\bV)\big)
    &&\cong\; P^c\times_{\hat\chi^c}\U(1).
\end{aligned}
\end{equation}
By \cite[Prop.~5.5]{KobayashiNomizu}, $\Lag(TM)\cong P^c/\Mp^c_{\bF}(\bV)$ and
$Y\cong P^c/\Sp_{\bF}(\bV)$, so that $Y\to\Lag(TM)$ is a principal
$\Mp^c_{\bF}(\bV)/\Sp_{\bF}(\bV)\cong\U(1)$-bundle. Let $\pi:\Lag(TM)\to M$ be the projection and
$\tau_{TM}$ be its tautological bundle.

\begin{proposition}\label{prop:gammadescends}
The central connection $\gamma$ descends to a principal $\U(1)$-connection $\tilde\gamma$ on
$Y\to\Lag(TM)$, whose curvature form is $\pi^\ast\omega$, and its holonomy on the generators of
each $\pi_1(\Lag(T_xM))$ is a primitive fourth root of unity.
\end{proposition}

\begin{proof}
A connection descends along the principal $\Sp_{\bF}(\bV)$-action of
$P^c\to Y\cong P^c/\Sp_{\bF}(\bV)$ when it is invariant and horizontal for that action. The
first condition of Definition~\ref{def:centralconnection} gives invariance, since $\U(1)$ is
central and $\Ad$ is trivial. The vertical fields of the action are the fundamental vector fields
$X_{\bxi}$ for $\bxi\in\mathfrak{sp}_{\bF}$, the Lie algebra of $\Sp_{\bF}(\bV)$. Since
$\Sp_{\bF}(\bV)\le\ker\eta^2$ by Lemma~\ref{lem:lifts}(\ref{it:liftmpc}), $\eta$ takes values in
$\mu_2$ on $\Sp_{\bF}(\bV)$ and $\dd_1\eta$ vanishes on $\mathfrak{sp}_{\bF}$. By the second
condition of Definition~\ref{def:centralconnection},
$\gamma(X_{\bxi})=\tfrac12\dd_1\eta(\bxi)=0$. So $\gamma$ descends to a $\u(1)$-valued $1$-form
$\tilde\gamma$ on $Y$. This form is a principal connection for $Y\to\Lag(TM)$ because it is
invariant under the central circle action and $\tfrac12\dd_1\eta|_{\u(1)}=\id$ by
\eqref{eq:mudef}. The identity $\dd\gamma=i(\rho^c)^\ast\omega$ descends to
$\dd\tilde\gamma=i\,\pi^\ast\omega$. On a fibre $\Lag(T_xM)$ of $\pi$ this curvature vanishes,
and $Y$ restricts there to $\Mp^c(\bV)\times_{\hat\chi^c}\U(1)$, so the fibrewise holonomy is
that of Theorem~\ref{thm:threeways}.
\end{proof}

The following circle bundles over $\Lag(TM)$ are defined:
\[ Y\;=\;Y_{\hat\chi^c}\;:=\;P^c\times_{\hat\chi^c}\U(1),
   \qquad Y_{\ori}\;:=\;P\times_{\ori}\U(1),
   \qquad \pi^\ast Y_\eta\;\cong\;P^c\times_{\eta|_{\Mp^c_{\bF}(\bV)}}\U(1), \]
where $Y_\eta:=P^c\times_\eta\U(1)$ is the circle bundle over $M$ associated to the character
$\eta$. Since $\eta$ squares the central circle, $Y_\eta$ has a principal connection modelled by
$\dd_1\eta$, of curvature form $2\omega$. The bundle $Y_{\ori}\to\Lag(TM)$ has transition functions
valued in $\mu_2$, so it has a flat connection.

If $M$ admits a metaplectic structure $\tilde P\to P\to M$, then by
Theorem~\ref{thm:threeways} the bundles
\[ Y_{\hat\chi_{\mathrm{BKS}}\circ\widetilde{\ind}}
   \;=\;Y_{\hat\chi}\;=\;Y_{\widetilde{\ori}} \]
are defined as associated bundles of the principal $\Mp_{\bF}(\bV)$-bundle
$\tilde P\to\Lag(TM)$. The bundle $Y_{\widetilde{\ori}}\to\Lag(TM)$ has transition functions valued in
$\mu_4$, so it has a flat connection.

The tensor product relations of line bundles on $M$
(cf.\ \cite[p.~58]{CzyzMaslov}, \cite[V\S6, p.~273]{GuilleminSternberg},
\cite[Def.~1.18, Prop.~1.19]{Hess}, \cite[\S6, \S8]{RobinsonRawnsley}) hold on $\Lag(TM)$.
Here, for principal circle bundles with connection $(Y_1,\gamma_1)$ and $(Y_2,\gamma_2)$ over
the same base $B$, the \textbf{Picard product} is
\[
  (Y_1,\gamma_1)\otimes(Y_2,\gamma_2)
  \;:=\;\big((Y_1\times_B Y_2)/\U(1),\ \gamma_1+\gamma_2\big),
\]
where $\U(1)$ acts by $(y_1,y_2)\cdot z=(y_1z,\,y_2z^{-1})$ and $\gamma_1+\gamma_2$ denotes the
descent of $\mathrm{pr}_1^\ast\gamma_1+\mathrm{pr}_2^\ast\gamma_2$. It corresponds to the tensor
product of the associated hermitian line bundles; curvatures add and holonomies multiply. For a
principal $H$-bundle $Q$ with connection and characters $\chi_1,\chi_2:H\to\U(1)$, with the
induced connections on the associated bundles,
\begin{equation}\label{eq:picardchar}
  Q\times_{\chi_1\chi_2}\U(1)\;\cong\;\big(Q\times_{\chi_1}\U(1)\big)\otimes
  \big(Q\times_{\chi_2}\U(1)\big),
  \qquad [q,z_1z_2]\longmapsto\big[[q,z_1],[q,z_2]\big].
\end{equation}

\begin{proposition}\label{prop:tensorfactor}
Let $(P^c,\bu,\gamma)$ be a metaplectic-c prequantization over a symplectic manifold $M$. As
circle bundles with connection over $\Lag(TM)$,
\begin{equation}\label{eq:tensorfactor}
  Y^{\otimes2}\;\cong\;\pi^\ast Y_\eta\;\otimes\;Y_{\ori}.
\end{equation}
If $w_2(TM)=0$, then $(P^c,\gamma)$ splits in the sense of
Corollary~\ref{cor:splittings}, as
\begin{equation}\label{eq:mpreduction}
  (P^c,\gamma)\;\cong\;\tilde P\times_{\mu_2}(Y_{\mathrm{KS}},\gamma_{\mathrm{KS}})
\end{equation}
for a metaplectic structure $\tilde P$ of $M$ and a Kostant--Souriau prequantization
$(Y_{\mathrm{KS}},\gamma_{\mathrm{KS}})$ of $(M,\omega)$. Moreover
\begin{equation}\label{eq:tensorfactorks}
  Y_\eta\;\cong\;Y_{\mathrm{KS}}^{\otimes2},
  \qquad
  Y\;\cong\;\pi^\ast Y_{\mathrm{KS}}\;\otimes\;Y_{\widetilde{\ori}}
\end{equation}
as circle bundles with connection.
\end{proposition}

\begin{proof}
Taking the associated bundles with respect to a product of characters gives the Picard product
of the associated bundles with connections by \eqref{eq:picardchar}, so in each case it is enough
to factor the character.

For \eqref{eq:tensorfactor}, $(\hat\chi^c)^2=(\ori\circ\ssigma^c)\cdot\eta$ by
\eqref{eq:chisplit}, since $\eta([\tilde g,\xi])=\xi^2$. The factor $\pi^\ast Y_\eta$ is pulled
back from $M$ because $\eta$ is defined on all of $\Mp^c(\bV)$ and not merely on
$\Mp^c_{\bF}(\bV)$.

For \eqref{eq:mpreduction}, apply Corollary~\ref{cor:splittings} to $\ssigma^c$. Its criterion
is $w(P,\ssigma)=w_2(TM)=0$, and the splitting consists of a metaplectic structure $\tilde P$
and a principal circle bundle with connection $(Y_{\mathrm{KS}},\gamma_{\mathrm{KS}})$. The
central connection $\gamma$ is of weight one on the central circle by
Definition~\ref{def:centralconnection}, so $\curv\gamma_{\mathrm{KS}}=\curv\gamma=\omega$ and
$(Y_{\mathrm{KS}},\gamma_{\mathrm{KS}})$ is a Kostant--Souriau prequantization of $(M,\omega)$.

Given the splitting \eqref{eq:mpreduction}, an element of $\Mp^c_{\bF}(\bV)$ is $[\tilde g,\xi]$ with
$\tilde g\in\Mp_{\bF}(\bV)$. Taking the associated bundle with respect to $\eta$ then gives the square of the bundle
associated to $\xi$, which is $Y_{\mathrm{KS}}$, and by \eqref{eq:chicdef}
$\hat\chi^c([\tilde g,\xi])=\hat\chi(\tilde g)\,\xi$ is the product of the $\mu_4$-valued
$\hat\chi$ with the weight-one central character, whose associated bundles are $Y_{\widetilde{\ori}}$
and $\pi^\ast Y_{\mathrm{KS}}$. This gives \eqref{eq:tensorfactorks}.
\end{proof}

\subsection{Maslov indices and the EBK condition}
\label{subsec:ebkcondition}

Let $(M,\omega)$ be a symplectic manifold and let $P$ be the symplectic frame bundle of $TM$.
Recall that
\[
  \Lag(TM)=P\times_{\Sp(\bV)}\Lag(\bV)\cong P\times_{\Sp(\bV)}\big(\Sp(\bV)/\Sp_{\bF}(\bV)\big)
  \cong P/\Sp_{\bF}(\bV)
\]
is the Lagrangian Grassmannian bundle, with projection $\pi:\Lag(TM)\to M$, and let
\[
  \Lag(TM;\mu_2):=P\times_{\Sp(\bV)}\Lag(\bV;\mu_2)
  \cong P\times_{\Sp(\bV)}\big(\Sp(\bV)/\Sp^+_{\bF}(\bV)\big)\cong P/\Sp^+_{\bF}(\bV)
\]
be the oriented Lagrangian Grassmannian bundle. Then there is a tower
\[
  \Lag(TM;\mu_2)\longrightarrow\Lag(TM)\xrightarrow{\ \pi\ }M ,
\]
where $\Lag(TM;\mu_2)\to\Lag(TM)$ is a double covering that restricts to the connected double
covering $\Lag(T_xM;\mu_2)\to\Lag(T_xM)$ on each fibre. Let $\tau_{TM}$ be the tautological bundle
over $\Lag(TM)$.

The obstruction $w(P,\ssigma)\in H^2(M;\Z_2)$ to a lifting of $P$ with respect to
$\ssigma:\Mp(\bV)\to\Sp(\bV)$ identifies with $w_2(TM)$ (\S\ref{subsec:bockstein}). When
$w_2(TM)=0$, this obstruction vanishes, so a metaplectic structure $\tilde P\to P\to M$ (a
lifting of $P$ with respect to $\ssigma$) exists. Let
\[
  \Lag(TM;\mu_4):=\tilde P\times_{\Mp(\bV)}\Lag(\bV;\mu_4)
  \cong\tilde P\times_{\Mp(\bV)}\big(\Mp(\bV)/\Sp^+_{\bF}(\bV)\big)\cong\tilde P/\Sp^+_{\bF}(\bV),
\]
where $\Sp^+_{\bF}(\bV)\le\Mp(\bV)$ is the lift of Lemma~\ref{lem:lifts}(\ref{it:liftplus}).
Then there is a tower
\[
  \Lag(TM;\mu_4)\longrightarrow\Lag(TM;\mu_2)\longrightarrow\Lag(TM)\xrightarrow{\ \pi\ }M ,
\]
where $\Lag(TM;\mu_4)\to\Lag(TM)$ is the fourfold Maslov covering associated to $\tilde P$,
that restricts to the connected fourfold covering $\Lag(T_xM;\mu_4)\to\Lag(T_xM)$ on each fibre,
and factors through the connected double covering
$\Lag(T_xM;\mu_4)/\mu_2\cong\Lag(T_xM;\mu_2)$.

Let $\widetilde{\Sp}(\bV)\to\Sp(\bV)$ be the simply connected Lie group covering, with kernel
$\Z$, so that $\Mp(\bV)\cong\widetilde{\Sp}(\bV)/(2\Z)$. Then the inclusion
$\Sp^+_{\bF}(\bV)\hookrightarrow\Sp(\bV)$ also lifts to $\widetilde{\Sp}(\bV)$, as in
Lemma~\ref{lem:lifts}(\ref{it:liftplus}). A compatible almost complex structure $J$ on $TM$
reduces the structure group of $P$ to the unitary group, so the obstruction
$w\big(P,\widetilde{\Sp}(\bV)\to\Sp(\bV)\big)\in H^2(M;\Z)$ identifies with $c_1(TM,J)$. When
$c_1(TM)=0$, this obstruction vanishes, so there exists a lifting
$P_\infty\to P$ with respect to $\widetilde{\Sp}(\bV)\to\Sp(\bV)$, and
$P_\infty/(2\Z)\to P_\infty/\Z=P\to M$ is a metaplectic structure. Let
\[
  \Lag(TM;\Z):=P_\infty\times_{\widetilde{\Sp}(\bV)}\Lag(\bV;\Z)
  \cong P_\infty\times_{\widetilde{\Sp}(\bV)}\big(\widetilde{\Sp}(\bV)/\Sp^+_{\bF}(\bV)\big)
  \cong P_\infty/\Sp^+_{\bF}(\bV),
\]
where $\Lag(\bV;\Z)\cong\widetilde{\Sp}(\bV)/\Sp^+_{\bF}(\bV)$ is the universal covering of
$\Lag(\bV)$. Then there is a tower
\[
  \Lag(TM;\Z)\longrightarrow\Lag(TM;\mu_4)\longrightarrow\Lag(TM;\mu_2)
  \longrightarrow\Lag(TM)\xrightarrow{\ \pi\ }M
\]
such that $\Lag(TM;\mu_4)\cong\Lag(TM;\Z)/(4\Z)$ and
$\Lag(TM;\mu_2)\cong\Lag(TM;\Z)/(2\Z)$, since a generator of
$\ker\big(\widetilde{\Sp}(\bV)\to\Sp(\bV)\big)$ acts on $\Lag(\bV;\Z)$ as $2\in\Z$
\cite[Lem.~2.1]{Seidel}.

More generally, $\Lag(TM)$ admits an $N$-fold Maslov covering if and only if $2c_1(TM)$
vanishes in $H^2(M;\Z_N)$, and an $\infty$-fold Maslov covering if and only if $2c_1(TM)=0$ in
$H^2(M;\Z)$ \cite[Lem.~2.2]{Seidel}. The $\mu_2$ orbit space quotient of a fourfold Maslov
covering, and the quotient by $2\Z$ of an $\infty$-fold Maslov covering, is a twofold Maslov
covering that is not necessarily the orientation covering $\Lag(TM;\mu_2)$ \cite[\S2b, Ex.~2.7]{Seidel}. For
example, the Enriques surface $X$ of Example~\ref{ex:enriques} has $2$-torsion first Chern class
$c_1(X)$. So $N$-fold Maslov coverings exist for every $N$, but no fourfold Maslov covering has
$\mu_2$ orbit space quotient isomorphic to $\Lag(TX;\mu_2)$, because the existence of such
structure would imply $X$ has a metaplectic structure.

Let $\iota_L:L\hookrightarrow M$ be a Lagrangian submanifold. Its tautological Gauss section
\[
  s_L:L\longrightarrow\Lag(TM)|_L,
  \qquad x\longmapsto T_xL,
\]
satisfies $s_L^\ast\tau_{TM}=TL$ and
$\pi\circ s_L=\iota_L$.
We view $\Lag(TM;\mu_2)$ and $\Lag(TM;\mu_4)$ as principal $\mu_2$- and $\mu_4$-bundles with
$\mu_k\le\U(1)$, and $\Lag(TM;\Z)$ as a principal $\Z$-bundle together with the character
$\rho_4:\Z\to\U(1)$, $k\mapsto i^k$. The associated bundles
\[
  \Lag(TM;\mu_k)\times_{\mu_k}\U(1),
  \qquad
  \Lag(TM;\Z)\times_{\rho_4}\U(1)
\]
have the same locally constant transition functions, so their parallel sections are well
defined, and they are circle bundles with flat connection. For $\lambda\in H_1(L;\Z)$, the
orientation monodromy is
\begin{equation}\label{eq:w1def}
  \big\langle s_L^\ast\Lag(TM;\mu_2),\lambda\big\rangle
  =\Hol_{s_L^\ast(\Lag(TM;\mu_2)\times_{\mu_2}\U(1))}(\lambda)
  =(-1)^{\langle w_1(TL),\lambda\rangle}
  \in\mu_2 .
\end{equation}
When $w_2(TM)=0$, the mod $4$ Maslov index $\Mas_4(\lambda)\in\Z_4$, for the choice of
metaplectic structure $\tilde P\to P\to M$, is defined by
\begin{equation}\label{eq:masdef}
  i^{\Mas_4(\lambda)}
  :=\big\langle s_L^\ast\Lag(TM;\mu_4),\lambda\big\rangle
  =\Hol_{s_L^\ast(\Lag(TM;\mu_4)\times_{\mu_4}\U(1))}(\lambda)
  \in\mu_4 .
\end{equation}
When $c_1(TM)=0$, the Maslov index, for the choice of lifting $P_\infty\to P$, is
\begin{equation}\label{eq:masint}
  \Mas(\lambda)
  :=\big\langle s_L^\ast\Lag(TM;\Z),\lambda\big\rangle
  \in\Z ,
  \qquad
  i^{\Mas(\lambda)}
  =\Hol_{s_L^\ast(\Lag(TM;\Z)\times_{\rho_4}\U(1))}(\lambda) .
\end{equation}
Compatibility of the tower gives $\Mas_4(\lambda)\equiv\langle w_1(TL),\lambda\rangle\pmod2$
and $\Mas(\lambda)\equiv\Mas_4(\lambda)\pmod4$.

\begin{definition}\label{def:quantizable}
Let $(P^c,\bu,\gamma)$ be a metaplectic-c prequantization of $(M,\omega)$, let $\hat\chi^c$ be a
normalized Hess character, and let $(Y,\tilde\gamma)\to\Lag(TM)$ be the descent by $\hat\chi^c$.
A Lagrangian submanifold $\iota_L:L\hookrightarrow M$ is \textbf{quantizable with respect to
$(P^c,\gamma,\hat\chi^c)$} if $s_L^\ast(Y,\tilde\gamma)$ admits a global parallel section,
equivalently if its holonomy character $\Hol_{s_L^\ast Y}:H_1(L;\Z)\to\U(1)$ is trivial.
\end{definition}

\begin{theorem}\label{thm:ebk}
Let $(P^c,\bu,\gamma)$ be a metaplectic-c prequantization of $(M,\omega)$ and let
$(Y,\tilde\gamma)\to\Lag(TM)$ be its descent by a normalized Hess character $\hat\chi^c$. Let
$\iota_L:L\hookrightarrow M$ be a Lagrangian submanifold and let $\lambda\in H_1(L;\Z)$.
\begin{enumerate}[label=\textup{(\arabic*)}]
\item The holonomy of $s_L^\ast(Y,\tilde\gamma)$ on $\lambda$ satisfies
\begin{equation}\label{eq:ebksquare}
  \Hol_{s_L^\ast Y}(\lambda)^2
   \;=\;\Hol_{\iota_L^\ast Y_\eta}(\lambda)
        \cdot\Hol_{s_L^\ast\Lag(TM;\mu_2)}(\lambda)
   \;=\;\exp\Big(2\pi i\oint_\lambda\tfrac{i\,s_\lambda^\ast\gamma_\eta}{2\pi}\Big)
        \cdot(-1)^{\langle w_1(TL),\lambda\rangle} .
\end{equation}
If $(Y',\tilde\gamma')$ is the descent by the other normalized Hess character
$\hat\chi^c\circ\alpha$, then
\[
  \Hol_{s_L^\ast Y'}(\lambda)
   \;=\;(-1)^{\langle w_1(TL),\lambda\rangle}\cdot\Hol_{s_L^\ast Y}(\lambda) .
\]

\item If $w_2(TM)=0$, let
$\big(\tilde P,\,(Y_{\mathrm{KS}},\gamma_{\mathrm{KS}}),\,\hv u\big)$ be a
$\,\cdot\,/\Delta$-lifting of $(P^c,\gamma)$, with
$\hv u:(P^c,\gamma)\xrightarrow{\ \cong\ }\tilde P\times_{\mu_2}(Y_{\mathrm{KS}},\gamma_{\mathrm{KS}})$, and let $\Lag(TM;\mu_4)$ be the fourfold Maslov covering
associated to $\tilde P$. Then for every Lagrangian submanifold, the holonomy of
$s_L^\ast(Y,\tilde\gamma)$ splits into well-defined factors,
\begin{equation}\label{eq:ebkclass}
  \Hol_{s_L^\ast Y}(\lambda)
   \;=\;\Hol_{\iota_L^\ast Y_{\mathrm{KS}}}(\lambda)
        \cdot\Hol_{s_L^\ast\Lag(TM;\mu_4)}(\lambda)
   \;=\;\exp\Big(2\pi i\oint_\lambda\tfrac{i\,s_\lambda^\ast\gamma_{\mathrm{KS}}}{2\pi}\Big)
        \cdot i^{\Mas_4(\lambda)} .
\end{equation}

\item If $c_1(TM)=0$, let $P_\infty\to P$ be a lifting with structure group
$\widetilde{\Sp}(\bV)$ such that
$\big(P_\infty/(2\Z),\,(Y_{\mathrm{KS}},\gamma_{\mathrm{KS}}),\,\hv u\big)$ is a
$\,\cdot\,/\Delta$-lifting of $(P^c,\gamma)$, with
$\hv u:(P^c,\gamma)\xrightarrow{\ \cong\ }\big(P_\infty/(2\Z)\big)\times_{\mu_2}
(Y_{\mathrm{KS}},\gamma_{\mathrm{KS}})$. Then $\Mas_4(\lambda)=\Mas(\lambda)\bmod 4$, and
\begin{equation}\label{eq:ebkint}
  \Hol_{s_L^\ast Y}(\lambda)
   \;=\;\exp\Big(2\pi i\Big[\oint_\lambda\tfrac{i\,s_\lambda^\ast\gamma_{\mathrm{KS}}}{2\pi}
     +\tfrac{\Mas(\lambda)}{4}\Big]\Big) .
\end{equation}

\item The splitting obstructions restrict to $L$ as
\begin{equation}\label{eq:obstructionL}
  \iota_L^\ast w_2(TM)=w_1(TL)^2\ \text{ in } H^2(L;\Z_2),
  \qquad
  \iota_L^\ast c_1(TM)=\beta_2\big(w_1(TL)\big)\ \text{ in } H^2(L;\Z).
\end{equation}
Consequently, if $w_1(TL)^2=0$, then $\iota_L^\ast(P^c,\gamma)$ admits a
$\,\cdot\,/\Delta$-lifting, and the holonomy of $s_L^\ast(Y,\tilde\gamma)$ splits as in
\eqref{eq:ebkclass} over $L$. If $\beta_2\big(w_1(TL)\big)=0$, for instance if $L$ is
orientable, there exist a $\,\cdot\,/\Delta$-lifting
$\big(\tilde P,(Y_{\mathrm{KS}},\gamma_{\mathrm{KS}}),\hv u\big)$ of $\iota_L^\ast(P^c,\gamma)$ and
a lifting of its mod $4$ Maslov index to an integer-valued Maslov index $\Mas:H_1(L;\Z)\to\Z$
such that
\[ \Hol_{s_L^\ast Y}(\lambda)
   \;=\;\exp\Big(2\pi i\Big[\oint_\lambda\tfrac{i\,s_\lambda^\ast\gamma_{\mathrm{KS}}}{2\pi}
     +\tfrac{\Mas(\lambda)}{4}\Big]\Big) . \]
\end{enumerate}
\end{theorem}

\begin{proof}
All connections below are flat after pullback to $L$, since
$\iota_L^\ast\omega=0$. Their holonomies therefore define characters of
$H_1(L;\Z)$.

\textup{(1)}
Let $Y_{\ori}=P\times_{\ori}\U(1)$ be the flat circle bundle of
\S\ref{subsec:ebklag}, associated to the orientation covering
$\Lag(TM;\mu_2)$ of the tautological bundle $\tau_{TM}$ over $\Lag(TM)$.
By \eqref{eq:chisplit},
\[
  (\hat\chi^c)^2
  =\eta\cdot(\ori\circ\ssigma^c).
\]
Taking the associated bundles with respect to this identity of characters gives
\[
  Y^{\otimes2}\cong\pi^\ast Y_\eta\otimes Y_{\ori}
\]
as circle bundles with connection, as in
Proposition~\ref{prop:tensorfactor}. The connection on $Y_\eta$ pulls
back to $2\gamma$ on $P^c$, and the orientation character is locally
constant and equips $Y_{\ori}$ with a flat connection.

Pulling back by $s_L$, we obtain
\[
  \Hol_{s_L^\ast Y}(\lambda)^2
  =
  \Hol_{\iota_L^\ast Y_\eta}(\lambda)\,
  (-1)^{\langle w_1(TL),\lambda\rangle}
\]
because $\pi\circ s_L=\iota_L$ and $s_L^\ast\tau_{TM}=TL$.
Then for a trivializing section $s_\lambda$ of $\lambda^\ast\iota_L^\ast Y_\eta$,
\[
  \Hol_{\iota_L^\ast Y_\eta}(\lambda)
  =\exp\Big(2\pi i\oint_\lambda\tfrac{i\,s_\lambda^\ast\gamma_\eta}{2\pi}\Big),
\]
where the quantity does not depend on the choice of trivializing section.

Moreover, \eqref{eq:gsflip} gives
\[
  \hat\chi^c\circ\alpha
  =\hat\chi^c\cdot(\ori\circ\ssigma^c).
\]
Thus $(Y',\tilde\gamma')\cong(Y,\tilde\gamma)\otimes Y_{\ori}$, with the tensor product
connection. So the holonomies on $\lambda$ differ by $(-1)^{\langle w_1(TL),\lambda\rangle}$.

\textup{(2)}
Through $\hv u$, an element of $\Mp^c_{\bF}(\bV)$ is written $[\tilde g,\xi]$ with
$\tilde g\in\Mp_{\bF}(\bV)$, and by Lemma~\ref{lem:chicrestricts},
\[
  \hat\chi^c([\tilde g,\xi])=\xi\,\hat\chi(\tilde g),
  \qquad
  \hat\chi=\widetilde{\ori}.
\]
Taking the associated bundles with respect to this product of characters gives
\[
  (Y,\tilde\gamma)
  \cong
  \pi^\ast(Y_{\mathrm{KS}},\gamma_{\mathrm{KS}})
  \otimes\big(\Lag(TM;\mu_4)\times_{\mu_4}\U(1)\big)
\]
as circle bundles with connection, since $\gamma$ is induced from $\gamma_{\mathrm{KS}}$ and
$\widetilde{\ori}$ is locally constant. Pulling back by $s_L$, we obtain
\[
  \Hol_{s_L^\ast Y}(\lambda)
  =\Hol_{\iota_L^\ast Y_{\mathrm{KS}}}(\lambda)\cdot i^{\Mas_4(\lambda)}
\]
because $\pi\circ s_L=\iota_L$, and by \eqref{eq:masdef}. Then for a trivializing section
$s_\lambda$ of $\lambda^\ast\iota_L^\ast Y_{\mathrm{KS}}$,
\[
  \Hol_{\iota_L^\ast Y_{\mathrm{KS}}}(\lambda)
  =\exp\Big(2\pi i\oint_\lambda\tfrac{i\,s_\lambda^\ast\gamma_{\mathrm{KS}}}{2\pi}\Big),
\]
where the quantity does not depend on the choice of trivializing section.

\textup{(3)}
By \S\ref{subsec:ebkcondition}, $P_\infty/(2\Z)$ is a metaplectic structure, and
$\Lag(TM;\mu_4)\cong\Lag(TM;\Z)/(4\Z)$ is the fourfold Maslov covering associated to it, so
$\Mas_4(\lambda)=\Mas(\lambda)\bmod4$ and, by \eqref{eq:masdef} and \eqref{eq:masint},
$i^{\Mas_4(\lambda)}=i^{\Mas(\lambda)}$. Part \textup{(2)}, applied to the given
$\,\cdot\,/\Delta$-lifting, yields
\[
  \Hol_{s_L^\ast Y}(\lambda)
  =\Hol_{\iota_L^\ast Y_{\mathrm{KS}}}(\lambda)\cdot i^{\Mas(\lambda)}
  =\exp\Big(2\pi i\Big[\oint_\lambda\frac{i\,s_\lambda^\ast\gamma_{\mathrm{KS}}}{2\pi}
    +\frac{\Mas(\lambda)}4\Big]\Big).
\]

\textup{(4)} Since $L$ is Lagrangian, $TL\cap J\,TL=0$ for a compatible complex structure $J$,
so $TM|_L\cong TL\oplus J\,TL\cong TL\otimes\C$. The Whitney sum formula gives
$\iota_L^\ast w_2(TM)=w_2(TL\oplus TL)=w_1(TL)^2$, and
$\iota_L^\ast c_1(TM)=c_1(TL\otimes\C)=\beta_2\big(w_1(TL)\big)$, by the splitting principle from
the case of a real line bundle. This gives \eqref{eq:obstructionL}.

The obstruction to a $\,\cdot\,/\Delta$-lifting of $\iota_L^\ast(P^c,\gamma)$ is
$w(P|_L,\ssigma)=\iota_L^\ast w_2(TM)$ by naturality (Remark~\ref{rem:naturality}) and
Proposition~\ref{prop:splittings}\textup{(3)}. If $w_1(TL)^2=0$, a $\,\cdot\,/\Delta$-lifting
therefore exists, and the proof of \textup{(2)}, applied over $L$ with the fourfold Maslov
covering of $\Lag(TM)|_L$ associated to it, gives the splitting \eqref{eq:ebkclass}.

If $\beta_2\big(w_1(TL)\big)=0$, then also $w_1(TL)^2=0$, since $TM|_L$ is complex, so that
$w_1(TL)^2=\iota_L^\ast w_2(TM)$ is the mod $2$ reduction of
$\iota_L^\ast c_1(TM)=\beta_2\big(w_1(TL)\big)$. By \S\ref{subsec:ebkcondition}, the vanishing
of $\iota_L^\ast c_1(TM)$ gives a lifting $P_\infty\to P|_L$ with structure group
$\widetilde{\Sp}(\bV)$, so that $P_\infty/(2\Z)$ is a metaplectic structure on $P|_L$. The
$\,\cdot\,/\Delta$-liftings of $\iota_L^\ast(P^c,\gamma)$ form a torsor under $H^1(L;\Z_2)$
(Proposition~\ref{prop:splittings}\textup{(3)}), as do the metaplectic structures on $P|_L$
(\S\ref{subsec:liftings}), and taking the first component is equivariant. An equivariant map
between torsors under the same group is bijective, so there is a $\,\cdot\,/\Delta$-lifting
with $\tilde P=P_\infty/(2\Z)$. The proof of \textup{(3)}, applied over $L$, then gives the
integer splitting.
\end{proof}

\begin{example}\label{ex:rp2}
On orientable Lagrangian submanifolds the condition of Definition~\ref{def:quantizable} does not
depend on the choice of Hess character, by Theorem~\ref{thm:ebk}\textup{(1)}. On nonorientable
ones it can. Let $M=\mathbb{CP}^2$ with $\omega=k\,\omega_{\mathrm{FS}}$, where
$[\omega_{\mathrm{FS}}/2\pi]$ is the hyperplane class $H$. Since $c_1(\mathbb{CP}^2)=3H$, the class
$[\omega/2\pi]+\tfrac12c_1(\mathbb{CP}^2)$ is integral for $k\in\tfrac12+\Z_{>0}$, so $M$ is
metaplectic-c prequantizable (Corollary~\ref{cor:Pc}). The real locus $L=\mathbb{RP}^2$ is
Lagrangian, with $H_1(L;\Z)\cong\Z_2$ generated by $\lambda$. The holonomy character factors
through $H_1(L;\Z)$, so $\Hol_{s_L^\ast Y}(\lambda)^2=1$, and since
$\langle w_1(T\mathbb{RP}^2),\lambda\rangle\neq0$, Theorem~\ref{thm:ebk}\textup{(1)} gives
$\Hol_{s_L^\ast Y'}(\lambda)=-\Hol_{s_L^\ast Y}(\lambda)$. Hence $\mathbb{RP}^2$ is quantizable
with respect to exactly one of the two normalized Hess characters.
\end{example}

\section*{Acknowledgments}
\addcontentsline{toc}{section}{Acknowledgments}

This research was partly funded by the United States--Israel Binational Science Foundation (BSF) and by the Natural Sciences and Engineering Research Council of Canada (NSERC), and by the National Research Foundation of Korea (NRF) grant funded by the Korea government (Ministry of Science and ICT, MSIT) (No.~RS-2024-00359647), and the Basic Science Research Program through the National Research Foundation of Korea (NRF) funded by the Ministry of Education (RS-2021-NR060138).

\end{document}